\documentclass[aos]{imsart}

\usepackage[utf8]{inputenc} 
\usepackage[T1]{fontenc}    
\usepackage{lmodern}
\usepackage[table]{xcolor}  
\usepackage{url}            
\usepackage{booktabs}       
\usepackage{amsfonts}       
\usepackage{nicefrac}       
\usepackage{microtype}      

\usepackage[linesnumbered,lined,boxed,commentsnumbered]{algorithm2e}
\usepackage{amsmath}
\usepackage{amssymb}
\usepackage{amsthm}
\usepackage{arydshln} 
\usepackage{enumitem}
\usepackage{tikz}
\usetikzlibrary{shapes,backgrounds,patterns,arrows}
\usepackage{pgfplots}
\usepackage{subcaption}
\usepackage{todonotes}
\usepackage{mathtools}
\usepackage{nth}

\usepackage[acronym]{glossaries}
\makeglossaries

\newacronym{SVD}{SVD}{Singular Value Decomposition}
\newacronym{SBM}{SBM}{Stochastic Block Model}
\newacronym{LSBM}{LSBM}{Labeled Stochastic Block Model}
\newacronym{BMC}{BMC}{Block Markov Chain}

\newtheorem{corollary}{Corollary}
\newtheorem*{definition}{Definition}
\newtheorem{lemma}{Lemma}
\newtheorem{proposition}{Proposition}

\newtheorem{theorem}{Theorem}

\setlist[enumerate]{label={\upshape(\roman*)}}

\makeatletter
\let\c@author\relax 
\makeatother

\usepackage[
backend=bibtex,
citestyle=numeric-comp,
doi=true,
isbn=false,
url=true,
eprint=false,
uniquename=init,
bibstyle=numeric,
maxnames=99,
natbib=true,
sorting=none,
giveninits]{biblatex}

\setattribute{journal}{name}{} 

\let\chapter\section
\let\section\subsection
\let\subsection\subsubsection
\let\subsubsection\paragraph

\usepackage{dsfont} 

\newcommand{\stateSpace}{ \Omega }

\newcommand{\cardinality}[1]{ | #1 | }
\renewcommand{\d}[1]{\ensuremath{\operatorname{d}\!{#1}}}
\newcommand{\e}[1]{ {\mathrm{e}}^{ #1 } }

\newcommand{\expectation}[1]{ \mathbb{E} [ #1 ] }
\newcommand{\expectationWrt}[2]{ \mathbb{E}_{#2} [ #1 ] }
\newcommand{\expectationBig}[1]{ \mathbb{E} \Bigl[ #1 \Bigr] }
\newcommand{\expectationBigWrt}[2]{ \mathbb{E}_{#2} \Bigl[ #1 \Bigr] }

\newcommand{\vect}[1]{ #1 }

\newcommand{\vectOnes}[1]{ \vect{1}_{#1} }
\newcommand{\vectInLine}[1]{ ( #1 )^{\mathrm{T}} }
\newcommand{\vectElementary}[2]{ \vect{e}_{#1,#2} }

\newcommand{\indicator}[1]{ \mathds{1} [ #1 ] }

\newcommand{\process}[2]{ \{ #1 \}_{ #2 } }
\newcommand{\smallO}[1]{ o(#1) }
\newcommand{\smallObig}[1]{ o\Bigl(#1\Bigr) }
\newcommand{\bigO}[1]{ O(#1) }
\newcommand{\bigObig}[1]{ O\Bigl(#1\Bigr) }
\newcommand{\smallOP}[1]{ o_{\mathbb{P}}(#1) }

\newcommand{\bigOP}[1]{ O_{\mathbb{P}}(#1) }
\newcommand{\bigOPbig}[1]{ O_{\mathbb{P}}\Bigl(#1\Bigr) }
\newcommand{\OmegaP}[1]{ \Omega_{\mathbb{P}}(#1) }
\newcommand{\OmegaPbig}[1]{ \Omega_{\mathbb{P}}\Bigl(#1\Bigr) }
\newcommand{\pnorm}[2]{ \| #1 \|{}_{#2} }
\newcommand{\pnormBig}[2]{ \bigpnorm{#1}{#2} }
\newcommand{\bigpnorm}[2]{ \Bigl\| #1 \Bigr\|_{#2} }
\newcommand{\probability}[1]{ \mathbb{P} [ #1 ] }
\newcommand{\probabilityWrt}[2]{ \mathbb{P}_{#2} [ #1 ] }
\newcommand{\probabilityBig}[1]{ \mathbb{P} \Bigl[ #1 \Bigr] }
\newcommand{\probabilityBigWrt}[2]{ \mathbb{P}_{#2} \Bigl[ #1 \Bigr] }

\newcommand{\transpose}[1]{ #1{}^{\mathrm{T}} }

\newcommand{\variance}[1]{ \mathrm{Var} [ #1 ] }

\newcommand{\varianceWrt}[2]{ \mathrm{Var}_{#2} [ #1 ] }
\newcommand{\varianceBigWrt}[2]{ \mathrm{Var}_{#2} \Bigl[ #1 \Bigr] }

\newcommand{\covarianceWrt}[3]{ \mathrm{Cov}_{#3} [ #1, #2 ] }

\newcommand{\naturalNumbersPlus}{ \mathbb{N}_{+} }
\newcommand{\naturalNumbersZero}{ \mathbb{N}_{0} }
\newcommand{\realNumbers}{ \mathbb{R} }
\newcommand{\positiveRealNumbers}{ [0,\infty) }
\newcommand{\strictlyPositiveRealNumbers}{ (0,\infty) }

\newcommand{\criticalpoint}[1]{  #1^{\textnormal{opt}} }
\newcommand{\iterand}[2]{ #1^{[#2]} }

\newcommand{\refFigure}[1]{{\textrm{Figure~\ref{#1}}}}

\newcommand{\refEquation}[1]{{\textrm{\eqref{#1}}}}

\newcommand{\refTheorem}[1]{{\textrm{Theorem~\ref{#1}}}}

\newcommand{\refCorollary}[1]{{\textrm{Corollary~\ref{#1}}}}
\newcommand{\refProposition}[1]{{\textrm{Proposition~\ref{#1}}}}
\newcommand{\refLemma}[1]{{\textrm{Lemma~\ref{#1}}}}

\newcommand{\refChapter}[1]{{\textrm{Chapter~\ref{#1}}}}
\newcommand{\refSection}[1]{{\textrm{Section~\ref{#1}}}}
\newcommand{\refSupplementaryMaterial}[1]{{\S{}\textrm{\ref{#1}}}}
\newcommand{\refAppendixSection}[1]{\S{}\textrm{\ref{#1}}}

\def\eqcom#1{\overset{\textnormal{(#1)}}}

\newcommand{\itr}[2]{ \iterand{#1}{#2} }

\def\({{\Bigl(}}
\def\){{\Bigr)}}

\newcommand{\ba}{\begin{array}}
\newcommand{\ea}{\end{array}}

\newcommand{\xdeleted}[1]{\deleted{}} 
\newcommand{\vectInLineTransposed}[1]{ ( #1 ) }

\def\jaron#1{\textcolor{blue!50!black}{#1}}

\def\revisedPartBegin{\color{blue!50!black}}
\def\revisedPartEnd{\color{black}}

\def\EHset{\mathcal{E}_{\mathcal{H}}}

\allowdisplaybreaks

\def\jaron#1{#1}

\def\revisedPartBegin{}
\def\revisedPartEnd{}

\begin{document}

\begin{frontmatter}
\title{Clustering in Block Markov Chains}
\runtitle{Clustering in Block Markov Chains}

\begin{aug}
\author{\fnms{Jaron} \snm{Sanders}\thanksref{m1,m2}\ead[label=e1]{j.sanders@tudelft.nl}\ead[label=e2b]{jarons@kth.se}}
\author{\fnms{Alexandre} \snm{Prouti\`{e}re}\thanksref{m1}\ead[label=e2a]{alepro@kth.se}}
\and
\author{\fnms{Se-Young} \snm{Yun}\thanksref{m3}\ead[label=e3]{yunseyoung@kaist.ac.kr}}

\runauthor{J.\ Sanders, A.\ Prouti\`{e}re and S.Y. Yun}

\affiliation{KTH Royal Institute of Technology\thanksmark{m1}, Sweden\\ Delft University of Technology\thanksmark{m2}, The Netherlands\\ Korea Advanced Institute of Science and Technology\thanksmark{m3}, South Korea}

\address{KTH Royal Institute of Technology\\
School of Electrical Engineering\\
Dept.\ of Automatic Control\\
Osquldasv\"{a}g 10, Stockholm 10044, Sweden\\
\printead{e2a}\\
\\
Delft University of Technology\\
Faculty of Electrical Engineering, Mathematics \& Computer Science\\
Dept.\ of Quantum \& Computer Engineering\\
Mekelweg 4, 2628CD Delft, The Netherlands\\
\printead{e1}\\
\\
Korea Advanced Institute of Science and Technology\\
Dept.\ of Industrial \& Systems Engineering\\
291 Daehak-ro, Eoeun-dong, Yuseong-gu, South Korea\\
\printead{e3}}
\end{aug}

\glsreset{SBM}
\glsreset{BMC}
\begin{abstract}
This paper considers cluster detection in \glspl{BMC}. These Markov chains are characterized by a block structure in their transition matrix. More precisely, the $n$ possible states are divided into a finite number of $K$ groups or clusters, such that states in the same cluster exhibit the same transition rates to other states. One observes a trajectory of the Markov chain, and the objective is to recover, from this observation only, the (initially unknown) clusters. In this paper we devise a clustering procedure that accurately, efficiently, and provably detects the clusters. We first derive a fundamental information-theoretical lower bound on the detection error rate satisfied under any clustering algorithm. This bound identifies the parameters of the \gls{BMC}, and trajectory lengths, for which it is possible to accurately detect the clusters. We next develop two clustering algorithms that can together accurately recover the cluster structure from the shortest possible trajectories, whenever the parameters allow detection. These algorithms thus reach the fundamental detectability limit, and are optimal in that sense. 
\end{abstract}

\begin{keyword}[class=MSC]
\kwd[Primary ]{62H30} 
\kwd{60J10} 
\kwd{60J20} 
\end{keyword}

\begin{keyword}
\kwd{clustering}
\kwd{Markov chains}
\kwd{mixing times}
\kwd{community detection}
\kwd{change of measure}
\kwd{asymptotic analysis}
\kwd{information theory}
\end{keyword}

\end{frontmatter}

\chapter{Introduction}

The ability to accurately discover all hidden relations between items that share similarities is of paramount importance to a wide range of disciplines. Clustering algorithms in particular are employed throughout social sciences, biology, computer science, economics, and physics. The reason these techniques have become prevalent is that once clusters of similar items have been identified, any subsequent analysis or optimization procedure benefits from a powerful reduction in dimensionality.

\glsreset{SBM}
The canonical \gls{SBM}, originally introduced in \cite{holland_stochastic_1983}, has become the benchmark to investigate the performance of cluster detection algorithms. This model generates random graphs that contain groups of similar vertices. Vertices within the same group are similar in that they share the same average edge densities to the other vertices. More precisely, if the set of $n$ vertices $\mathcal{V}$ is for example partitioned into two groups $\mathcal{V}_1$ and $\mathcal{V}_2$, an edge is drawn between two vertices $x,y \in \mathcal{V}$ with probability $p \in (0,1)$ if they belong to the same group, and with probability $q \in (0,1)$, $p \neq q$, if they belong to different groups. Edges are drawn independently of all other edges. Within the context of the \gls{SBM} and its generalizations, the problem of cluster detection is to infer the clusters from observations of a realization of the random graph with the aforementioned structure.

\glsreset{BMC}
This paper deviates by considering the problem of cluster detection when the observation is instead the sample path of a Markov chain over the set of vertices. Specifically, we introduce the \gls{BMC}, which is a Markov chain characterized by a block structure in its transition matrix. States that are in the same cluster are similar in the sense that they have the same transition rates. The goal is to detect the clusters from an observed sample path $X_0, X_1, \ldots, X_T$ of the Markov chain (\refFigure{fig:Infer_the_hidden_cluster_structure_from_a_MC}). This new clustering problem is mathematically more challenging because consecutive samples of the random walk are \emph{not} independent: besides noise, there is bias in a sample path. Intuitively though there is hope for accurate cluster detection if the Markov chain can get close to stationarity within $T$ steps. Indeed, as we will show, the mixing time \cite{levin_markov_2009} of the \gls{BMC} plays a crucial role in the detectability of the clusters.

\begin{figure}[!hbtp]
\centering
\begin{tikzpicture}[scale=3.5]
\tikzstyle{vertex}=[circle, draw, thick, fill=black!0, inner sep=0pt, minimum width=4pt]
\tikzstyle{selected vertex} = [circle, draw, thick, fill=black!100, inner sep=0pt, minimum width=4pt]
\tikzstyle{visited vertex} = [circle, draw, thick, fill=red!50, inner sep=0pt, minimum width=4pt]
\tikzstyle{revisited vertex} = [circle, draw, thick, fill=red!100, inner sep=0pt, minimum width=4pt]
\tikzstyle{edge} = [draw,thick,->]
\tikzstyle{weight} = [font=\small]
\tikzstyle{selected edge} = [draw,line width=5pt,-,red!50]

\foreach \pos/\name in {{(1.3598,-0.51594)/A1}, {(1.08392,0.0930052)/A2}, {(0.501281,0.0295443)/A3}, {(-0.46591,0.067417)/A4}, {(-1.67351,-0.254816)/A5}, {(-1.92225,0.413082)/A6}, {(1.06916,-0.0580674)/A7}, {(-0.0310976,-0.101057)/A8}, {(-0.598362,0.323275)/A9}, {(-0.105402,-0.488388)/A10}, {(-0.342471,0.00274053)/A11}, {(-0.13477,0.153304)/A12}, {(1.00977,0.303578)/A13}, {(-0.638966,0.477291)/A14}, {(-0.0818126,0.462008)/A15}, {(0.827614,0.031649)/A16}, {(-1.68586,0.104984)/A17}, {(1.49428,-0.0177128)/A18}, {(0.126462,-0.254585)/A19}, {(1.2249,-0.422183)/A20}, {(-1.34723,0.394874)/A21}, {(-0.418778,-0.574002)/A22}, {(1.37497,-0.458331)/A23}, {(-0.254659,-0.0291145)/A24}, {(1.0664,0.600568)/A25}, {(-1.43921,-0.221509)/A26}, {(-1.73224,-0.569778)/A27}, {(-1.77642,0.162335)/A28}, {(-0.399194,-0.534274)/A29}, {(0.283448,-0.413682)/A30}, {(-1.69618,-0.523518)/A31}, {(1.57535,-0.591451)/A32}}
	\node[vertex] (\name) at \pos {}; 

\node[vertex, label=right:{\small $X_0$}] (X0) at (-1,0) {};
\node[vertex, label=above:{\small $X_1$}] (X1) at (-1.2,-0.1) {};
\node[vertex, label=above:{\small $\ldots$}] (X2) at (-1.8,0.5) {};
\node[selected vertex, label=left:{\small $X_T$}] (XT) at (-0.75,-0.5) {};
\path[edge] (X0) edge (X1) {}; 
\path[edge] (X1) edge (X2) {}; 
\path[edge] (X2) edge (A21) {}; 
\path[edge] (A21) edge (A4) {}; 
\path[edge] (A4) edge (A15) {};
\path[edge] (A15) edge [bend left=10] (A3) {};
\path[edge] (A3) edge [bend left=10] (A15) {};
\path[edge] (A15) edge (A25) {};
\path[edge] (A25) edge (A30) {};
\path[edge] (A30) edge (A20) {};
\path[edge] (A20) edge (A7) {};
\path[edge] (A7) edge (A19) {};
\path[edge] (A19) edge (A26) {};
\path[edge] (A26) edge [bend left=10] (A5) {};
\path[edge] (A5) edge [bend left=10] (A26) {};
\path[edge] (A26) edge (A17) {};
\path[edge] (A17) edge (A21)  {};
\path[edge] (A21) edge (A14) {};
\path[edge] (A14) edge (A9) {};
\path[edge] (A9) edge (XT) {};

\node[visited vertex] at (X0) {};
\node[visited vertex] at (X1) {};
\node[visited vertex] at (X2) {};
\node[revisited vertex] at (A21) {};
\node[visited vertex] at (A4) {};
\node[revisited vertex] at (A15) {};
\node[visited vertex] at (A3) {};
\node[visited vertex] at (A25) {};
\node[visited vertex] at (A30) {};
\node[visited vertex] at (A20) {};
\node[visited vertex] at (A7) {};
\node[visited vertex] at (A19) {};
\node[revisited vertex] at (A26) {};
\node[visited vertex] at (A5) {};
\node[visited vertex] at (A17) {};
\node[visited vertex] at (A14) {};
\node[visited vertex] at (A9) {};

\end{tikzpicture}
\caption{The goal of this paper is to infer the hidden cluster structure underlying a Markov chain $\process{X_t}{t \geq 0}$, from one observation of a sample path $X_0, X_1, \ldots, X_T$ of length $T$.}
\label{fig:Infer_the_hidden_cluster_structure_from_a_MC}
\end{figure}
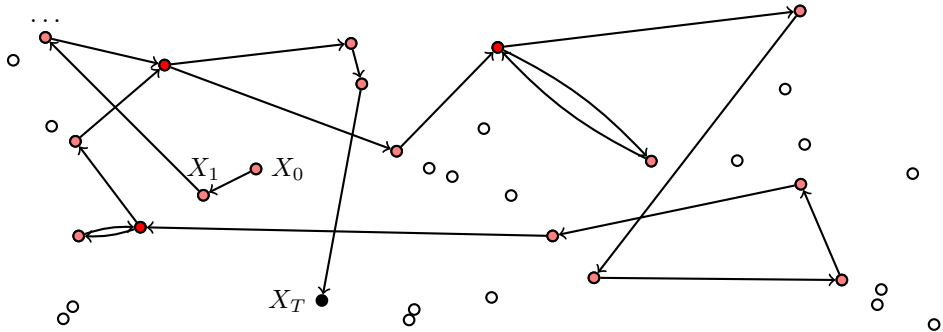

Clustering in \glspl{BMC} is motivated by Reinforcement Learning (RL) problems \cite{Sutton} with large state spaces. These problems have received substantial attention recently due to the wide spectrum of their applications in robotics, games, medicine, finance, etc. RL problems are concerned with the control of dynamical systems modeled as Markov chains whose transition kernels are initially unknown. The objective is to identify an optimal control policy as early as possible by observing the trajectory of a Markov chain generated under some known policy. The time it takes to learn efficient policies using standard algorithms such as Q-learning dramatically increases with the number of possible states, so that these algorithms become useless when the state space is prohibitively large. In most practical problems however, different states may yield similar reward and exhibit similar transition probabilities to other states, i.e., states can be grouped into clusters. In this scenario it becomes critical to learn and leverage this structure in order to speed up the learning process. In this paper we consider \emph{uncontrolled} Markov chains, and we aim to identify clusters of states as quickly as possible. In the future we hope to extend the techniques developed here for an uncontrolled \gls{BMC} to the more general case of controlled Markov chains, and hence to devise reinforcement learning algorithms that will efficiently exploit an underlying cluster structure. The idea of clustering states in reinforcement learning to speed up the learning process has been investigated in \cite{singh1995} and \cite{ortner2013}, but no theoretical guarantees were provided in these early papers.

This paper answers two important questions for the problem of cluster detection on \glspl{BMC}. First, we derive a fundamental information-theoretical clustering error lower bound. The latter allows us to identify the parameters of the \gls{BMC} and the sample path lengths $T$ for which it is theoretically impossible to accurately detect the underlying cluster structure. Second, we develop two clustering algorithms that when combined, are able to accurately detect the underlying cluster structure from the shortest possible sample paths, whenever the parameters of the \gls{BMC} allow detection, and that provably work as $n \to \infty$. These algorithms thus reach the fundamental detectability limit, and are optimal in that sense.

\section{Related work}

Clustering in the \gls{SBM} and the \glspl{BMC} may be seen as similar problems: the objective in both cases is to infer the cluster structure from random observations made on the relationships between pairs of vertices. However, the way these observations are gathered differ significantly in the \gls{SBM} and the \glspl{BMC}. In the \gls{SBM}, these observations are independent random variables, which allows the use of theoretical developments in random matrices with independent entries. In the \glspl{BMC} on the contrary, observations are successive states of a Markov chain and hence are not independent. Furthermore, observed edges in the \gls{SBM} are scattered and undirected, whereas in a \gls{BMC}, the observed path is a concatenation of directed edges. Generally the probabilities to move from state $x$ to state $y$ and from $y$ to $x$ are different. Finally, the sparsity of the observations in the \gls{BMC} is controlled by the length $T$ of the observed sample path, while it is hard-coded in the \gls{SBM}. For all these reasons, it is difficult to quantitatively compare or relate the recovery rates in the two models. Nevertheless, techniques as those used in the \gls{SBM} can be exploited in the analysis of the \gls{BMC} if they are properly extended to handle the differences between the two models. For this reason, we now provide a brief survey of the techniques and results available for the \gls{SBM}.  

Significant advances have been made on cluster recovery within the context of the \gls{SBM} and its generalizations. We defer the reader to \cite{gao_achieving_2015} for an extensive overview. Substantial focus has in particular been on characterizing the set of parameters for which some recovery objectives can be met. 

In the {\it sparse} regime, i.e., when the average degree of vertices is $\bigO{1}$, necessary and sufficient conditions on the parameters have been identified under which it is possible to extract clusters that are positively correlated with the true clusters \cite{decelle_inference_2011,massoulie_community_2014,mossel_reconstruction_2015}. More precisely, for example if $p = a / n$ and $q = b / n$ and in the case of two clusters of equal sizes, it was conjectured in \cite{decelle_inference_2011} that $a-b \ge \sqrt{2(a+b)}$ is a necessary and sufficient condition for the existence of algorithms that can {\it detect} the clusters (in the sense that they perform better than a random assignment of items to clusters). This result was established in \cite{mossel_reconstruction_2015} (necessary condition) and in \cite{massoulie_community_2014} (sufficient condition).

In the {\it dense} regime, i.e., when the average degree is $\omega(1)$, it is possible to devise algorithms under which the proportion of misclassified vertices vanishes as the size of the graph grows large \cite{yun_community_2014}. In this case, one may actually characterize the \emph{minimal} asymptotic (as $n$ grows large) classification error, and develop clustering algorithms achieving this fundamental limit \cite{yun_optimal_2016}. We may further establish conditions under which asymptotic \emph{exact} cluster recovery is possible \cite{yun_accurate_2014,abbe_community_2015,abbe_recovering_2015,jog_information-theoretic_2015,mossel_consistency_2015,abbe_exact_2016,hajek_achieving_2016, yun_optimal_2016}. 

This paper draws considerable inspiration from \cite{yun_community_2014,yun_accurate_2014,yun_optimal_2016}. Over the course of these papers, the authors consider the problem of clustering in the \gls{LSBM}, which is a generalization of the \gls{SBM}. They identify the set of \gls{LSBM}-parameters for which the clusters can be detected using change-of-measure arguments, and develop algorithms based on spectral methods that achieve this fundamental performance limit. Our contributions in this paper include the extension of the approaches to the context of Markov chains. This required us in particular to design novel changes-of-measure, carefully incorporate the effect of mixing, deal with new and non-convex log-likelihood functions, and widen the applicability of spectral methods to random matrices with bias. Note that we restrict the analysis in this paper to the case that the number of clusters $K$ is known. This reduces the complexity of the analysis. Based on the findings in \cite{yun_community_2014,yun_accurate_2014,yun_optimal_2016} however, we are confident that this assumption can be relaxed in future work.

\section{Methodology}

Similar to the extensive efforts for the \gls{SBM}, we will first identify parameters of the \gls{BMC} for which it is theoretically possible to detect the clusters. To this aim, we use techniques from information theory to derive a lower bound on the number of misclassified states that holds for any classification algorithm. This relies on a powerful change-of-measure argument, originally explored in \cite{lai_asymptotically_1985} in the context of online stochastic optimization. First, we relate the probability of misclassifying a state in the \gls{BMC} to a log-likelihood ratio that the sample path was generated by a perturbed Markov chain instead. Then, given any \gls{BMC}, we show how to construct a perturbed Markov chain that assigns a nonzero probability to the event that all clustering algorithms misclassify at least one particular state. Finally, we maximize over all possible perturbations to get the best possible lower bound that holds for any algorithm. 

We will further provide a clustering algorithm that achieves this fundamental limit. Specifically, the algorithm consists of two steps. The first step consists in applying a classical \emph{Spectral Clustering Algorithm}. This algorithm essentially creates a rank-$K$ approximation of a random matrix corresponding to the empirical transition rates between any pair of states, and then uses a $K$-means algorithm \cite{lloyd_least_1982} to cluster all states. We show that this first step clusters the majority of states roughly correctly. Next, we introduce the \emph{Cluster Improvement Algorithm}. This algorithm uses the rough structure learned from the Spectral Clustering Algorithm, together with the sample path, to move each individual state into the cluster the state most likely belongs to. This is achieved through a recursive, local maximization of a log-likelihood ratio.

The key difference between clustering in \glspl{SBM} and clustering in \glspl{BMC} is that instead of observing (the edges of) a random graph, we here try to infer the cluster structure from an as short as possible sample path of the Markov chain. This necessitates a careful analysis of the mixing time of the Markov chain \cite{levin_markov_2009}, for which we use a rate of convergence result in terms of Dobrushin's ergodicity coefficient \cite{bremaud_markov_1999}. The observed sample path will be inherently noisy and biased by construction. The noise and bias within the sample path have to first be related to the spectrum of the random matrix recording the number of times transitions between any two states have been observed. This is done by using techniques from \cite{halko_finding_2011}. The spectrum of this random matrix has then to be analyzed which constitutes a major challenge. Indeed, most results investigating the spectrum of random matrices hold for matrices with independent and weakly dependent entries \cite{wigner_distribution_1958,tao_topics_2012,tropp_introduction_2015,hochstattler_semicircle_2016,kirsch_sixty_2016,kirsch_semicircle_2017}, or when the transition matrix of the Markov chain itself is random \cite{bordenave_spectrum_2010,bordenave_spectrum_2011}. Our random matrix has dependent entries, but by taking proof inspiration from \cite{feige_spectral_2005}, using concentration results from \cite{paulin_concentration_2015}, and smartly leveraging the way it is constructed from the observed sample path and Markov property, the analysis of its spectrum can be conducted.

\section{Overview}

This paper is structured as follows. We introduce the \gls{BMC} in \refSection{sec:The_block_Markov_chain}. \refSection{sec:Summary_of_classifiability_results} provides an overview of our results and our algorithms. We assess the performance of both algorithms, i.e., we quantify their asymptotic error rates. \refSection{sec:Numerical_experiments} discusses several numerical experiments designed to test the algorithms. We subsequently prove our results by first deriving an information lower bound and developing an optimal change-of-measure in \refSection{sec:The_information_bound_and_the_change_of_measure}, and then by developing the Spectral Clustering Algorithm in \refSection{sec:The_SVD_clustering_algorithm} and the Cluster Improvement Algorithm in \refSection{sec:The_cluster_improvement_algorithm}.

\chapter*{Notation}

\newcommand{\DDelta}{\Delta\!\!\!\!\Delta}
For any two sets $\mathcal{A}, \mathcal{B} \subseteq \mathcal{V} \triangleq \{ 1, \ldots, n \}$ we define their symmetric difference by $\mathcal{A} \triangle \mathcal{B} = \{ \mathcal{A} \backslash \mathcal{B} \} \cup \{ \mathcal{B} \backslash \mathcal{A} \}$. For any two numbers $a, b \in \realNumbers$ we introduce the shorthand notations $a \wedge b = \min\{a,b\}$ and $a \vee b = \max\{a,b\}$. For any $n$-dimensional vector $\vect{x} = \vectInLine{ x_1, \ldots, x_n } \in \realNumbers^n$, we define its $l_p$ norms by
\begin{equation}
\pnorm{ \vect{x} }{p} = \Bigl( \sum_{r=1}^n | x_i |^p \Bigr)^{1/p}
\quad
\textrm{where}
\quad
p \in [1,\infty).
\end{equation}
The $n$-dimensional unit vector of which the $r$-th component equals $1$ will be denoted by $\vectElementary{n}{r}$, and the $n$-dimensional vector for which all elements $r \in \mathcal{A} \subseteq \{ 1, \ldots, n \}$ equal $1$ will be denoted by $\vectOnes{\mathcal{A}}$. For any $m \times n$ matrix $A \in \realNumbers^{m \times n}$, we indicate its rows by ${A}_{r,\cdot}$ for $r = 1, \ldots, m$ and its columns by ${A}_{\cdot,c}$ for $c = 1, \ldots, n$. We also introduce the short-hand notation $A_{\mathcal{A},\mathcal{B}} = \sum_{x \in \mathcal{A}} \sum_{y \in \mathcal{B}} A_{x,y}$ for all subsets $\mathcal{A}, \mathcal{B} \subseteq \mathcal{V}$. Its Frobenius norm and spectral norm are defined by
\begin{equation}
\pnorm{A}{\mathrm{F}} 
= \sqrt{ \sum_{r = 1}^m \sum_{c = 1}^n A_{r,c}^2 },
\quad
\pnorm{A}{} 
= \sup_{ \vect{b} \in \mathbb{S}^{n-1} } \{ \pnorm{ A \vect{b} }{2} \},
\end{equation}
respectively. Here, $\mathbb{S}^{n-1} = \{ \vect{x}=(x_1, \ldots, x_{n}) \in (0,1)^{n} : \pnorm{ \vect{x} }{2} = 1 \}$ denotes the $n$-dimensional unit sphere. We define the probability simplex of dimension $n-1$ by
$
\Delta^{n-1} 
= \bigl\{ \vect{x} \in (0,1)^{n} : \pnorm{ \vect{x} }{1} = 1 \bigr\}
$
as well as the set of left-stochastic matrices by
$
\DDelta^{n \times (n-1)}  
= \bigl\{ ( (x_{1,1},\ldots,x_{1,n}), \ldots, \allowbreak (x_{n,1},\ldots,x_{n,n}) ) \in \allowbreak [0,1]^{n \times n} : \sum_{c=1}^n x_{r,c} = 1 \textrm{ for } r = 1, \ldots, n \bigr\}
$
similarly.

In our asymptotic analyses, we write $f(n) \sim g(n)$ if $\lim_{n \to \infty} f(n) / \allowbreak g(n) = 1$, $f(n) = \smallO{g(n)}$ if $\lim_{n \to \infty} f(n) / \allowbreak g(n) = 0$ and $f(n) = \bigO{g(n)}$ if $\limsup_{n \to \infty} \allowbreak f(n) / \allowbreak g(n) < \infty$. Whenever $\{ X_n \}_{n=1}^\infty$ is a sequence of real-valued random variables and $\{ a_n \}_{n=1}^\infty$ a deterministic sequence, we write 
\begin{gather}
X_n = \smallOP{a_n} 
\Leftrightarrow \probabilityBig{ \Bigl| \frac{ X_n }{ a_n } \Bigr| \geq \delta } \rightarrow 0 \, \forall_{ \delta > 0 }
\Leftrightarrow \forall_{\varepsilon, \delta} \exists_{N_{\varepsilon,\delta}} : \probabilityBig{ \Bigl| \frac{ X_n }{ a_n } \Bigr| \geq \delta } \leq \varepsilon \, \forall_{ n > N_{\varepsilon,\delta} },
\\
\textrm{and} \quad
X_n = \bigOP{a_n} 
\Leftrightarrow \forall_{\varepsilon} \exists_{\delta_\varepsilon,N_\varepsilon} : \probabilityBig{ \Bigl| \frac{ X_n }{ a_n } \Bigr| \geq \delta_\varepsilon } \leq \varepsilon \, \forall_{ n > N_\varepsilon }.
\nonumber
\end{gather}
Similarly, $X_n = \OmegaP{a_n}$ denotes $\forall_{\varepsilon} \exists_{\delta_\varepsilon,N_\varepsilon} : \probability{ | X_n / a_n | \leq \delta_\varepsilon } \leq \varepsilon \, \forall_{ n > N_\varepsilon }$, and $X_n \asymp_{\mathbb{P}}(a_n)$ means $\forall_{\varepsilon} \exists_{\delta_\varepsilon^-,\delta_\varepsilon^+,N_\varepsilon} : \probability{ \delta_\varepsilon^- \leq | X_n / a_n | \leq \delta_\varepsilon^+ } \geq 1 - \varepsilon \, \forall_{ n > N_\varepsilon }$.

\chapter{Block Markov Chains (BMCs)}
\label{sec:The_block_Markov_chain}

We assume that we have $n$ states $\mathcal{V} = \{ 1, \ldots, n \}$, each of which is associated to one of $K$ clusters. This means that the set of states is partitioned so that $\mathcal{V} = \cup_{k=1}^K \mathcal{V}_k$ with $\mathcal{V}_k \cap \mathcal{V}_l = \emptyset$ for all $k \neq l$. Let $\sigma(v)$ denote the cluster of a state $v \in \mathcal{V}$. We also assume that there exist constants $\alpha \in \Delta^{K-1}$ so that $\lim_{n \to \infty} \cardinality{ \mathcal{V}_k } / ( n \alpha_k ) = 1$.

For any $\alpha \in \Delta^{K-1}$ and $p \in \DDelta^{K \times (K-1)}$, we define the \gls{BMC} $\process{X_t}{t \geq 0}$ as follows. Its transition matrix $P \in \DDelta^{n \times (n-1)}$ will be defined as
\begin{equation}
P_{x,y} 
\triangleq \frac{ p_{\sigma(x),\sigma(y)} }{ \cardinality{ \mathcal{V}_{\sigma(y)} } - \indicator{ \sigma(x) = \sigma(y) } } \indicator{ x \neq y }
\quad
\textrm{for all }
\quad
x, y \in \mathcal{V}.
\label{eqn:Definition_of_P}
\end{equation}
Note that this Markov chain is not necessarily reversible. Furthermore, note that in this paper we assume that $K, \alpha, p$ are fixed, and that we study the asymptotic regime $n \to \infty$. We assume that the smallest cluster has a size linearly growing with $n$: $\alpha_{\min}\triangleq \min_k\alpha_k >0$. Finally, since we are interested in clustering of the states, we will assume that $\exists_{1 < \eta} : \max_{a,b,c} \{ p_{b,a} / p_{c,a}, \allowbreak p_{a,b} / p_{a,c} \} \leq \eta$, which guarantees a minimum level of separability of the parameters.

\section{Equilibrium behavior}

We assume that the stochastic matrix $p$ is such that the equilibrium distribution of $\process{ X_t }{ t \geq 0}$ exists, and we will denote it by $\Pi_x$ for $x \in \mathcal{V}$. By symmetry, $\Pi_x = \Pi_y \triangleq \bar{\Pi}_k$ for any two states $x,y \in \mathcal{V}_k$ for all $k = 1, \ldots, K$. Consider the scaled quantity
\begin{equation}
\pi_k 
\triangleq \lim_{n \to \infty} \sum_{x \in \mathcal{V}_k} \Pi_x 
= \lim_{n \to \infty} \cardinality{ \mathcal{V}_{k} } \bar{\Pi}_k
\quad 
\textrm{for}
\quad
k = 1, \ldots, K.
\end{equation}
\refProposition{prop:Equilibrium_behavior_of_pi}'s proof can be found in \refSupplementaryMaterial{suppl:Asymptotic_equilibrium_behavior_of_a_BMC}, and follows from the symmetries between the states within the same clusters and the specific scalings of $P$'s elements.

\begin{proposition}
\label{prop:Equilibrium_behavior_of_pi}
The quantity $\vect{\pi}$ solves $\transpose{ \vect{\pi} } p = \transpose{ \vect{\pi} }$, and is therefore the equilibrium distribution of a Markov chain with transition matrix $p$ and state space $\stateSpace = \{ 1, \ldots, K \}$.
\end{proposition}

\section{Mixing time}

\refProposition{prop:Mixing_times_of_Markov_chains_with_transition_matrices_P_and_Q} gives a bound on the mixing time $t_{\mathrm{mix}} \in \strictlyPositiveRealNumbers$, which is defined by
$
d(t) 
\triangleq \sup_{x \in \mathcal{V}} \bigl\{ d_{\mathrm{TV}}( {P}^t_{x,\cdot}, \vect{\Pi} ) \bigr\}
$
and
$
t_{\mathrm{mix}}(\varepsilon) 
\triangleq \min \{ t \geq 0 : d(t) \leq \varepsilon \},
$
where
\begin{equation}
d_{\mathrm{TV}}( \vect{\mu}, \vect{\nu} )
\triangleq \tfrac{1}{2} \sum_{x \in \mathcal{V}} | \mu_x - \nu_x | .
\label{eqn:Definition_of_total_variation_distance}
\end{equation}
The proof of \refProposition{prop:Mixing_times_of_Markov_chains_with_transition_matrices_P_and_Q} is deferred to \refAppendixSection{sec:Proof_of_Mixing_times_of_Markov_chains_with_transition_matrices_P_and_Q}. The result follows after bounding Dobrushin's ergodicity coefficient \cite{bremaud_markov_1999} using $P$'s structure, and invoking a convergence rate result in terms of Dobrushin's coefficient.

\begin{proposition}
\label{prop:Mixing_times_of_Markov_chains_with_transition_matrices_P_and_Q}
For any BMC with $n\ge 4/\alpha_{\min}$, $t_{\mathrm{mix}}(\varepsilon) \leq - c_{\mathrm{mix}} \ln{\varepsilon}$, where $c_{\mathrm{mix}}=-1/\ln(1-1/2\eta)$.
\end{proposition}
%

\refProposition{prop:Mixing_times_of_Markov_chains_with_transition_matrices_P_and_Q} implies that the mixing times are short enough so that our results will hold \emph{irrespective} of whether we assume that the Markov chain is initially in equilibrium. We will show in \refSection{sec:Asymptotic_negligibilit_of_VarLwrtQ} that what is important is that the chain reaches stationarity \emph{within} $T$ steps (the length of the observed trajectory), and consequentially, $T$ needs to be chosen sufficiently large with respect to $n$ to ensure that this occurs. Throughout this paper we therefore assume for simplicity that the chain is started from equilibrium. This eliminates the need of tracking higher order correction terms.

\paragraph{Examples} \refFigure{fig:Example_BMC_with_K2} illustrates the structure of a \gls{BMC} when there are $K=2$ groups. We find after solving the balance equations that the limiting equilibrium behavior is given by $\pi_1 = p_{21} / ( p_{12} + p_{21} )$ and $\pi_2 = p_{12} / ( p_{12} + p_{21} )$.

\begin{figure}[!hbtp]
\centering
\begin{tikzpicture}[scale=2.4]
\tikzstyle{vertex}=[circle, draw, thick, fill=black!0, inner sep=0pt, minimum width=4pt]
\tikzstyle{selected vertex} = [circle, draw, thick, fill=black!100, inner sep=0pt, minimum width=4pt]
\tikzstyle{edge} = [draw,thick,->]
\tikzstyle{weight} = [font=\small]
\tikzstyle{selected edge} = [draw,line width=5pt,-,red!50]

\node[draw=none] (V1) at (-1.7,0.3) {Cluster $\mathcal{V}_1$}; 
\node[draw=none] (V2) at (1.4,-0.125) {Cluster $\mathcal{V}_2$}; 

\node[selected vertex, label={$X_t$}] (X) at (-0.5,0.1) {};
\node[vertex] (A) at (-0.6,-0.4) {};
\node[vertex] (B) at (0.25,0.25) {};
\path[edge] (X) -- node[label=left:{\small $\frac{1-p_{1,2}}{ \cardinality{ \mathcal{V}_1 } - 1 }$}] {} (A);
\path[edge] (X) -- node[label=above:{\small $\frac{p_{1,2}}{ \cardinality{ \mathcal{V}_2 } }$}] {} (B);

\node[vertex] (Y) at (0.4,-0.2) {};
\node[vertex] (C) at (-0.7,0.05) {};
\node[vertex] (D) at (0.6,0.4) {};
\path[edge] (Y) -- node[label=below:{\small $\frac{p_{2,1}}{ \cardinality{ \mathcal{V}_1 } }$}] {} (C);
\path[edge] (Y) -- node[label=right:{\small $\frac{1-p_{2,1}}{ \cardinality{ \mathcal{V}_2 } - 1 }$}] {} (D);

\foreach \pos/\name in {{(-0.720098,-0.416295)/A1}, {(-1.14936,0.424865)/A3}, {(-1.63295,-0.0786896)/A4}, {(-2.23676,0.329519)/A5}, {(-2.36113,-0.444716)/A6}, {(-0.865418,-0.535262)/A7}, {(-1.41555,-0.548913)/A8}, {(-1.69918,-0.123351)/A9}, {(-1.4527,0.0875855)/A10}, {(-1.57124,-0.52412)/A11}, {(-1.46739,0.486784)/A12}, {(-0.895113,-0.51594)/A13}, {(-1.71948,0.0930052)/A14}, {(-1.44091,0.0295443)/A15}, {(-0.986193,0.067417)/A16}, {(-2.24293,-0.254816)/A17}, {(-0.652858,0.413082)/A18}, {(-1.33677,-0.0580674)/A19}}
	\node[vertex] (\name) at \pos {};

\foreach \pos/\name in {{(1.9231,-0.534274)/B1}, {(0.609729,-0.413682)/B2}, {(1.40443,-0.523518)/B3}, {(1.64806,-0.591451)/B4}, {(1.89123,-0.36334)/B5}, {(2.17232,-0.534266)/B6}, {(2.14759,-0.0435297)/B7}, {(1.45121,-0.407763)/B8}, {(1.56988,0.311601)/B9}, {(1.37134,0.481318)/B10}, {(0.988049,0.23501)/B11}, {(0.716856,-0.0812623)/B12}}
	\node[vertex] (\name) at \pos {};

\draw[thick,dashed] (-1.5,0) ellipse (1.2 and 0.8);
\draw[thick,dashed] (1.3,-0.1) ellipse (1.3 and 0.75);
\end{tikzpicture}
\caption{In the \gls{BMC} with $K = 2$ groups $\mathcal{V}_1 \cup \mathcal{V}_2 = \mathcal{V}$, whenever the Markov chain is at some state $X_t \in \mathcal{V}_1$, it will next jump with probability $p_{1,2}$ to cluster $\mathcal{V}_2$, and with probability $1-p_{1,2}$ to some other state in cluster $\mathcal{V}_1$. Similarly, if $X_t \in \mathcal{V}_2$, it would next jump to cluster $\mathcal{V}_1$ with probability $p_{2,1}$, or stay within its own cluster with probability $1-p_{2,1}$.}
\label{fig:Example_BMC_with_K2}
\end{figure}
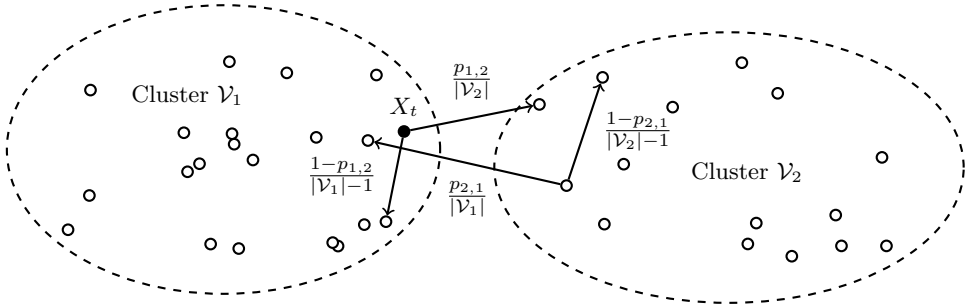

For $K = 3$, we find after solving the balance equations that the limiting equilibrium behavior is given by
\begin{equation}
\pi_1 = \frac{ p_{23} p_{31} + p_{21} ( p_{31} + p_{32} ) }{Z(p)},
\quad
\pi_2 = \frac{ p_{13} p_{32} + p_{12} ( p_{31} + p_{32} ) }{Z(p)},
\end{equation}
$\pi_3 = 1 - \pi_1 - \pi_2$, with $Z(p) = ( p_{21} + p_{23} ) ( p_{13} + p_{31} ) + ( p_{13} + p_{21} ) p_{32} + p_{12} ( p_{23} + p_{31} + p_{32} )$. Let us also illustrate the structure of the transition matrix when $\alpha = ( 2/10, 3/10, 5/10 )$ and $n = 10$:
\begin{equation}
P
=
\left(
\begin{array}{cc;{2pt/2pt}ccc;{2pt/2pt}ccccc}
0 & p_{1,1} & \frac{p_{1,2}}{3} & \frac{p_{1,2}}{3} & \frac{p_{1,2}}{3} & \frac{p_{1,3}}{5} & \frac{p_{1,3}}{5} & \frac{p_{1,3}}{5} & \frac{p_{1,3}}{5} & \frac{p_{1,3}}{5} \\
p_{1,1} & 0 & \frac{p_{1,2}}{3} & \frac{p_{1,2}}{3} & \frac{p_{1,2}}{3} & \frac{p_{1,3}}{5} & \frac{p_{1,3}}{5} & \frac{p_{1,3}}{5} & \frac{p_{1,3}}{5} & \frac{p_{1,3}}{5} \\ \hdashline[2pt/2pt]
\frac{p_{2,1}}{2} & \frac{p_{2,1}}{2} & 0 & \frac{p_{2,2}}{2} & \frac{p_{2,2}}{2} & \frac{p_{2,3}}{5} & \frac{p_{2,3}}{5} & \frac{p_{2,3}}{5} & \frac{p_{2,3}}{5} & \frac{p_{2,3}}{5} \\
\frac{p_{2,1}}{2} & \frac{p_{2,1}}{2} & \frac{p_{2,2}}{2} & 0 & \frac{p_{2,2}}{2} & \frac{p_{2,3}}{5} & \frac{p_{2,3}}{5} & \frac{p_{2,3}}{5} & \frac{p_{2,3}}{5} & \frac{p_{2,3}}{5} \\
\frac{p_{2,1}}{2} & \frac{p_{2,1}}{2} & \frac{p_{2,2}}{2} & \frac{p_{2,2}}{2} & 0 & \frac{p_{2,3}}{5} & \frac{p_{2,3}}{5} & \frac{p_{2,3}}{5} & \frac{p_{2,3}}{5} & \frac{p_{2,3}}{5} \\ \hdashline[2pt/2pt]
\frac{p_{3,1}}{2} & \frac{p_{3,1}}{2} & \frac{p_{3,2}}{3} & \frac{p_{3,2}}{3} & \frac{p_{3,2}}{3} & 0 & \frac{p_{3,3}}{4} & \frac{p_{3,3}}{4} & \frac{p_{3,3}}{4} & \frac{p_{3,3}}{4} \\
\frac{p_{3,1}}{2} & \frac{p_{3,1}}{2} & \frac{p_{3,2}}{3} & \frac{p_{3,2}}{3} & \frac{p_{3,2}}{3} & \frac{p_{3,3}}{4} & 0 & \frac{p_{3,3}}{4} & \frac{p_{3,3}}{4} & \frac{p_{3,3}}{4} \\
\frac{p_{3,1}}{2} & \frac{p_{3,1}}{2} & \frac{p_{3,2}}{3} & \frac{p_{3,2}}{3} & \frac{p_{3,2}}{3} & \frac{p_{3,3}}{4} & \frac{p_{3,3}}{4} & 0 & \frac{p_{3,3}}{4} & \frac{p_{3,3}}{4} \\
\frac{p_{3,1}}{2} & \frac{p_{3,1}}{2} & \frac{p_{3,2}}{3} & \frac{p_{3,2}}{3} & \frac{p_{3,2}}{3} & \frac{p_{3,3}}{4} & \frac{p_{3,3}}{4} & \frac{p_{3,3}}{4} & 0 & \frac{p_{3,3}}{4} \\
\frac{p_{3,1}}{2} & \frac{p_{3,1}}{2} & \frac{p_{3,2}}{3} & \frac{p_{3,2}}{3} & \frac{p_{3,2}}{3} & \frac{p_{3,3}}{4} & \frac{p_{3,3}}{4} & \frac{p_{3,3}}{4} & \frac{p_{3,3}}{4} & 0 \\
\end{array}
\right)
\label{eqn:Example_P_matrix}
\end{equation}

\chapter{Main results}
\label{sec:Summary_of_classifiability_results}

In this paper we obtain quantitative statements on the set of misclassified states,
\begin{equation}
\mathcal{E} 
\triangleq \bigcup_{k=1}^K \hat{\mathcal{V}}_{\criticalpoint{\gamma}(k)} \backslash \mathcal{V}_k
\quad
\textrm{where}
\quad
\criticalpoint{\gamma}
\in \arg \min_{ \gamma \in \textrm{Perm}(K) } \Bigl| \bigcup_{k=1}^K \hat{\mathcal{V}}_{\gamma(k)} \backslash \mathcal{V}_k \Bigr|.
\end{equation}
Here, the sets $\hat{\mathcal{V}}_1, \ldots, \hat{\mathcal{V}}_K$ will always denote an approximate cluster assignment obtained from some clustering algorithm. For notational convenience we will always number the approximate clusters so as to minimize the number of misclassifications, allowing us to forego defining it formally via a permutation.

\section{Information theoretical lower bound}

Our results identify an important information quantity $I(\alpha,p) \geq 0$ that measures how difficult it is to cluster in a \gls{BMC}. Its role will become clear in \refTheorem{thm:Information_bound}. The reason we call it an information quantity stems from fact that we have derived it as the leading coefficient in an asymptotic expansion of a log-likelihood function. Note that while it resembles one, this information quantity is not a Kullback--Leibler divergence. The individual terms are weighted according to the equilibrium distribution $\pi$, and there are two extra terms.

\begin{definition}
For $\alpha \in \Delta^{K-1}$ and $p \in \DDelta^{(K-1) \times K}$, let
\begin{equation}
I(\alpha,p) 
\triangleq \min_{ a \neq b }I_{a,b}(\alpha,p),
\label{eqn:Ialphabetap}
\end{equation}
where 
$
I_{a,b}(\alpha,p) \triangleq \Bigl\{ \sum_{k=1}^K \frac{1}{\alpha_a} \Bigl( \pi_a p_{a,k} \ln{ \frac{ p_{a,k} }{ p_{b,k} } } + \pi_k p_{k,a} \ln{ \frac{ p_{k,a} \alpha_b }{ p_{k,b} \alpha_a } } \Bigr) + \Bigl( \frac{ \pi_b }{ \alpha_b } - \frac{ \pi_a }{ \alpha_a } \Bigr) \Bigr\}.
$
Here $\pi$ denotes the solution to $\transpose{\pi} p = \transpose{\pi}$.
\end{definition}

\begin{theorem}
\label{thm:Information_bound}
\jaron{An algorithm is $(\varepsilon,c)$-locally good at $(\alpha,p)$ if it satisfies $\expectationWrt{ \cardinality{\mathcal{E}} }{P} \leq \varepsilon$ for all BMC models constructed from the given $p$ and partitions satisfying $| \cardinality{ \mathcal{V}_k } - \alpha_k n | \leq c$ for all $k$.}
Assume that $T=\omega(n)$. Then there exists a strictly positive and finite constant $C$ independent of $n$ such that: \jaron{there exists no $(\varepsilon,1)$-locally good} clustering algorithm \jaron{at $(\alpha,p)$} when
\begin{equation}
\varepsilon 
< C n \exp{ \Big( - I(\alpha,p) \frac{T}{n}\big( 1+ \smallO{ 1 }\big) \Bigr) }.
\end{equation}

\end{theorem}

\refTheorem{thm:Information_bound} allows us to state necessary conditions for the existence of \jaron{$(\varepsilon,1)$-locally good} clustering algorithms \jaron{at $(\alpha,p)$} that either detect clusters {\it asymptotically accurately}, namely with $\expectationWrt{ \cardinality{\mathcal{E}} }{P}=\smallO{n}$, or recover clusters {\it asymptotically exactly}, i.e., with $\expectationWrt{ \cardinality{\mathcal{E}} }{P}=\smallO{1}$. 

\paragraph{Conditions for asymptotically accurate detection} In view of the lower bound in \refTheorem{thm:Information_bound}, there may exist asymptotically accurate \jaron{$(\varepsilon,1)$-locally good} clustering algorithms \jaron{at $(\alpha,p)$} only if $I(\alpha,p)>0$ and $T=\omega(n)$.

\paragraph{Conditions for asymptotically exact detection} Necessary conditions for the existence of an asymptotically exact algorithm are $I(\alpha,p)>0$ and $T-{n\ln(n)\over I(\alpha,p)}=\omega(1)$. In particular, $T$ must be larger than $n\ln(n)$. We refer to the scenario where $T$ is of the order $n\ln(n)$ as the {\it critical regime}. In this regime when $T=n\ln(n)$, the necessary condition for exact recovery is $I(\alpha,p)>1$. 

Note that qualitatively, the above conditions on the number of observations for accurate and exact recovery are similar to those in the \gls{SBM}. In the latter, the average degrees of vertices should be such that the average total number of edges is $\omega(n)$ \cite{yun_community_2014} for accurate detection, whereas this average must be at least $cn\ln(n)$ for exact recovery\cite{abbe_community_2015}. Here, $c$ is known and depends on the parameters of the \gls{SBM}.   

\paragraph{The information quantity $I(\alpha,p)$ for $K = 2$ clusters} In the case of two clusters, we study the set of parameters $(\alpha,p)$ of the \glspl{BMC} for which $I(\alpha,p)>0$ and $I(\alpha,p)>1$, the latter condition being necessary in the critical regime when $T=n\ln(n)$. 

A system with two clusters can be specified entirely with three parameters: $\alpha_2$, $p_{1,2}$, and $p_{2,1}$. Examining the explicit expression for \refEquation{eqn:Ialphabetap} in this case, we can conclude that $I(\alpha,p) = 0$ if and only if $\alpha_2 = p_{1,2} = 1 - p_{2,1}$. Asymptotic accurate (resp. exact) recovery seems thus possible as soon as $T=\omega(n)$ (resp. $T=\omega(n\ln(n))$) for almost any \gls{BMC} with two clusters -- the only exception are \glspl{BMC} with parameters on this line. Note that if we did not have the information quantity at our disposal, it would be challenging to give a heuristic argument whether a specific \gls{BMC} allows for asymptotic exact recovery. Consider for instance a \gls{BMC} with $\alpha_1 = \alpha_2 = \tfrac{1}{2}$ and $p_{1,2} = 1 - p_{2,1} \neq \tfrac{1}{2}$ and $p_{1,2} > p_{2,1}$ w.l.o.g. In this scenario, $P_{x,z} = P_{y,z}$ for all $x,y,z \in \mathcal{V}$, that is, every row of the kernel is identical to any other row. Looking at this kernel, we would not expect to be able to cluster. However here $\pi_2 > \pi_1$, and we could cluster based on the equilibrium distribution as $T \to \infty$. The information quantity takes the fact that we are dealing with a Markov chain appropriately into account, and correctly asserts for this case that asymptotic recovery is possible.

\refFigure{fig:Clusterable_regions_on_critical_threshold} illustrates for which parameters one can possibly recover the two clusters asymptotically exactly when $T = n \ln{n}$. Specifically, it depicts all parameters $\alpha_2, p_{1,2}, p_{2,1} \in (0,1)$ for which $I(\alpha,p) > 1$. If we fix $\alpha_2$, note that when $p_{1,2}, p_{2,1} \downarrow 0$ (bottom left), the Markov chain tends to stay within the current cluster for a substantial time. Similarly when $p_{1,2}, p_{2,1} \uparrow 1$ (top right), the Markov chain tends to jump into the other cluster every time. In both scenarios, the states are relatively easy to cluster. This draws parallels with the \gls{SBM}. When either $p_{1,2} \downarrow 0$ (left) or $p_{2,1} \downarrow 0$ (bottom), clustering is again doable: in these scenarios, the Markov chain tends to stay in the cluster of the starting state -- and the fact that you never see the other vertices suggests that they have other transitions rates and therefore belong to the other cluster.

\begin{figure}[!hbtp]
\captionsetup[subfigure]{labelformat=empty}
\centering
\begin{subfigure}[t]{0.30\textwidth}
\centering
\includegraphics[width=\linewidth]{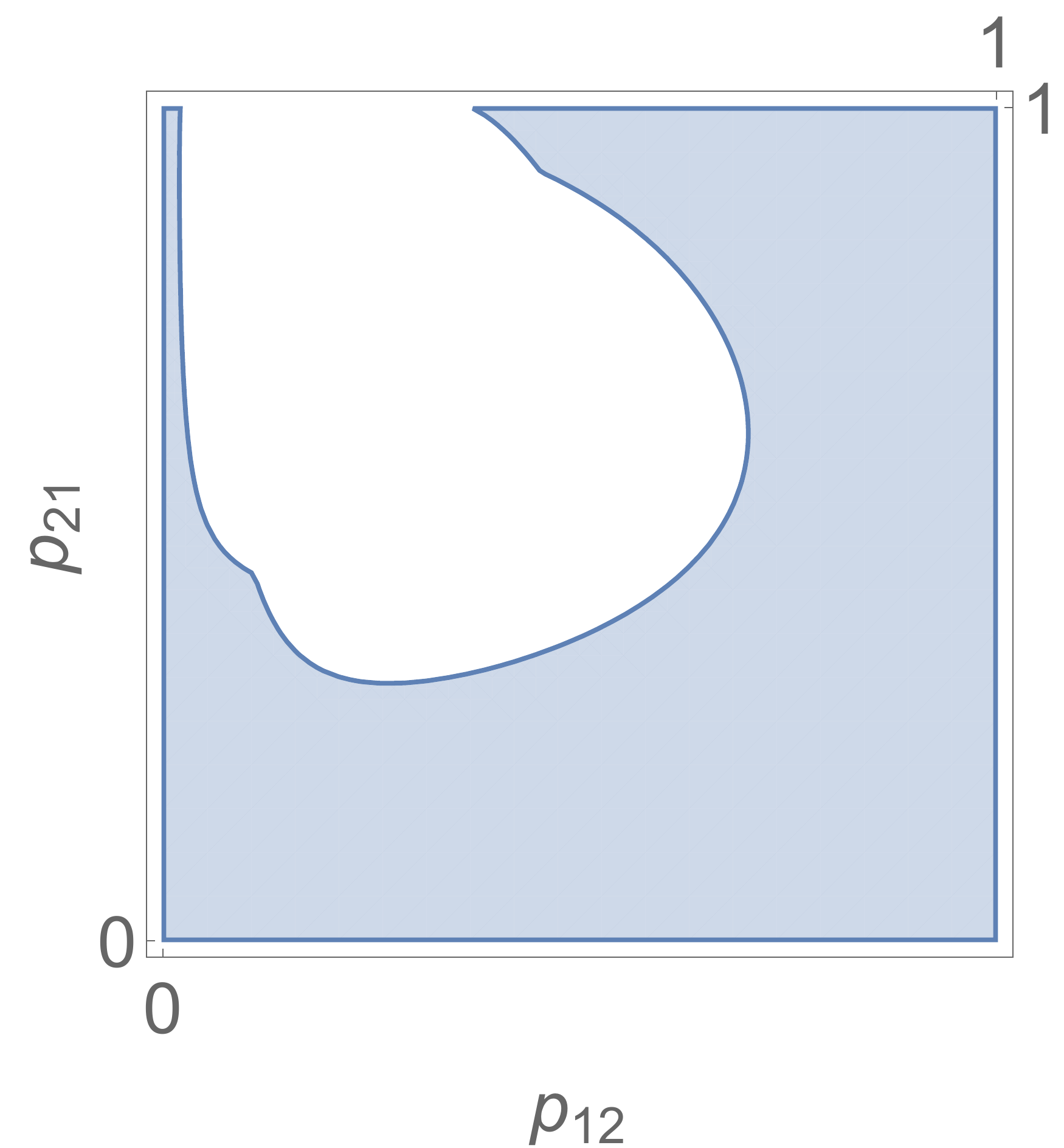}
\caption{$\alpha_2 = 1/4$}
\end{subfigure}
~
\begin{subfigure}[t]{0.30\textwidth}
\centering
\includegraphics[width=\linewidth]{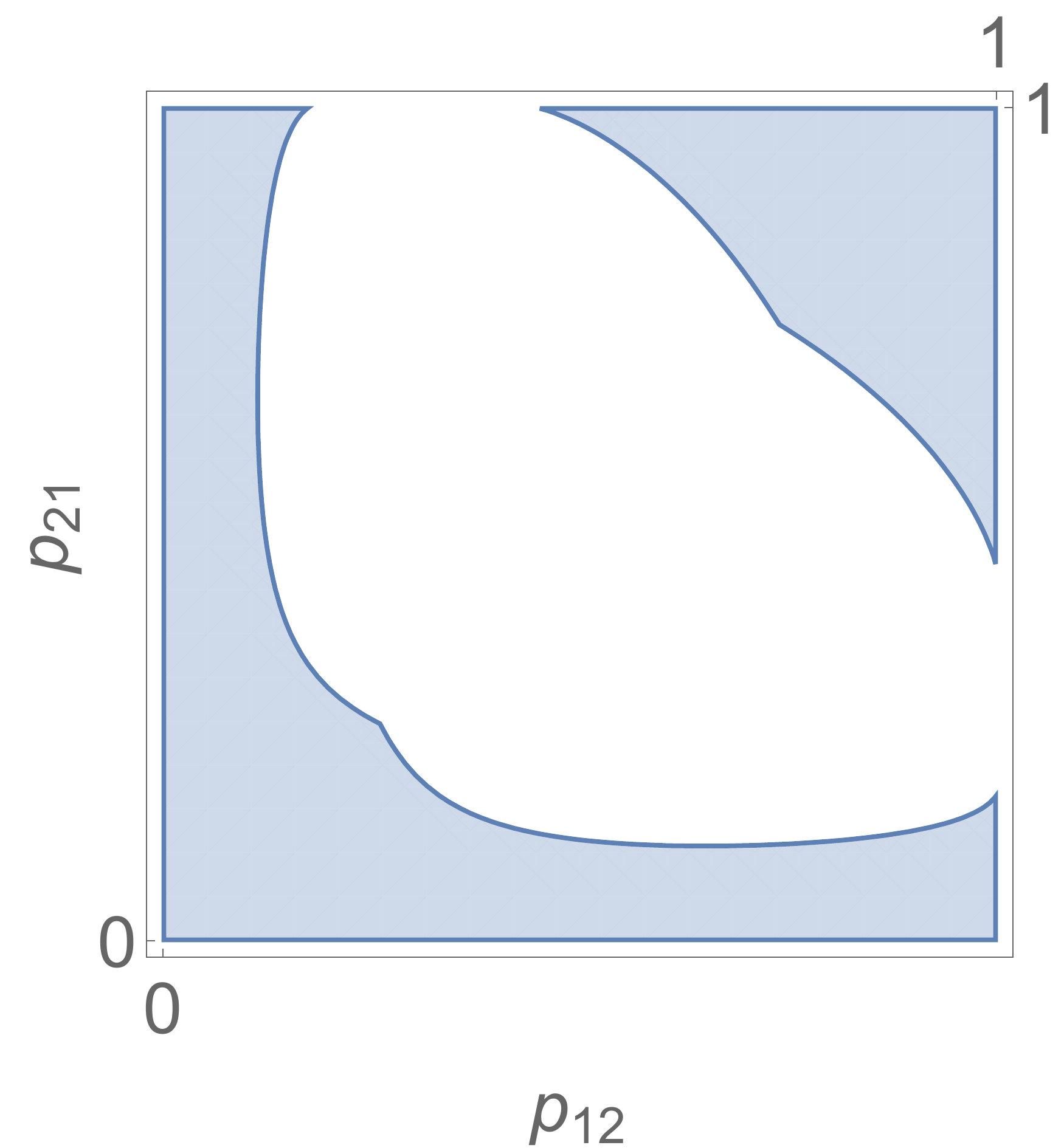}
\caption{$\alpha_2 = 1/2$}
\end{subfigure}
~
\begin{subfigure}[t]{0.33\textwidth}
\centering
\includegraphics[width=\linewidth]{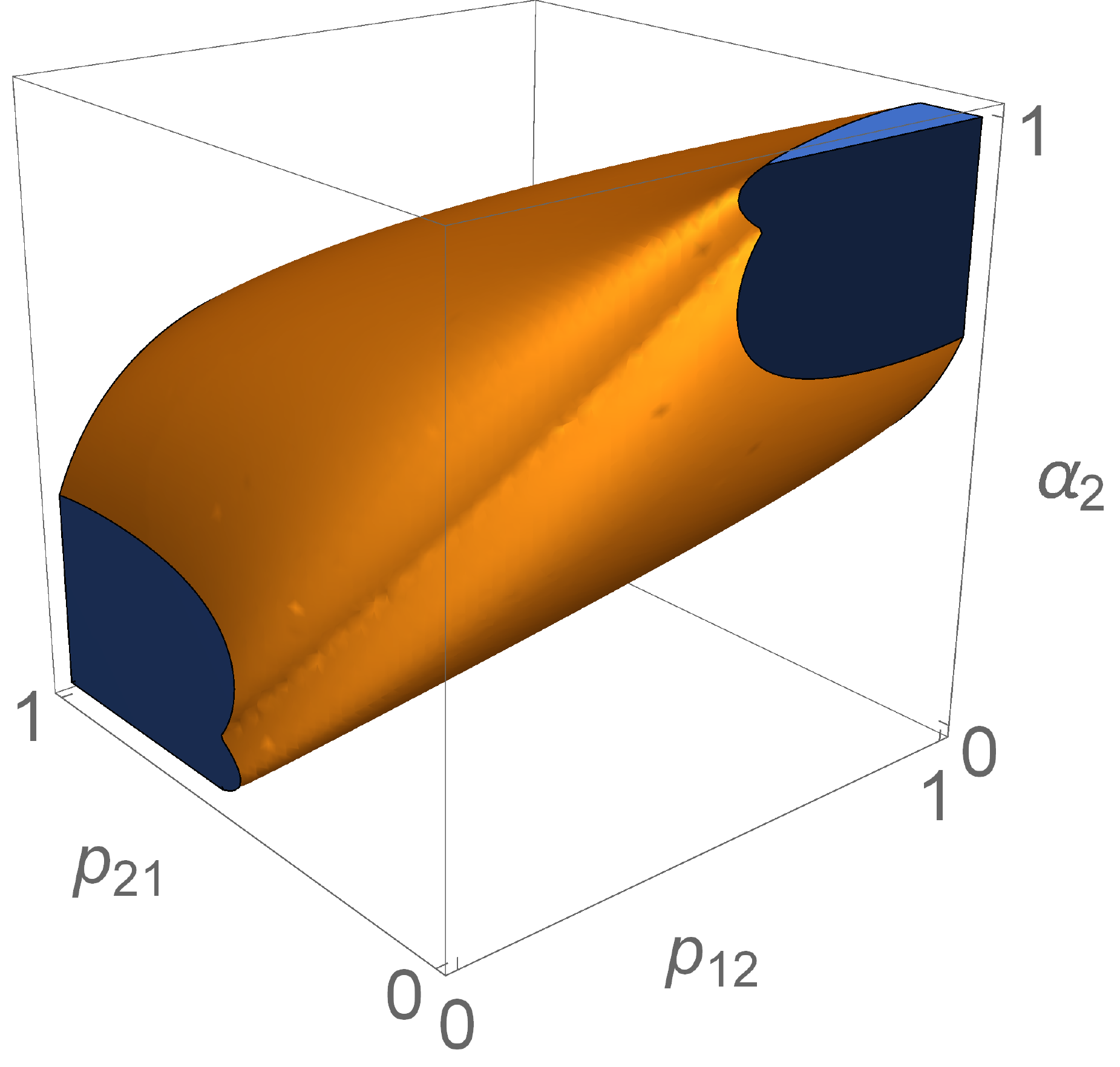}
\caption{Region where $I(\alpha,p)<1$.}
\end{subfigure}
\caption{(left, middle) The parameters $(p_{1,2}, p_{2,1})$ in blue for which asymptotic exact recovery should be possible in the critical regime $T = n \ln{n}$ for $K = 2$ clusters. (right) The parameters $( \alpha_2, p_{1,2}, p_{2,1} )$ for which asymptotic exact recovery is likely not possible, i.e., $I(\alpha,p)<1$.}
\label{fig:Clusterable_regions_on_critical_threshold}
\end{figure}

\section{Procedure for cluster recovery}

\refTheorem{thm:Information_bound} has established necessary conditions for asymptotically accurate and exact recovery, and has identified performance limits satisfied by any \jaron{$(\varepsilon,1)$-locally good} clustering algorithm \jaron{at $(\alpha,p)$}. In this section, we devise a \jaron{$(\varepsilon,1)$-locally good} clustering algorithm \jaron{at $(\alpha,p)$} that reaches these limits order-wise. The proposed algorithm proceeds in two steps: The first step performs a spectral decomposition of the random matrix $\hat{N}$ corresponding to the empirical transition rates between any pair of states, and defined by
\begin{equation}
\hat{N}_{x,y} \triangleq \sum_{t=0}^{T-1} \indicator{ X_t = x, X_{t+1} = y}
\quad
\textrm{for} 
\quad
x, y \in \mathcal{V}.
\label{eqn:Definitions_of_Nhatxy_and_Nxy}
\end{equation}
The rank-$K$ approximation of $\hat{N}$ is used to get initial estimates of the clusters. The second step sequentially improves the cluster estimates. In each iteration, the parameters of the \glspl{BMC} are inferred from the previous cluster estimates, and states are re-assigned to clusters based on these estimated parameters and the observed trajectory (by maximizing a log-likelihood).

\subsection{Spectral Clustering Algorithm}

The first step of our procedure is the Spectral Clustering Algorithm, presented in Algorithm~\ref{pseudo:The_SVD_Clustering_Algorithm}. It leverages the spectral decomposition of $\hat{N}$ to estimate the clusters. 

Before applying a singular value decomposition (SVD) to $\hat{N}$, we first need to trim the matrix so as to remove states that have been visited abnormally often. These states would namely perturb the spectral decomposition of $\hat{N}$. More precisely, we define the set $\Gamma$ of states obtained from ${\cal V}$ by removing the $\lfloor n\exp{ ( -(T/n) \ln{ (T/n) } ) } \rfloor$ states with the highest numbers of visits in the observed sample path of length $T$. The spectral decomposition is applied to the matrix $\hat{N}_\Gamma$ obtained from $\hat{N}$ by setting all entries on the rows and columns corresponding to states not in $\Gamma$ to zero. 

The SVD of $\hat{N}_\Gamma$ is $U \Sigma \transpose{V}$, from which we deduce $\hat{R}$ the best rank-$K$ approximation of $\hat{N}_\Gamma$: $\hat{R} \triangleq \sum_{k=1}^K \sigma_k \vect{U}_{\cdot,k} \transpose{ \vect{V}_{\cdot,k} }$,  where the values $\sigma_1 \geq \sigma_2 \geq \cdots \geq \sigma_n \geq 0$ denote the singular values of $\hat{N}_\Gamma$ in decreasing order. We apply a clustering algorithm to the rows and columns of $\hat{R}$ to determine the clusters. While in practice you may choose to use a different algorithm, for the analysis we use the following: first we calculate the \emph{neighborhoods}
\begin{equation}
\mathcal{N}_x 
\triangleq \Bigl\{ y \in \mathcal{V} \Big| \sqrt{ \pnorm{ {\hat{R}}_{x,\cdot} - {\hat{R}}_{y,\cdot} }{2}^2 + \pnorm{ {\hat{R}}_{\cdot,x} - {\hat{R}}_{\cdot,y} }{2}^2} \leq \frac{1}{n} \cdot \Bigl( \frac{T}{n} \Bigr)^{3/2} \Bigl( \ln{ \frac{T}{n} } \Bigr)^{4/3} \Bigr\}
\label{eqn:Definition_of_neighborhoods}
\end{equation}
for $x \in \mathcal{V}$. Then we initialize $\hat{\mathcal{V}}_k \gets \emptyset$ for $k = 1, \ldots, K$ and sequentially select $K$ \emph{centers} $z_1^\ast, \ldots, z_K^\ast \in \mathcal{V}$ from which we construct approximate clusters. Specifically, we iterate for $k = 1, \ldots K$:
\begin{equation}
\hat{\mathcal{V}}_k
\gets \mathcal{N}_{z_k^\ast} \backslash \bigl\{ \cup_{l=1}^{k-1} \hat{\mathcal{V}}_l \bigr\}
\quad
\textrm{where}
\quad
z_k^\ast 
\triangleq \arg \max_{x \in \mathcal{V}} \cardinality{ \mathcal{N}_x \backslash \bigl\{ \cup_{l=1}^{k-1} \hat{\mathcal{V}}_l \bigr\} }.
\label{eqn:SVD_centers}
\end{equation}
Any remaining state is finally associated to the center closest to it, i.e., we iterate for $y \in \{ \cup_{k=1}^K \hat{\mathcal{V}}_k \}^{\mathrm{c}}$
\begin{equation}
\hat{\mathcal{V}}_{k_y^\ast}
\gets \hat{\mathcal{V}}_{k_y^\ast} \cup \{ y \}
\,\,\,
\textrm{with}
\,\,\,
k_y^\ast 
\triangleq \arg \min_{k = 1, \ldots, K} \sqrt{ \pnorm{ {\hat{R}}_{z_k^\ast,\cdot} - {\hat{R}}_{y,\cdot} }{2}^2 +\pnorm{ {\hat{R}}_{\cdot,z_k^\ast} - {\hat{R}}_{\cdot,y} }{2}^2}.
\label{eqn:SVD_remainders}
\end{equation}
Finally, the Spectral Clustering Algorithm outputs $\hat{\mathcal{V}}_k$ for $k = 1, \ldots, K$. \refTheorem{thm:Upper_bound_on_SVDCAs_performance} provides an upper bound on the number of misclassified states after executing the algorithm.

\begin{algorithm}[!hbtp]
\KwIn{$n,K$, and a trajectory $X_0, X_1, \ldots, X_T$}
\KwOut{An approximate cluster assignment $\itr{\hat{\mathcal{V}}_1}{0}, \ldots, \itr{\hat{\mathcal{V}}_K}{0}$, and matrix $\hat{N}$}
\Begin{
\For{$x \leftarrow 1$ \KwTo $n$}{
\For{$y \leftarrow 1$ \KwTo $n$}{
$\hat{N}_{x,y} \leftarrow \sum_{t=0}^{T-1} \indicator{ X_{t} = x, X_{t+1} = y }$\;
}
}
Calculate the trimmed matrix $\hat{N}_{\Gamma}$\;
Calculate the \gls{SVD} $U \Sigma \transpose{V}$ of $\hat{N}_\Gamma$\;
Order $U, \Sigma, V$ s.t.\ the singular values $\sigma_1 \geq \sigma_2 \geq \ldots \geq \sigma_n \geq 0$ are in descending order\;
Construct the rank-$K$ approximation $\hat{R} = \sum_{k=1}^K \sigma_k \vect{U}_{\cdot,k} \transpose{ \vect{V}_{\cdot,k} }$\;
Apply a $K$-means algorithm to $[\hat{R},\hat{R}^\top ]$ to determine $\itr{\hat{\mathcal{V}}_1}{0}, \ldots, \itr{\hat{\mathcal{V}}_K}{0}$\;
}
\caption{Pseudo-code for the Spectral Clustering Algorithm.}
\label{pseudo:The_SVD_Clustering_Algorithm}
\end{algorithm}

\begin{theorem}
\label{thm:Upper_bound_on_SVDCAs_performance}
Assume that $T=\omega(n)$ and $I(\alpha,p)>0$. Then the proportion of misclassified states after the Spectral Clustering Algorithm satisfies:
\begin{equation}
\frac{ \cardinality{\mathcal{E}} }{n}
= \bigOPbig{ \frac{ n }{ T } \ln{ \frac{T}{n} } }=\smallOP{1}.
\label{eqn:DNalphabetap}
\end{equation}
\end{theorem}

From the above theorem, we conclude that the first step of our recovery procedure (i.e., the Spectral Clustering Algorithm) alone achieves an asymptotically accurate detection whenever this is at all possible, say when $I(\alpha,p)>0$ and $T=\omega(n)$. However, it fails at ensuring asymptotic exact recovery, even in certain cases of $T=\omega(n\ln(n))$, and we cannot guarantee that its recovery rate approaches the fundamental limit identified in \refTheorem{thm:Information_bound}.

\subsection{Cluster Improvement Algorithm}

The second step of our clustering procedure, referred to as the Cluster Improvement Algorithm, aims at sequentially improving the cluster estimates obtained from the Spectral Clustering Algorithm until the recovery rate approaches the limits predicted in \refTheorem{thm:Information_bound}. The pseudo-code of the Cluster Improvement Algorithm is presented in Algorithm \ref{pseudo:The_Cluster_Improvement_Algorithm}.

The Cluster Improvement Algorithm works as follows. Given a cluster assignment $\{ \itr{\hat{\mathcal{V}}_k}{t} \}_{k=1,\ldots,K}$ obtained after the $t$-th iteration, it first calculates the estimates
\begin{gather}
\hat{p}_{a,b} 
= \hat{N}_{ \itr{\hat{\mathcal{V}}_a}{t}, \itr{\hat{\mathcal{V}}_b}{t} } / \hat{N}_{ \itr{\hat{\mathcal{V}}_a}{t}, \mathcal{V} }
\quad
\textrm{for}
\quad
a, b = 1, \ldots, K,
\nonumber \\
\hat{\pi}_k 
= \frac{1}{T} \sum_{ x \in \itr{\hat{\mathcal{V}}_k}{t} } \sum_{ y \in \mathcal{V} } \hat{N}_{x,y}
\quad
\textrm{and}
\quad
\hat{\alpha}_k 
= \frac{ \cardinality{ \itr{\hat{\mathcal{V}}_k}{t} } }{n}
\quad
\textrm{for}
\quad
k = 1, \ldots, K.
\label{eqn:Definitions_for_phat_pihat_alphahat}
\end{gather}
It then initializes $\itr{\hat{\mathcal{V}}_k}{t+1} = \emptyset$ for $k = 1, \ldots, K$, and assigns each state $x = 1, \ldots, n$ to $\itr{\mathcal{V}_{\criticalpoint{c_x}}}{t+1} \gets \itr{\mathcal{V}_{\criticalpoint{c_x}}}{t+1} \cup \{ x \}$, where $\criticalpoint{c_x} \triangleq \arg \max_{c = 1, \ldots, K} \itr{u_x}{t}( c )$, and
\begin{equation}
\itr{u_x}{t}(c) 
\triangleq \Bigl\{ \sum_{k=1}^K \bigl( \hat{N}_{x,\itr{\hat{\mathcal{V}}_k}{t}} \ln{ \hat{p}_{c,k} } + \hat{N}_{\itr{\hat{\mathcal{V}}_k}{t},x} \ln{ \frac{ \hat{p}_{k,c} }{ \hat{\alpha}_c } } \bigr) 
- \frac{T}{n} \cdot \frac{ \hat{\pi}_c }{ \hat{\alpha}_c } \Bigr\}.
\label{eqn:Objective_function_for_the_improvement_algorithm}
\end{equation}
This results in a new cluster assignment $\{ \itr{\hat{\mathcal{V}}_k}{t+1} \}_{k=1,\ldots,K}$. Note that the algorithm works by placing each state in the cluster it most likely belongs to, based on the known structure and the sample path. This can be seen by noting that the objective function in \refEquation{eqn:Objective_function_for_the_improvement_algorithm} is the difference between two log-likelihood functions. 

The second step of our clustering procedure applies the Cluster Improvement Algorithm several times, using as the initial input the cluster assignment $\{ \itr{\hat{\mathcal{V}}_k}{0} \}_{k=1,\ldots,K}$ obtained from the Spectral Clustering Algorithm. We denote by $\itr{{\cal E}}{t}$ the set of misclassified state after the $t$-th iteration of the Clustering Improvement Algorithm. The overall performance of the clustering procedure is quantified in \refTheorem{thm:Conditions_for_improvement_algorithm}.

\begin{algorithm}[!hbtp]
\KwIn{An approximate assignment $\itr{\hat{\mathcal{V}}_1}{t}, \ldots, \itr{\hat{\mathcal{V}}_K}{t}$, and matrix $ \hat{N}$}
\KwOut{A revised assignment $\itr{\hat{\mathcal{V}}_1}{t+1}, \ldots, \itr{\hat{\mathcal{V}}_K}{t+1}$}
\Begin{
$n \gets \mathrm{dim}(\hat{N})$, $\mathcal{V} \gets \{ 1, \ldots, n \}$, $T \gets \sum_{x \in \mathcal{V}} \sum_{y \in \mathcal{V}} \hat{N}_{x,y}$\;
\For{$a \leftarrow 1$ \KwTo $K$}{
$\hat{\pi}_a \leftarrow \hat{N}_{ \itr{\hat{\mathcal{V}}_a}{t}, \mathcal{V} } / T$, $\hat{\alpha}_a \leftarrow \cardinality{ \itr{\hat{\mathcal{V}}_a}{t} } / n$, $\itr{\hat{\mathcal{V}}_a}{t+1} \leftarrow \emptyset$\;
\For{$b \leftarrow 1$ \KwTo $K$}{
$\hat{p}_{a,b} \leftarrow \hat{N}_{ \itr{\hat{\mathcal{V}}_a}{t}, \itr{\hat{\mathcal{V}}_b}{t} } / \hat{N}_{ \itr{\hat{\mathcal{V}}_a}{t}, \mathcal{V} }$\;
}
}
\For{$x \leftarrow 1$ \KwTo $n$}{
$\criticalpoint{c_x} \leftarrow \arg \max_{c = 1, \ldots, K} \Bigl\{ \sum_{k=1}^K \bigl( \hat{N}_{x,\itr{\hat{\mathcal{V}}_k}{t}} \ln{ \hat{p}_{c,k} } + \hat{N}_{\itr{\hat{\mathcal{V}}_k}{t},x} \ln{ \frac{ \hat{p}_{k,c} }{ \hat{\alpha}_c } } \bigr) 
- \frac{T}{n} \cdot \frac{ \hat{\pi}_c }{ \hat{\alpha}_c } \Bigr\}$\;
$\itr{ \hat{\mathcal{V}}_{\criticalpoint{c_{x}}} }{t+1} \gets \itr{ \hat{\mathcal{V}}_{\criticalpoint{c_{x}}} }{t+1} \cup \{ x \}$\;
}
}
\caption{Pseudo-code for the Cluster Improvement Algorithm.}
\label{pseudo:The_Cluster_Improvement_Algorithm}
\end{algorithm}

\begin{theorem}
\label{thm:Conditions_for_improvement_algorithm}
Assume that $T=\omega(n)$ and $I(\alpha,p) > 0$. Then for any $t\ge 1$, after $t$ iterations of the Clustering Improvement Algorithm, initially applied to the output of the Spectral Clustering Algorithm, we have: 
\begin{equation}
{\cardinality{ \itr{\mathcal{E}}{t} }\over n}
=  \bigOPbig{ 
\e{ - t \bigl( \ln{ \frac{T}{n} } - \ln{ \ln{ \frac{T}{n} } } \bigr) } 
+ \e{ - \frac{\alpha_{\min}^2}{720 \eta^3 \alpha_{\max}^2} \frac{T}{n} I(\alpha, p) } 
}.
\end{equation}
\end{theorem}

Observe that for $t = \ln{n}$, the number of misclassified states after applying $t$ times the Clustering Improvement Algorithm is at most of the order $n \exp{ ( - C (T/n) I(\alpha,p) ) }$ with $C\triangleq \frac{\alpha_{\min}^2}{720 \eta^3 \alpha_{\max}^2}$. Up to the constant $C$, this corresponds to the fundamental recovery rate limit identified in \refTheorem{thm:Information_bound}. In particular, our clustering procedure achieves asymptotically exact detection under the following nearly tight sufficient condition: $I(\alpha,p)>0$ and $T-{ n \ln{n} \over C\cdot I(\alpha,p)}=\omega(1)$.

\chapter{Numerical experiments}
\label{sec:Numerical_experiments}

In this section, we numerically assess the performance of our algorithms. We first investigate a simple illustrative example. Then we study the sensitivity of the error rate of the Spectral Clustering Algorithm w.r.t.\ the number of states and the length of the observed trajectory. Finally we show the performance of the Cluster Improvement Algorithm depending on the number of times it is applied to the output of the Spectral Clustering Algorithm. 

\section{An example}
\label{sec:Numerical_experiments__Verification_of_results}

Consider $n =300$ states grouped into three clusters of respective relative sizes $\vect{\alpha} = (0.15, 0.35, 0.5)$, i.e., the cluster sizes are cluster sizes $\cardinality{ {\cal V}_1 } = 48$, $\cardinality{ {\cal V}_2 } = 93$ and $\cardinality{ {\cal V}_3 } = 159$. The transition rates between these clusters are defined by: 
$
p = \bigl( 0.9200, 0.0450, 0.0350; \, \allowbreak 0.0125, 0.8975, 0.0900; \, \allowbreak 0.0175, 0.0200, 0.9625 \bigr)
$.
 
We generate a sample path of the Markov chain of length $T = n^{1.025} \ln{n} \approx 1973$ and calculate $\hat{N}$. A density plot of a typical sample of $\hat{N}$ is shown in \refFigure{fig:Simulation__K3__Observation_images__Nhat_unsorted}. The same density plot is presented in \refFigure{fig:Simulation__K3__Observation_images__Nhat_sorted} where the states have been sorted so as states in the same cluster are neighbors. It is important to note that the algorithms are of course not aware of the structure initially -- sorting states constitutes their objective. Next in \refFigure{fig:Simulation__K3__Observation_images__P}, we show a color representation of the kernel $P$ with sorted rows and columns, in which we can clearly see the groups. Note that the specific colors have no meaning, except for the fact that within the same image two entries with the same color have the same numerical value.

\begin{figure}[!hbtp]
\centering
\begin{subfigure}[t]{0.3\textwidth}
\includegraphics[width=\linewidth]{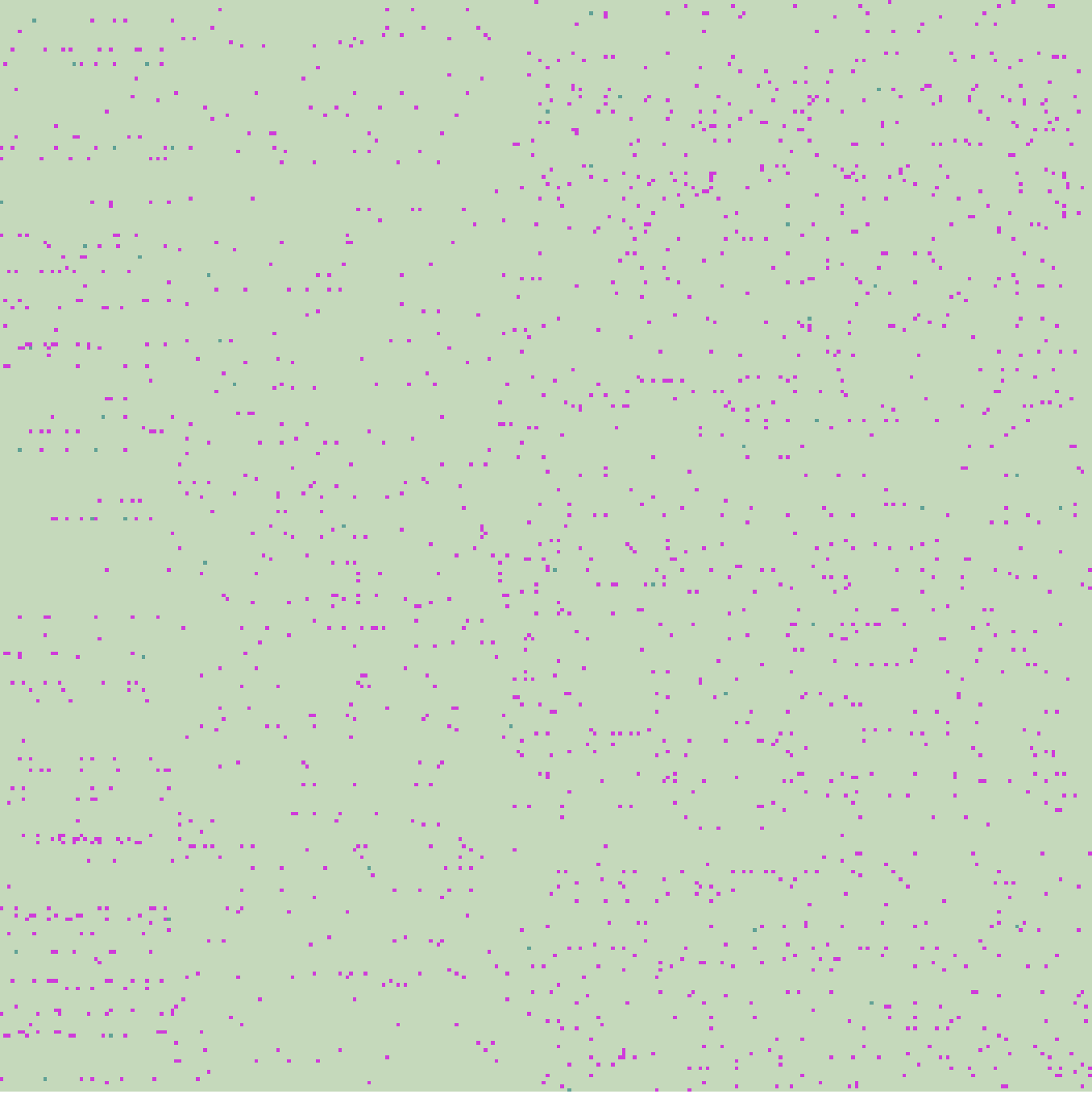}
\caption{$\hat{N}$, unsorted}
\label{fig:Simulation__K3__Observation_images__Nhat_unsorted}
\end{subfigure}
~
\begin{subfigure}[t]{0.3\textwidth}
\includegraphics[width=\linewidth]{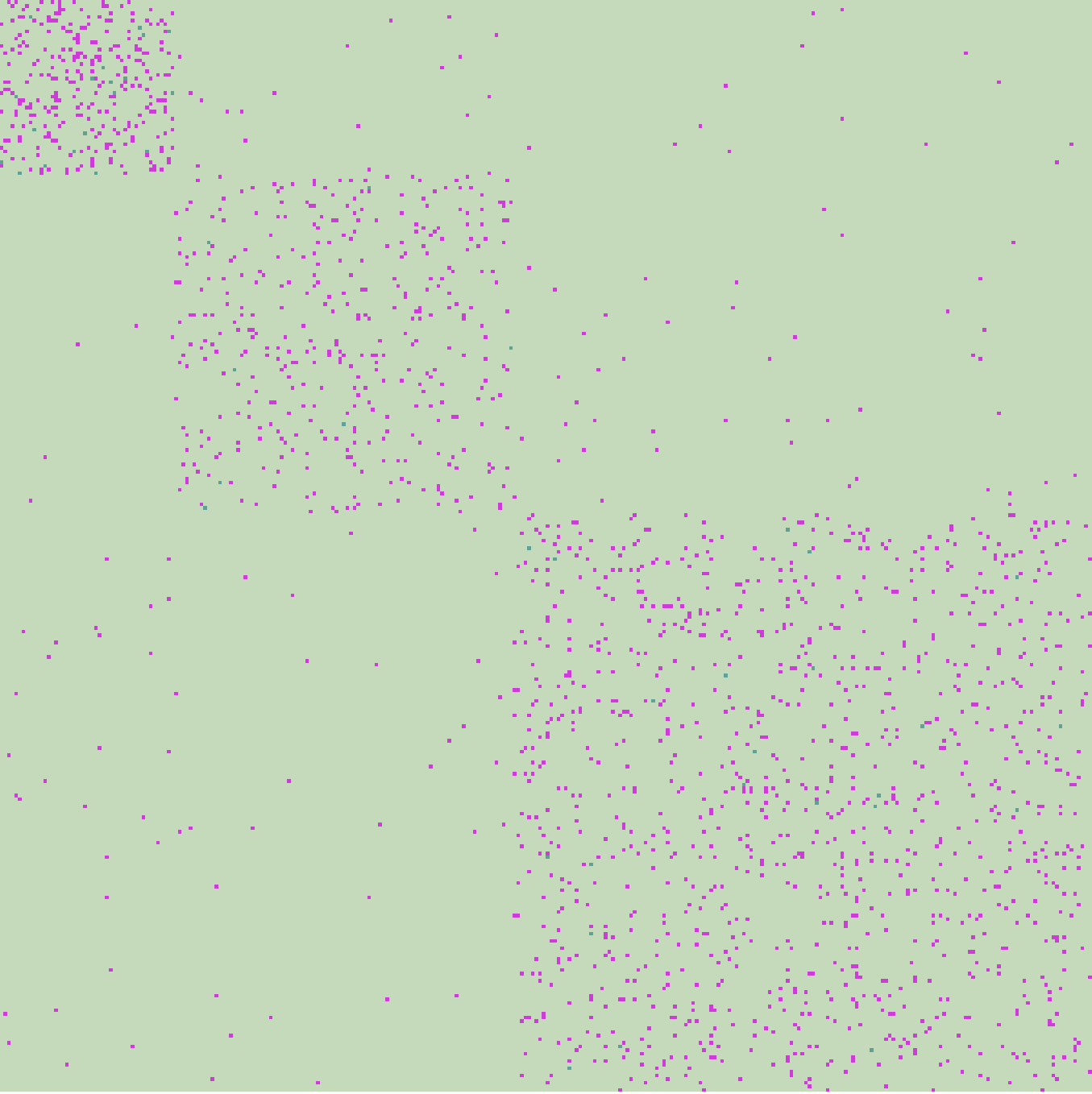}
\caption{$\hat{N}$, sorted}
\label{fig:Simulation__K3__Observation_images__Nhat_sorted}
\end{subfigure}
~
\begin{subfigure}[t]{0.3\textwidth}
\includegraphics[width=\linewidth]{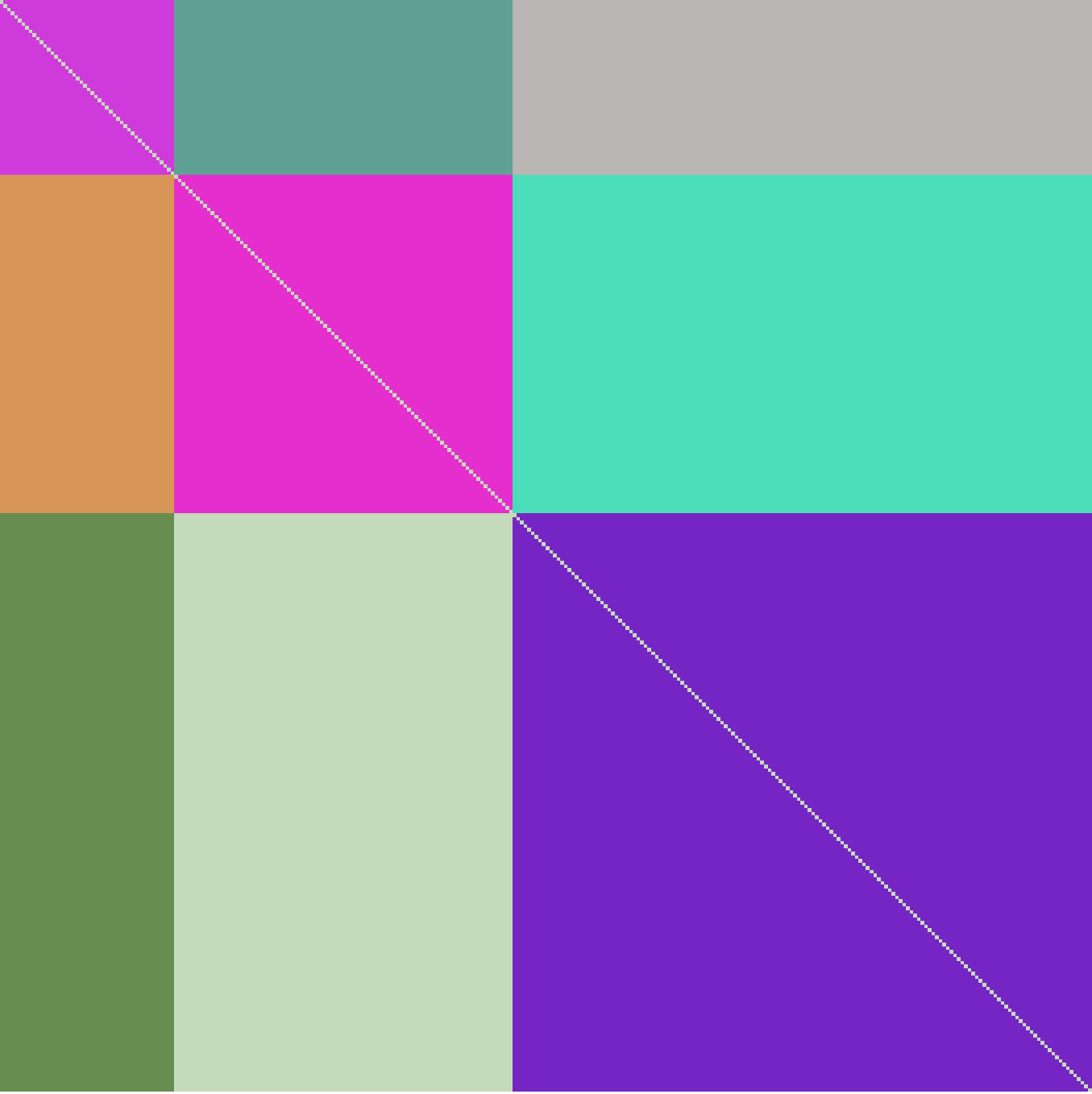}
\caption{$P$, sorted}
\label{fig:Simulation__K3__Observation_images__P}
\end{subfigure}
\caption{A sample path of length $T = n^{1.025} \ln{n} \approx 1973$ was generated, from which $\hat{N}$ is calculated. If we sort the states according to the clusters they belong to, we can see that states within the same cluster share similar dynamics.}
\label{fig:Simulation__K3__Observation_images}
\end{figure}

Next we apply the Spectral Clustering Algorithm. This generates an initial approximate clustering $\itr{ \hat{\mathcal{V}}_1 }{0}, \itr{ \hat{\mathcal{V}}_2 }{0}, \itr{ \hat{\mathcal{V}}_3 }{0}$ of the states. We generate a visual representation of this clustering by constructing $\itr{\hat{P}}{0}$ from the approximate cluster structure and the estimate $\itr{\hat{p}}{0}$. This represents the belief that the algorithm has at this point of the true \gls{BMC} kernel $P$. A color representation of this kernel is shown in \refFigure{fig:Simulation__K3__Clustering_images__Pinitial}. We finally execute the Cluster Improvement Algorithm. After $3$ iterations, it has settled on a final clustering. We generate a color representation of the clustering similar to before, resulting in \refFigure{fig:Simulation__K3__Clustering_images__Pfinal}. The algorithms achieved a $99.7\%$ accuracy: all but one state have been accurately clustered.

\begin{figure}[!hbtp]
\centering
\begin{subfigure}[t]{0.45\textwidth}
\includegraphics[width=\linewidth]{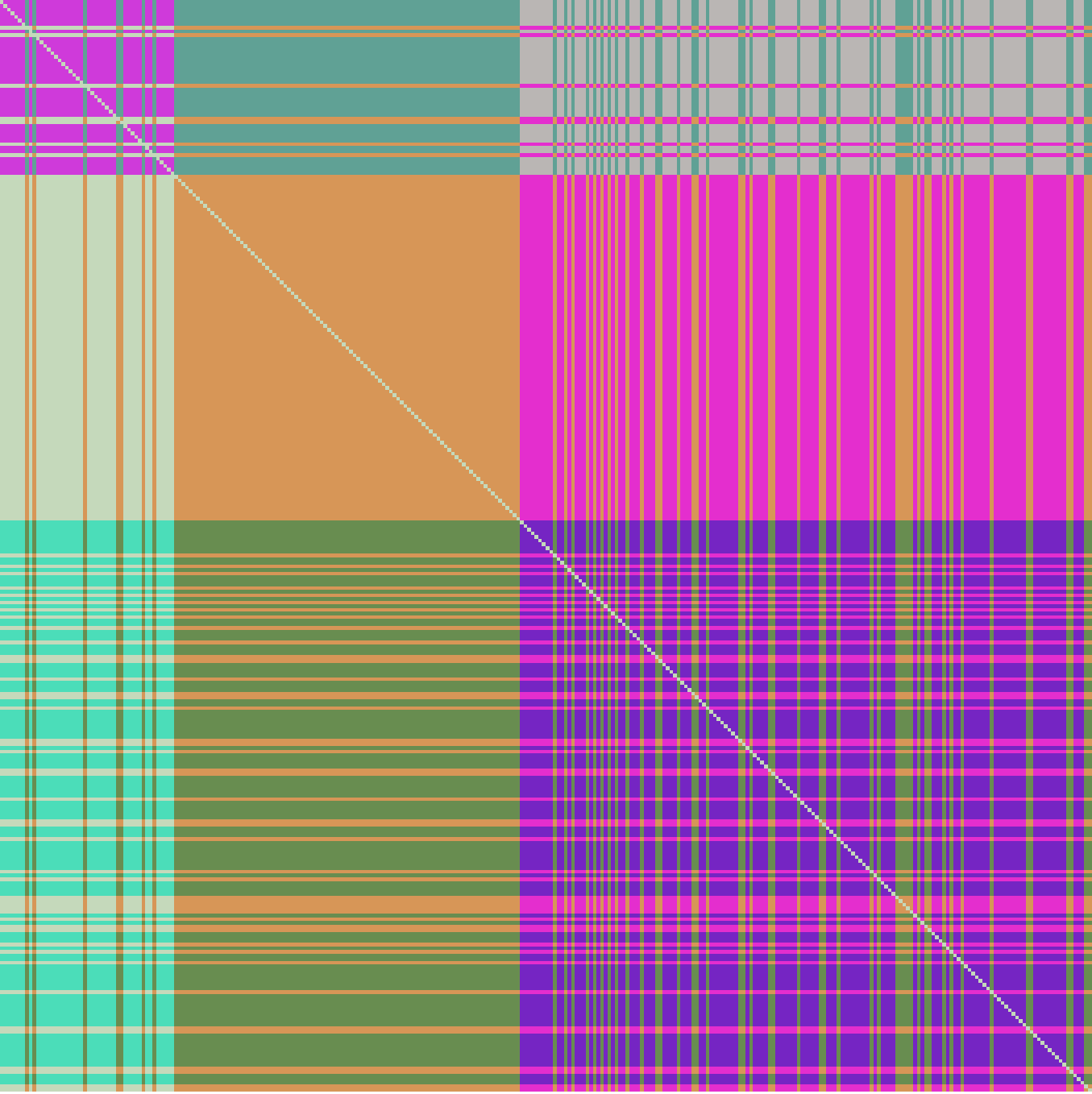}
\caption{Initial clustering.}
\label{fig:Simulation__K3__Clustering_images__Pinitial}
\end{subfigure}
~
\begin{subfigure}[t]{0.45\textwidth}
\includegraphics[width=\linewidth]{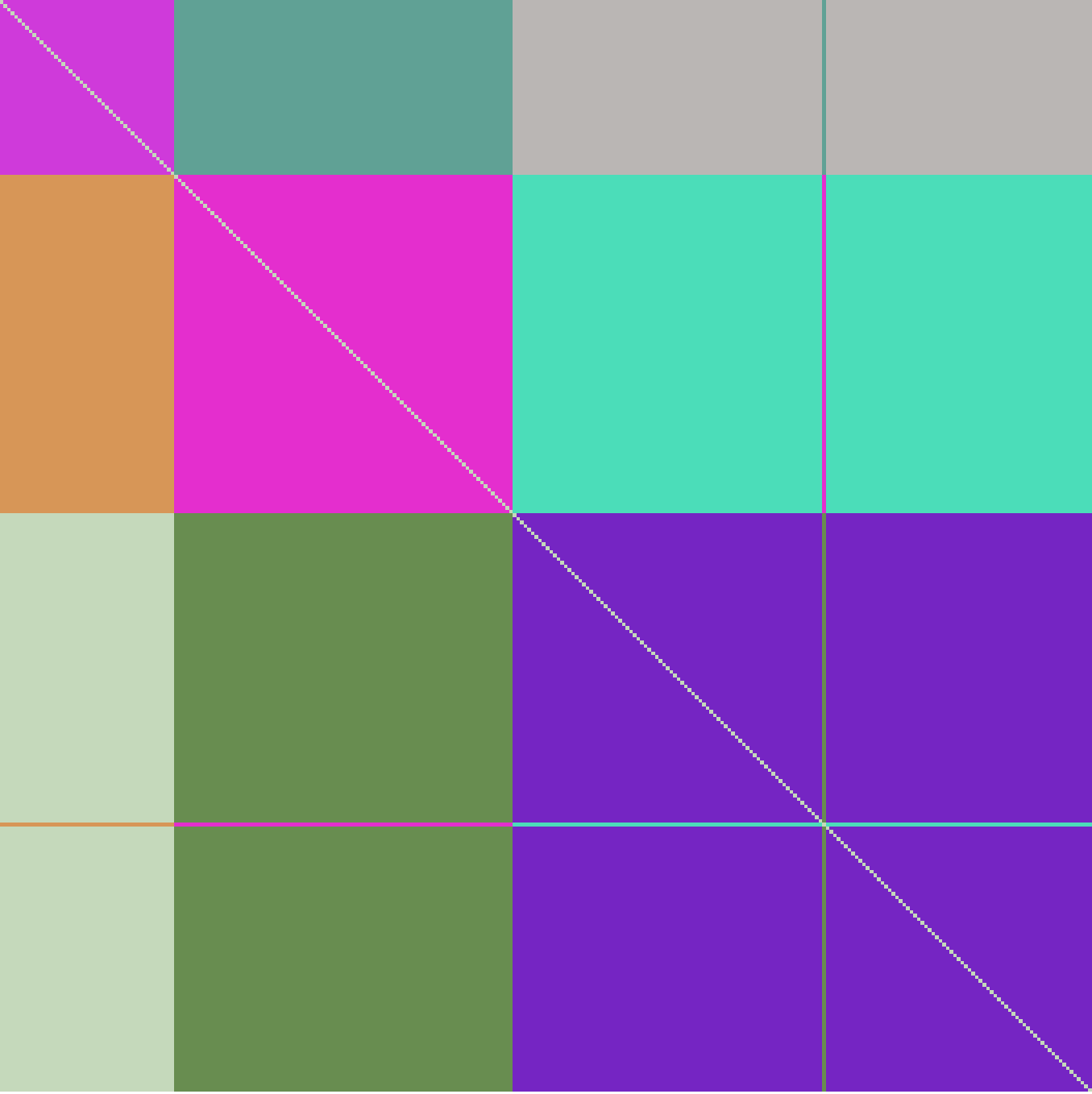}
\caption{Final clustering.}
\label{fig:Simulation__K3__Clustering_images__Pfinal}
\end{subfigure}
\caption{(a) Result after applying the Spectral Clustering Algorithm to the approximation $\hat{N}$. (b) Result after applying 3 iterations of the Cluster Improvement Algorithm. $99.7\%$ of all states were accurately clustered.}
\label{fig:Simulation__K3__Clustering_images}
\end{figure}

\section{Performance sensitivity of the Spectral Clustering Algorithm}
\label{sec:Error_rate_of_the_SVD_Clustering_Algorithm}

In this section, we examine the dependency of the number of misclassified states on the size of the kernel $n$, when we only apply the Spectral Clustering Algorithm. We choose $\vect{\alpha} = (0.15, 0.35, 0.5)$, and set
$
p = \bigl( 0.50, 0.20, 0.30; \, \allowbreak 0.10, 0.70, 0.20; \, \allowbreak 0.35, 0.05, 0.60 \bigr)
$.
These parameters imply that $I(\alpha,p) \approx 0.88 \allowbreak > 0$. This value for $I(\alpha,p)$ is lower than in the case examined in \refSection{sec:Numerical_experiments__Verification_of_results}, so we expect clustering to be more difficult. We have selected a more challenging model so that the initial number of misclassified states will be large and the asymptotics clear.

\refFigure{fig:Error_rate_of_the_SVD_Clustering_Algorithm_as_function_of_n} displays the error rate of the Spectral Clustering algorithm as a function of $n$, for different trajectory lengths $T$. As benchmarks, we include a dashed line that indicates the error rate obtained by assigning states to clusters uniformly at random, i.e., $\probability{ v \not\in {\cal V}_{\sigma(v)} } = \sum_{k=1}^K \probability{ v \not\in {\cal V}_k | \sigma(v) = k } \alpha_k = 1 - 1/K$, as well as a dotted line that indicates the error rate when assigning all states to the smallest cluster, i.e., $1 - \min_{k} \{ \alpha_k \}$. For the $K$-means step of the algorithms, we use Mathematica's default implementation for convenience. Observe that when $T = n \ln{n}$, the fraction of misclassified states hardly decrease as a function of $n$. This is in line with our lower bound. When $T$ gets larger, the error converges to zero faster. Note that the Spectral Clustering Algorithm recovers the clusters exactly when the sample path is sufficiently long.

\begin{figure}[!hbtp]
\centering
\begin{subfigure}[t]{0.3\textwidth}
\centering
\begin{tikzpicture}[scale=0.475]
\begin{axis}[
	ymin = 0, ymax = 1,
	xlabel=$n$, ylabel={Fraction of misclassified vertices},
	label style={font=\Large},
	tick label style={font=\Large}
]
\addplot[mark=none,dotted] plot coordinates { (50,0.85) (750,0.85) };
\addplot[mark=none,dashed] plot coordinates { (50,0.67) (750,0.67) };
\addplot[mark=x,only marks,error bars/.cd, y dir=both, y explicit] plot coordinates {
	(50.,0.3975) += (50.,0.0267251) -= (50.,0.0267251)
	(75.,0.442333) += (75.,0.0200714) -= (75.,0.0200714)
	(100.,0.44775) += (100.,0.0252016) -= (100.,0.0252016)
	(125.,0.4304) += (125.,0.0210899) -= (125.,0.0210899)
	(150.,0.445833) += (150.,0.0188935) -= (150.,0.0188935)
	(175.,0.451571) += (175.,0.0190072) -= (175.,0.0190072)
	(200.,0.44725) += (200.,0.01283) -= (200.,0.01283)
	(225.,0.452889) += (225.,0.012231) -= (225.,0.012231)
	(250.,0.4529) += (250.,0.012029) -= (250.,0.012029)
	(275.,0.450909) += (275.,0.0122723) -= (275.,0.0122723)
	(300.,0.446833) += (300.,0.0114004) -= (300.,0.0114004)
	(325.,0.452) += (325.,0.0136182) -= (325.,0.0136182)
	(350.,0.457929) += (350.,0.00991431) -= (350.,0.00991431)
	(375.,0.451467) += (375.,0.011987) -= (375.,0.011987)
	(400.,0.453938) += (400.,0.0108425) -= (400.,0.0108425)
	(425.,0.448882) += (425.,0.0100245) -= (425.,0.0100245)
	(450.,0.444667) += (450.,0.00997344) -= (450.,0.00997344)
	(475.,0.447474) += (475.,0.00822297) -= (475.,0.00822297)
	(500.,0.44375) += (500.,0.00879182) -= (500.,0.00879182)
	(525.,0.44519) += (525.,0.00788072) -= (525.,0.00788072)
	(550.,0.454591) += (550.,0.00887242) -= (550.,0.00887242)
	(575.,0.447) += (575.,0.00895271) -= (575.,0.00895271)
	(600.,0.448292) += (600.,0.00817759) -= (600.,0.00817759)
	(625.,0.45044) += (625.,0.00926881) -= (625.,0.00926881)
	(650.,0.446308) += (650.,0.00995441) -= (650.,0.00995441)
	(675.,0.441963) += (675.,0.00804083) -= (675.,0.00804083)
	(700.,0.456857) += (700.,0.00940801) -= (700.,0.00940801)
	(725.,0.45031) += (725.,0.00713299) -= (725.,0.00713299)
	(750.,0.436367) += (750.,0.00710811) -= (750.,0.00710811)
};
\end{axis}
\end{tikzpicture}
\caption{$T = n \ln n$}
\end{subfigure}
~
\begin{subfigure}[t]{0.3\textwidth}
\centering
\begin{tikzpicture}[scale=0.475]
\begin{axis}[
	ymin = 0, ymax = 1,
	xlabel=$n$, 
	label style={font=\Large},
	tick label style={font=\Large}
]
\addplot[mark=none,dotted] plot coordinates { (50,0.85) (750,0.85) };
\addplot[mark=none,dashed] plot coordinates { (50,0.67) (750,0.67) };
\addplot[mark=x,only marks,error bars/.cd, y dir=both, y explicit] plot coordinates {
	(50.,0.3335) += (50.,0.0391287) -= (50.,0.0391287)
	(75.,0.371667) += (75.,0.0426922) -= (75.,0.0426922)
	(100.,0.375) += (100.,0.0394903) -= (100.,0.0394903)
	(125.,0.3606) += (125.,0.0402877) -= (125.,0.0402877)
	(150.,0.329833) += (150.,0.050178) -= (150.,0.050178)
	(175.,0.334429) += (175.,0.0500808) -= (175.,0.0500808)
	(200.,0.32475) += (200.,0.0520453) -= (200.,0.0520453)
	(225.,0.299222) += (225.,0.052103) -= (225.,0.052103)
	(250.,0.3119) += (250.,0.0512958) -= (250.,0.0512958)
	(275.,0.319091) += (275.,0.0547952) -= (275.,0.0547952)
	(300.,0.327167) += (300.,0.0517492) -= (300.,0.0517492)
	(325.,0.351154) += (325.,0.049967) -= (325.,0.049967)
	(350.,0.281071) += (350.,0.059145) -= (350.,0.059145)
	(375.,0.3348) += (375.,0.0530062) -= (375.,0.0530062)
	(400.,0.28825) += (400.,0.0576737) -= (400.,0.0576737)
	(425.,0.240588) += (425.,0.0619748) -= (425.,0.0619748)
	(450.,0.300389) += (450.,0.05623) -= (450.,0.05623)
	(475.,0.259263) += (475.,0.0598241) -= (475.,0.0598241)
	(500.,0.2931) += (500.,0.0565529) -= (500.,0.0565529)
	(525.,0.250857) += (525.,0.0600143) -= (525.,0.0600143)
	(550.,0.312818) += (550.,0.0540754) -= (550.,0.0540754)
	(575.,0.275522) += (575.,0.0602867) -= (575.,0.0602867)
	(600.,0.267292) += (600.,0.0618255) -= (600.,0.0618255)
	(625.,0.25768) += (625.,0.0619659) -= (625.,0.0619659)
	(650.,0.234769) += (650.,0.062146) -= (650.,0.062146)
	(675.,0.205556) += (675.,0.062828) -= (675.,0.062828)
	(700.,0.176857) += (700.,0.0617906) -= (700.,0.0617906)
	(725.,0.217103) += (725.,0.0639663) -= (725.,0.0639663)
	(750.,0.2354) += (750.,0.0629804) -= (750.,0.0629804)
};
\end{axis}
\end{tikzpicture}
\caption{$T = n ( \ln n )^{3/2}$}
\end{subfigure}
~
\begin{subfigure}[t]{0.3\textwidth}
\centering
\begin{tikzpicture}[scale=0.475]
\begin{axis}[
	ymode = log, ymin = 10^(-4), ymax = 1,
	xlabel=$n$, 
	label style={font=\Large},
	tick label style={font=\Large}	
]
\addplot[mark=none,dotted] plot coordinates { (50,0.85) (750,0.85) };
\addplot[mark=none,dashed] plot coordinates { (50,0.67) (750,0.67) };
\addplot[mark=x,only marks,error bars/.cd, y dir=both, y explicit] plot coordinates {
	(50.,0.2405) += (50.,0.0590955) -= (50.,0.0590955)
	(75.,0.164333) += (75.,0.0543274) -= (75.,0.0543274)
	(100.,0.07475) += (100.,0.0388619) -= (100.,0.0388619)
	(125.,0.0412) += (125.,0.0315015) -= (125.,0.0315015)
	(150.,0.0175) += (150.,0.00818152) -= (150.,0.00818152)
	(175.,0.0161429) += (175.,0.0113616) -= (175.,0.0113616)
	(200.,0.032) += (200.,0.0303226) -= (200.,0.0303226)
	(225.,0.00444444) += (225.,0.00114755) -= (225.,0.00114755)
	(250.,0.0038) += (250.,0.00194171) -= (250.,0.00194171)
	(275.,0.00309091) += (275.,0.0012637) -= (275.,0.0012637)
	(300.,0.0025) += (300.,0.00118754) -= (300.,0.00118754)
	(325.,0.00207692) += (325.,0.00127318) -= (325.,0.00127318)
	(350.,0.00171429) += (350.,0.000732014) -= (350.,0.000732014)
	(375.,0.00173333) += (375.,0.000906828) -= (375.,0.000906828)
	(400.,0.0010625) += (400.,0.00061726) -= (400.,0.00061726)
	(425.,0.00111765) += (425.,0.000558526) -= (425.,0.000558526)
	(450.,0.00133333) += (450.,0.000569345) -= (450.,0.000569345)
	(475.,0.00110526) += (475.,0.00047645) -= (475.,0.00047645)
	(500.,0.0011) += (500.,0.000623472) -= (500.,0.000623472)
	(525.,0.000904762) += (525.,0.000431074) -= (525.,0.000431074)
	(550.,0.000863636) += (550.,0.0004508) -= (550.,0.0004508)
	(575.,0.000913043) += (575.,0.000465784) -= (575.,0.000465784)
	(600.,0.00104167) += (600.,0.000487559) -= (600.,0.000487559)
	(625.,0.00036) += (625.,0.000213973) -= (625.,0.000213973)
	(650.,0.000730769) += (650.,0.000330285) -= (650.,0.000330285)
	(675.,0.00103704) += (675.,0.000521899) -= (675.,0.000521899)
	(700.,0.000392857) += (700.,0.000289128) -= (700.,0.000289128)
	(725.,0.000344828) += (725.,0.000236853) -= (725.,0.000236853)
	(750.,0.000366667) += (750.,0.000213237) -= (750.,0.000213237)
};
\end{axis}
\end{tikzpicture}
\caption{$T = n ( \ln n )^2$}
\end{subfigure}
\caption{The error rate of the Spectral Clustering Algorithm as function of $n$, for different scalings of $T$. Every point is the average result of $40$ simulations, and the bars indicate a $95\%$-confidence interval.}
\label{fig:Error_rate_of_the_SVD_Clustering_Algorithm_as_function_of_n}
\end{figure}
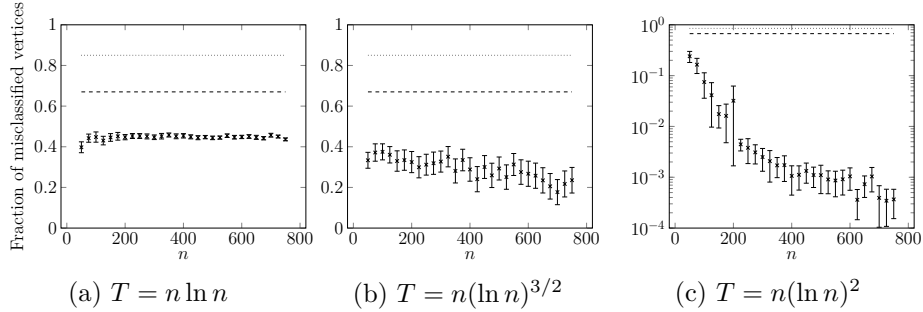

\section{Performance sensitivity of the Cluster Improvement algorithm}

We now examine the number of misclassified states as a function of $T$, when we apply the Spectral Clustering Algorithm and a certain number of iterations of the Cluster Improvement Algorithm. We choose $\vect{\alpha} = (1/3, 1/3, 1/3)$, and set
$
p = \bigl( 0.1, 0.4, 0.5; \, \allowbreak 0.7, 0.1, 0.2; \, \allowbreak 0.6, 0.3, 0.1 \bigr)
$.
Different from the previous experiments, the clusters are now of equal size and the off-diagonal entries of $p$ are dominant. These parameters imply that $I(\alpha,p) \approx 0.27 > 0$, so the cluster algorithms should work, but the situation is again more challenging than in  \refSection{sec:Numerical_experiments__Verification_of_results} and \refSection{sec:Error_rate_of_the_SVD_Clustering_Algorithm}.

\refFigure{fig:Error_after_SVD_and_Cluster_Improvement_as_a_function_of_T} depicts the error after applying the Spectral Clustering Algorithm and subsequently the Cluster Improvement Algorithm up to two times, as a function of $T$. We have chosen both $n,T$ relatively small so that the inputs are significantly noisy. For short sample paths, $T \lesssim 15000$, the data is so noisy that the Cluster Improvement Algorithm does not provide any improvement over the Spectral Clustering Algorithm. For $T \gtrsim 15000$, the Spectral Clustering Algorithm provides a sufficiently accurate initial clustering for the Cluster Improvement Algorithm to work. Because marks $1$ and $2$ overlap in almost all cases, we can conclude that there is (on average, and in the present situation) no benefit in running the Clustering Improvement Algorithm more than once. There is no mark $2$ at $T = 30000$ in this logarithmic plot, because the Cluster Improvement Algorithm achieved $100\%$ accurate detection after $2$ iterations in \emph{all} $200$ simulations.

\section{Critical regime where $T=n\ln(n)$}
\label{sec:Numerical_feasibility_region}

We now study how well our clustering procedure performs in the critical regime $T = n \ln{n}$. Here, we will consider $K = 2$ clusters of equal size: $\alpha_1 = \alpha_2 = \tfrac{1}{2}$. Recall that every such \gls{BMC} can then be completely parameterized by $(p_{1,2}, p_{2,1}) \in (0,1)^2$. Our goal in this section is to numerically evaluate
$
\hat{\mathcal{F}}_1(\varepsilon) 
= \bigl\{ (p_{1,2}, p_{2,1}) \in (0,1)^2 | \expectationWrt{ { \cardinality{ \itr{ \mathcal{E} }{t} } } / {n} }{P} \geq \varepsilon \bigr\}
$
as a proxy estimate to the region $\{ (p_{1,2},p_{2,1}) | I(\alpha,p) \allowbreak \leq 1 \}$. We rasterized $(0,1)^2$ and ran our clustering procedure for $n = 300$ with $t = 6$ improvement steps for each parameter pair $(p_{1,2}, p_{2,1})$. The results are shown in \refFigure{fig:Feasibility_region_of_our_procedure}. Note that the sample mean at each rasterpoint was calculated from $10$ independent runs.

\begin{figure}[!p]
\centering
\begin{tikzpicture}[scale=0.95]
\footnotesize
\begin{axis}[
	width=0.9*\linewidth, height=0.9*0.618\linewidth,
	xmin = 0, xmax = 32500,
	ymode = log, ymin = 10^(-4.5), ymax = 1,
	xlabel=$T$, ylabel={Fraction of misclassified vertices},
	xtick = {0, 2500, 10000, 20000, 30000},
	scaled x ticks=false
]
\addplot[mark=none,dotted] plot coordinates { (2500,0.67) (30000,0.67) };
\addplot[mark=none,dashed] plot coordinates { (2500,0.67) (30000,0.67) };
\addplot[mark=square*, mark size=0.5em, mark options={fill=white},
		x filter/.code={\pgfmathparse{\pgfmathresult+0}},
		nodes near coords={0},
		every node near coord/.style = {font=\scriptsize\sffamily\bfseries, text=black, anchor=center},
		error bars/.cd, y dir=both, y explicit] plot coordinates {
	(2500.,0.383354) += (2500.,0.00885146) -= (2500.,0.00885146)
	(5000.,0.346313) += (5000.,0.0120713) -= (5000.,0.0120713)
	(7500.,0.34275) += (7500.,0.0131798) -= (7500.,0.0131798)
	(10000.,0.305833) += (10000.,0.0157152) -= (10000.,0.0157152)
	(12500.,0.271938) += (12500.,0.0177027) -= (12500.,0.0177027)
	(15000.,0.247708) += (15000.,0.0193041) -= (15000.,0.0193041)
	(17500.,0.194417) += (17500.,0.019671) -= (17500.,0.019671)
	(20000.,0.127063) += (20000.,0.015532) -= (20000.,0.015532)
	(22500.,0.100938) += (22500.,0.0147924) -= (22500.,0.0147924)
	(25000.,0.0747917) += (25000.,0.0111) -= (25000.,0.0111)
	(27500.,0.046875) += (27500.,0.00755765) -= (27500.,0.00755765)
	(30000.,0.0328333) += (30000.,0.00426048) -= (30000.,0.00426048)
};
\addplot[mark=square*, mark size=0.5em, mark options={fill=white},
		x filter/.code={\pgfmathparse{\pgfmathresult+75}},
		nodes near coords={1},
		every node near coord/.style = {font=\scriptsize\sffamily\bfseries, text=black, anchor=center},
		error bars/.cd, y dir=both, y explicit] plot coordinates {
	(2500.,0.318188) += (2500.,0.00810409) -= (2500.,0.00810409)
	(5000.,0.288417) += (5000.,0.00971182) -= (5000.,0.00971182)
	(7500.,0.273729) += (7500.,0.0102353) -= (7500.,0.0102353)
	(10000.,0.230208) += (10000.,0.0145643) -= (10000.,0.0145643)
	(12500.,0.187063) += (12500.,0.0172433) -= (12500.,0.0172433)
	(15000.,0.161625) += (15000.,0.0184153) -= (15000.,0.0184153)
	(17500.,0.103396) += (17500.,0.0180017) -= (17500.,0.0180017)
	(20000.,0.0491667) += (20000.,0.0137575) -= (20000.,0.0137575)
	(22500.,0.0355833) += (22500.,0.0126826) -= (22500.,0.0126826)
	(25000.,0.0177917) += (25000.,0.00865939) -= (25000.,0.00865939)
	(27500.,0.00591667) += (27500.,0.00497494) -= (27500.,0.00497494)
	(30000.,0.000854167) += (30000.,0.000842283) -= (30000.,0.000842283)
};
\addplot[mark=square*, mark size=0.5em, mark options={fill=white},
		x filter/.code={\pgfmathparse{\pgfmathresult-75}},
		nodes near coords={2},
		every node near coord/.style = {font=\scriptsize\sffamily\bfseries, text=black, anchor=center},
		error bars/.cd, y dir=both, y explicit] plot coordinates {
	(2500.,0.3095) += (2500.,0.0143065) -= (2500.,0.0143065)
	(5000.,0.334792) += (5000.,0.0259375) -= (5000.,0.0259375)
	(7500.,0.342688) += (7500.,0.0309807) -= (7500.,0.0309807)
	(10000.,0.283188) += (10000.,0.0356428) -= (10000.,0.0356428)
	(12500.,0.210063) += (12500.,0.03551) -= (12500.,0.03551)
	(15000.,0.182438) += (15000.,0.0354482) -= (15000.,0.0354482)
	(17500.,0.124542) += (17500.,0.0334648) -= (17500.,0.0334648)
	(20000.,0.0525208) += (20000.,0.0225795) -= (20000.,0.0225795)
	(22500.,0.040875) += (22500.,0.0209185) -= (22500.,0.0209185)
	(25000.,0.0156667) += (25000.,0.012663) -= (25000.,0.012663)
	(27500.,0.00633333) += (27500.,0.00887812) -= (27500.,0.00887812)
	(30000.,0.) += (30000.,0.) -= (30000.,0.)
};
\end{axis}
\end{tikzpicture}
\caption{The error after applying the Spectral Clustering Algorithm (mark $0$), and subsequently the Cluster Improvement Algorithm (marks $1, 2$) several times, as a function of $T$. Each number represents the number of improvement steps. Here, $n = 240$. Every point is the average result of $200$ simulations, and the bars indicate a $95\%$-confidence interval. We have minorly offset marks $1,2$ to the right and left for readability, respectively. At $T = 30000$, the Cluster Improvement Algorithm achieved $100\%$ accurate detection after $2$ iterations in \emph{all} $200$ instances.}
\label{fig:Error_after_SVD_and_Cluster_Improvement_as_a_function_of_T}
\end{figure}
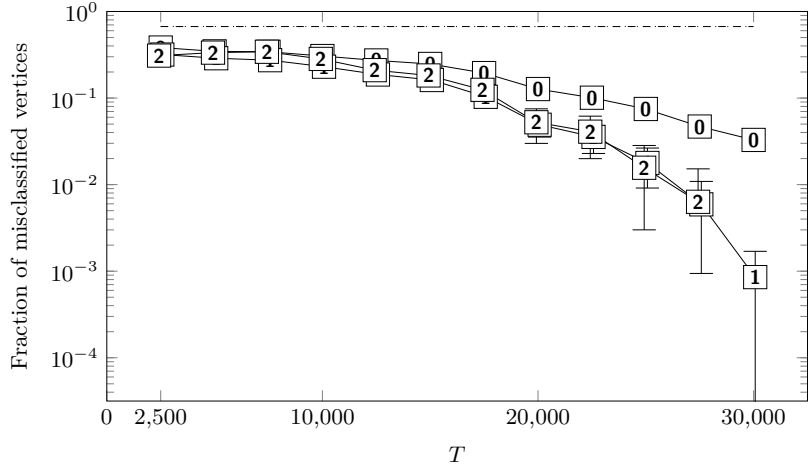

\begin{figure}[!p]
\captionsetup[subfigure]{labelformat=empty}
\centering
\begin{subfigure}[t]{0.315\textwidth}
\centering
\includegraphics[width=\linewidth]{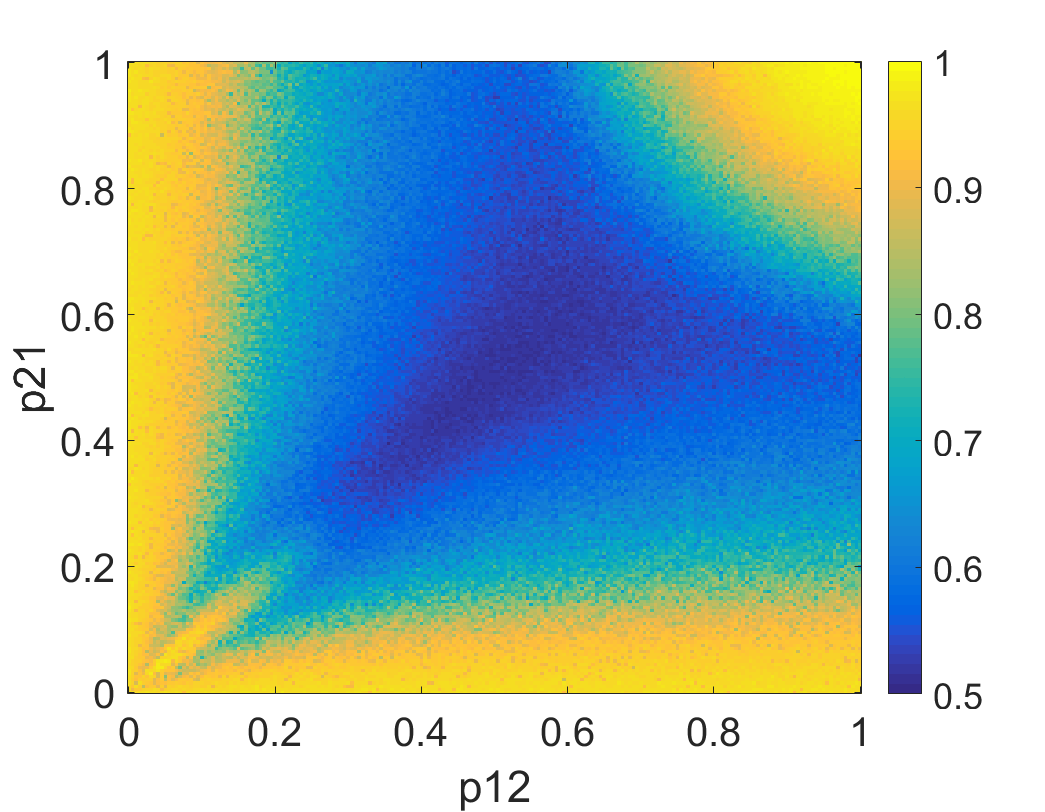}
\caption{After the SCA.}
\end{subfigure}
~
\begin{subfigure}[t]{0.315\textwidth}
\centering
\includegraphics[width=\linewidth]{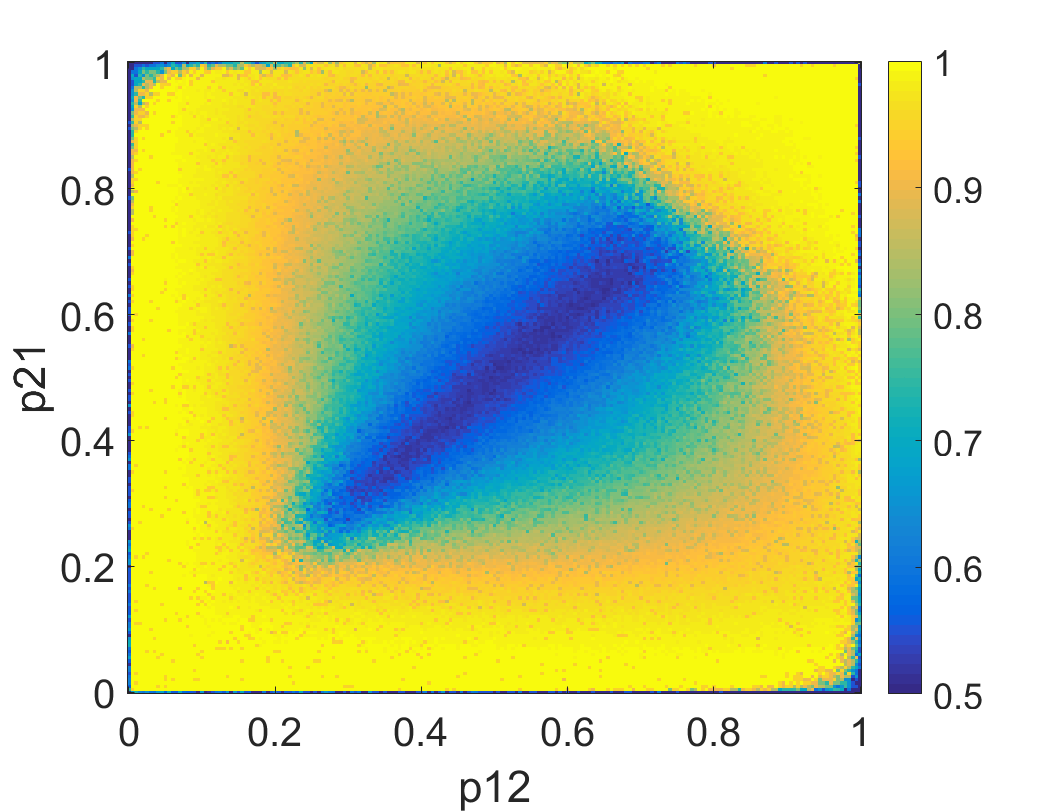}
\caption{After the CIA.}
\end{subfigure}
~
\begin{subfigure}[t]{0.30\textwidth}
\centering
\includegraphics[width=\linewidth]{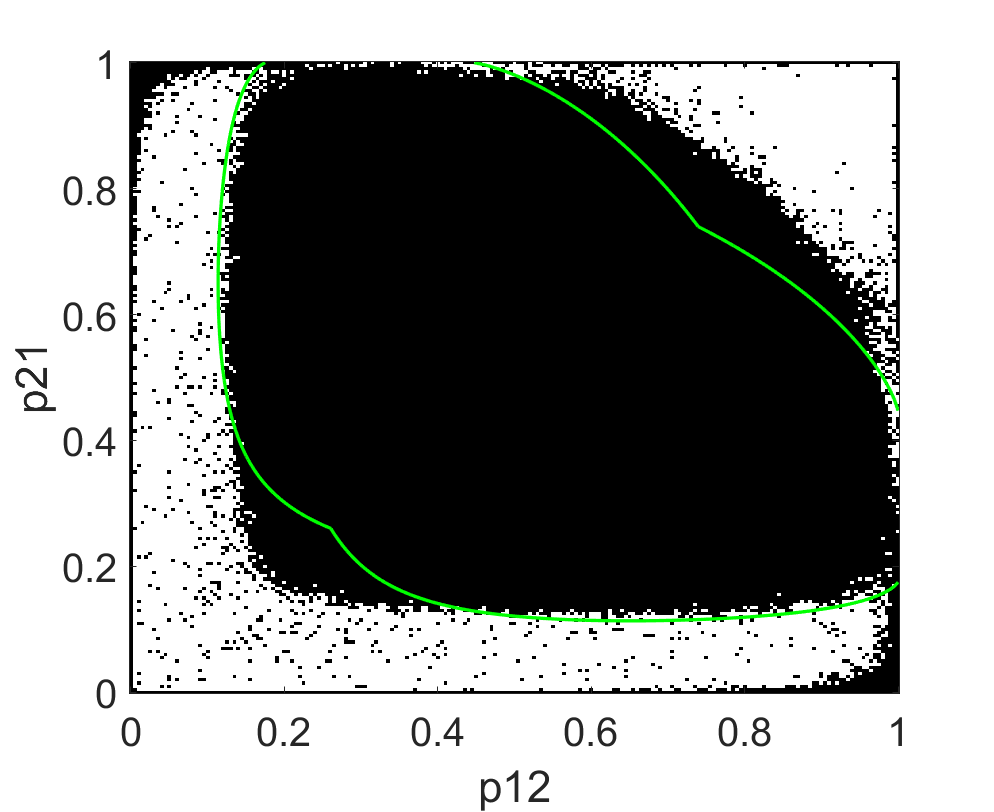}
\caption{$\hat{\mathcal{F}}_1( \varepsilon = 0.027 )$}
\end{subfigure}
\caption{The average proportion of well-classified states for each rasterpoint $(p_{1,2}, p_{2,1}) \in (0,1)^2$ after the Spectral Clustering Algorithm (left) and Cluster Improvement Algorithm (middle), and numerical feasibility region of our clustering procedure (right), both in the critical regime $T = n \ln{n}$. The line outlines the theoretical region $I(\alpha,p) \leq 1$ within which no algorithm exists able to asymptotically recover the clusters exactly.}
\label{fig:Feasibility_region_of_our_procedure}
\end{figure}

\def\Vstar{V^\ast}
\def\sigmaVstar{0}
\let\star\sigmaVstar

\chapter{Lower bounds via change-of-measure arguments}
\label{sec:The_information_bound_and_the_change_of_measure}

In this section, we prove \refTheorem{thm:Information_bound} using {change-of-measure} arguments such as those frequently used to identify information-theoretical limits in online stochastic optimization \cite{lai_asymptotically_1985}. We sketch the proof below, and provide a detailed description of its steps later in this section. We distinguish two cases:

$\mathrm{(i)}$ If $I(\alpha,p)=0$, we prove that there are two clusters whose states cannot be distinguished from any set of observations. Hence, the number of misclassified states grows linearly with $n$, which corresponds to the statement of the theorem when $I(\alpha,p)=0$. 

$\mathrm{(ii)}$ If $I(\alpha,p)>0$, we use a change-of-measure argument. We denote by $\Phi$ the true stochastic model under which the observations $X_0,\ldots,X_T$ are generated, and by $\mathbb{P}_\Phi$ (resp.\ $\mathbb{E}_\Phi$) the corresponding probability measure (resp.\ expectation). The true model is characterized by the transition matrix $P$. In a change-of-measure argument, we pretend that the observations are instead generated under a slightly different stochastic model $\Psi$ (denote by $\mathbb{P}_\Psi$ and $\mathbb{E}_\Psi$ the corresponding probability measure and expectation). The modified stochastic model $\Psi$ is constructed as follows. We pick a state $\Vstar$ randomly from clusters ${\cal V}_a$ and ${\cal V}_b$ with $a\neq b$, and place this state in its own cluster, indexed by 0. We further construct a transition matrix $Q$ depending on $\Vstar$ and slightly different than $P$. $\Psi$ is hence characterized by both $\Vstar$ and the transition matrix $Q$. Next, we introduce the log-likelihood ratio of the observations $X_0,\ldots,X_T$ under the transition matrices $P$ and $Q$:
 \begin{equation}
L
\triangleq \ln{ \frac{ \probabilityWrt{ X_0, X_1, \ldots, X_T }{Q} }{ \probabilityWrt{ X_0, X_1, \ldots, X_T }{P} } },
\label{eqn:Definition_of_log_likelihood_L}
\end{equation}
where
\begin{equation}
\probabilityWrt{ X_0, X_1, \ldots, X_T }{P}
= \prod_{t=1}^T P_{X_{t-1},X_t},
\textrm{ and }
\probabilityWrt{ X_0, X_1, \ldots, X_T }{Q}
= \prod_{t=1}^T Q_{X_{t-1},X_t}.
\label{eqn:Probability_of_a_path_when_kernel_P_is_used}
\end{equation}
Note that $L$ is random because it depends on the observations, but also on $\Vstar$. We then establish that there exist constants $\delta, C>0$ (independent of $n$) such that under any clustering algorithm the expected number of misclassified states satisfies:
\begin{equation}
\expectationWrt{ \cardinality{\mathcal{E}} }{\Phi} 
\geq C n\exp{ \Bigl( - \expectationWrt{L}{\Psi} - \sqrt{ \frac{2}{\delta} } \sqrt{ \varianceWrt{L}{\Psi} } \Bigr) }.
\label{eqn:Change_of_measure_lower_bound}
\end{equation}  
Finally for particular choices of $a$, $b$ and the transition matrix $Q$, the analysis of $\expectationWrt{L}{\Psi}$ and $ \varianceWrt{L}{\Psi}$ reveals that: for $T=\omega(n)$,
$
\expectationWrt{L}{\Psi} 
\leq (T/n) I(\alpha,p) + \smallO{ {T} / {n} }
$
and
$
\varianceWrt{L}{\Psi} = \smallO{ {T^2} / {n^2} }
$.
Combining this with \refEquation{eqn:Change_of_measure_lower_bound} completes the proof of \refTheorem{thm:Information_bound}.

\section{Necessary and sufficient condition for $I(\alpha,p)=0$}

\begin{lemma}
\label{lem:noinformation}
For any \gls{BMC}, $I(\alpha,p) = 0$ if and only if there exists $i \neq j$ such that $p_{i,c} = p_{j,c}$ and $p_{c,i} / \alpha_i = p_{c,j} / \alpha_j$ for all $c\in \{ 1,\dots, K\}$.
\end{lemma}

\refLemma{lem:noinformation} implies that observing a trajectory of the Markov chain does not provide any information allowing us to distinguish between states in ${\cal V}_i$ and ${\cal V}_j$. As a consequence, any clustering algorithm will misclassify at least a fixed proportion of states. The proof of Lemma \ref{lem:noinformation} is provided in \textsection \ref{ch5:proofs}.

\section{Change-of-measure argument}
\label{sec:Change_of_measure_argument}

In the remainder of this section, we assume that $I(\alpha,p)>0$. The argument consists in considering that the observations $X_0, \allowbreak \ldots, \allowbreak X_T$ are generated by a slightly different stochastic model than the true model defined by the clusters and the transition matrix $P$. As mentioned above, the modified model is obtained by randomly choosing a state $\Vstar$ from two clusters ${\cal V}_a$ and ${\cal V}_b$, and by constructing a transition matrix $Q$ depending on $\Vstar$ that is slightly different from $P$. Under the new model, $\Vstar$ is placed in its own cluster indexed by $0$. The matrix $Q$ is parametrized by $q$, a $2(K+1)$-dimensional vector in 
\begin{equation}
\mathcal{Q} 
\triangleq \bigl\{ ( q_{k,0}, q_{0,k} )_{k=0,\ldots,K} \in (0,\infty)
\big| 
q_{0,0} = 0, \sum_{l=1}^K q_{0,l} = 1
\bigr\},
\end{equation}
where $q_{k,0}$ (resp. $q_{0,k}$) will define the probability to move from a state in cluster ${\cal V}_k$ to $\Vstar$ (resp. from state $\Vstar$ to a state in ${\cal V}_k$) under the new model $\Psi$. We also define:
\begin{equation}
q_{k,l} 
\triangleq p_{k,l} - \frac{q_{k,\star}}{ K n }
\quad
\textrm{for}
\quad
k, l = 1, \ldots, K,
\label{eqn:Coupling_of_the_parameters_within_q_to_p}
\end{equation}
and assume that $n > \lceil \max_{ k,l = 1, \ldots,K } \{ q_{k,\star} / ( K p_{k,l} ) \} \rceil$ so that the entries of \refEquation{eqn:Coupling_of_the_parameters_within_q_to_p} are strictly positive. Note that the collection $\{ q_{k,l} \}_{ k,l \in \{ \star, 1, \ldots, K \} }$ does \emph{not} constitute a stochastic matrix, but does resemble the transition matrix $p$ for sufficiently large $n$. We are now ready to define $Q$ component-wise:
\begin{equation}
Q_{x,y} 
\triangleq \frac{ q_{ \omega(x), \omega(y) } \indicator{ x \neq y } }{ \cardinality{ \mathcal{W}_{ \omega(y) } } - \indicator{ \omega(x) = \omega(y) } },
\quad
Q_{x,\Vstar} 
\triangleq \frac{q_{\omega(x),\star}}{ n }
\quad
\textrm{for}
\quad
x \in \mathcal{V}, y \neq \Vstar,
\label{eqn:Definition_of_Qs_entries}
\end{equation}
where
\begin{equation}
\omega(x)
\triangleq
\begin{cases}
\star & \textrm{if } x = \Vstar, \\
\sigma(x) & \textrm{if } x \neq \Vstar, \\
\end{cases}
\quad
\textrm{and}
\quad
\mathcal{W}_k 
\triangleq 
\begin{cases}
\{ \Vstar \} & \textrm{if } k = \star, \\
\mathcal{V}_k \backslash \{ \Vstar \} & \textrm{if } k = 1, \ldots, K, \\ 
\end{cases}
\end{equation}
for notational convenience. This has the added benefit of giving \refEquation{eqn:Definition_of_Qs_entries} a similar form as \refEquation{eqn:Definition_of_P}. 

$Q$ is by construction a stochastic matrix (see \refAppendixSection{sec:Q_is_a_stochastic_matrix}). Note furthermore that because $Q$ is constructed from $P$, which by assumption describes an irreducible Markov chain, and because the entries $\{ q_{k,\star}, q_{\star,k} \}_{k=1,\ldots,K}$ are all strictly positive, $Q$ also describes an irreducible Markov chain. Next, we highlight other basic properties of this chain, and provide an example of matrix $Q$.

\paragraph{Equilibrium distribution}

Let $\Pi^{(Q)}$ denote the equilibrium distribution of a Markov chain with transition matrix $Q$, i.e., the solution to $\transpose{ \Pi^{(Q)} } Q = \transpose{ \Pi^{(Q)} }$. By symmetry of states in the same cluster $\Pi_x^{(Q)} = \Pi_y^{(Q)} \triangleq \bar{\Pi}^{(Q)}_k$ for any two states $x,y \in \mathcal{W}_k$ and all $k \in \{ \star, 1, \ldots, K \}$. Define
\begin{equation}
\itr{\gamma_{k}}{0}
\triangleq \lim_{n \to \infty} \sum_{x \in \mathcal{W}_k} \Pi_x^{(Q)}
= \lim_{n \to \infty} \cardinality{ \mathcal{W}_k } \bar{\Pi}_{k}^{(Q)} 
\quad
\textrm{for}
\quad
k \in \{ \star, 1,\ldots,K \}.
\label{eqn:Definition_of_gamma_k_0}
\end{equation}
We can expect $\itr{\gamma_{\star}}{0}$ to be zero, because by our construction of $Q$ we can expect that $\Pi_{x}^{(Q)} = \bigO{1/n}$ for all $x \in \mathcal{V}$ (including $\Vstar$). We therefore also define its higher order statistic
$
\itr{\gamma_{\star}}{1}
\triangleq \lim_{n \to \infty} n \Pi_{\Vstar}^{(Q)}
$.
The following proposition relates these scaled quantities to the parameters of our \gls{BMC} $\process{ X_t }{t \geq 0}$. The proof is deferred to \refAppendixSection{sec:Proof_of_Equilibrium_behavior_of_piQ}, and relies on several applications of the balance equations and a subsequent asymptotic analysis.

\begin{proposition} 
\label{prop:Equilibrium_behavior_of_piQ}
For $k = 1, \ldots, K$, $\itr{\gamma_k}{0} = \pi_k$. Furthermore $\itr{\gamma_{\star}}{0} = 0$ and $\itr{\gamma_{\star}}{1} = \sum_{k=1}^K \pi_k q_{k,\star}$.
\end{proposition}

\paragraph{Mixing time} 

It is important to note that \refProposition{prop:Mixing_times_of_Markov_chains_with_transition_matrices_P_and_Q} holds for a Markov chain with $Q$ as its transition matrix as well. This follows when applying the exact same proof.

\paragraph{Example} 

It is illustrative to explicitly write down at least one example kernel $Q$. For $K = 3$, $\alpha = ( 2/10, 3/10, 5/10 )$ and $n = 10$, $\Vstar = 7$, it is given by
\begin{align}
Q
&
=
\left(
\begin{array}{cc;{2pt/2pt}ccc;{2pt/2pt}c>{\columncolor{black!10}}cccc}
0 & p_{1,1} & \frac{p_{1,2}}{3} & \frac{p_{1,2}}{3} & \frac{p_{1,2}}{3} & \frac{p_{1,3}}{4} & \frac{q_{1,\star}}{10} & \frac{p_{1,3}}{4} & \frac{p_{1,3}}{4} & \frac{p_{1,3}}{4} \\
p_{1,1} & 0 & \frac{p_{1,2}}{3} & \frac{p_{1,2}}{3} & \frac{p_{1,2}}{3} & \frac{p_{1,3}}{4} & \frac{q_{1,\star}}{10} & \frac{p_{1,3}}{4} & \frac{p_{1,3}}{4} & \frac{p_{1,3}}{4} \\ \hdashline[2pt/2pt]
\frac{p_{2,1}}{2} & \frac{p_{2,1}}{2} & 0 & \frac{p_{2,2}}{2} & \frac{p_{2,2}}{2} & \frac{p_{2,3}}{4} & \frac{q_{2,\star}}{10} & \frac{p_{2,3}}{4} & \frac{p_{2,3}}{4} & \frac{p_{2,3}}{4} \\
\frac{p_{2,1}}{2} & \frac{p_{2,1}}{2} & \frac{p_{2,2}}{2} & 0 & \frac{p_{2,2}}{2} & \frac{p_{2,3}}{4} & \frac{q_{2,\star}}{10} & \frac{p_{2,3}}{4} & \frac{p_{2,3}}{4} & \frac{p_{2,3}}{4} \\
\frac{p_{2,1}}{2} & \frac{p_{2,1}}{2} & \frac{p_{2,2}}{2} & \frac{p_{2,2}}{2} & 0 & \frac{p_{2,3}}{4} & \frac{q_{2,\star}}{10} & \frac{p_{2,3}}{4} & \frac{p_{2,3}}{4} & \frac{p_{2,3}}{4} \\  \hdashline[2pt/2pt]
\frac{p_{3,1}}{2} & \frac{p_{3,1}}{2} & \frac{p_{3,2}}{3} & \frac{p_{3,2}}{3} & \frac{p_{3,2}}{3} & 0 & \frac{q_{3,\star}}{10} & \frac{p_{3,3}}{3} & \frac{p_{3,3}}{3} & \frac{p_{3,3}}{3} \\
\rowcolor{black!10}
\frac{q_{\star,1}}{2} & \frac{q_{\star,1}}{2} & \frac{q_{\star,2}}{3} & \frac{q_{\star,2}}{3} & \frac{q_{\star,2}}{3} & \frac{q_{\star,3}}{4} & 0 & \frac{q_{\star,3}}{4} & \frac{q_{\star,3}}{4} & \frac{q_{\star,3}}{4} \\
\frac{p_{3,1}}{2} & \frac{p_{3,1}}{2} & \frac{p_{3,2}}{3} & \frac{p_{3,2}}{3} & \frac{p_{3,2}}{3} & \frac{p_{3,3}}{3} & \frac{q_{3,\star}}{10} & 0 & \frac{p_{3,3}}{3} & \frac{p_{3,3}}{3} \\
\frac{p_{3,1}}{2} & \frac{p_{3,1}}{2} & \frac{p_{3,2}}{3} & \frac{p_{3,2}}{3} & \frac{p_{3,2}}{3} & \frac{p_{3,3}}{3} & \frac{q_{3,\star}}{10} & \frac{p_{3,3}}{3} & 0 & \frac{p_{3,3}}{3} \\
\frac{p_{3,1}}{2} & \frac{p_{3,1}}{2} & \frac{p_{3,2}}{3} & \frac{p_{3,2}}{3} & \frac{p_{3,2}}{3} & \frac{p_{3,3}}{3} & \frac{q_{3,\star}}{10} & \frac{p_{3,3}}{3} & \frac{p_{3,3}}{3} & 0 \\
\end{array}
\right)
\nonumber \\ &
- \frac{1}{3 \cdot 10}
\left(
\begin{array}{cc;{2pt/2pt}ccc;{2pt/2pt}ccccc}
0 & q_{1,\star} & \frac{q_{1,\star}}{3} & \frac{q_{1,\star}}{3} & \frac{q_{1,\star}}{3} & \frac{q_{1,\star}}{4} & 0 & \frac{q_{1,\star}}{4} & \frac{q_{1,\star}}{4} & \frac{q_{1,\star}}{4} \\ 
q_{1,\star} & 0 & \frac{q_{1,\star}}{3} & \frac{q_{1,\star}}{3} & \frac{q_{1,\star}}{3} & \frac{q_{1,\star}}{4} & 0 & \frac{q_{1,\star}}{4} & \frac{q_{1,\star}}{4} & \frac{q_{1,\star}}{4} \\ \hdashline[2pt/2pt]
\frac{q_{2,\star}}{2} & \frac{q_{2,\star}}{2} & 0 & \frac{q_{2,\star}}{2} & \frac{q_{2,\star}}{2} & \frac{q_{2,\star}}{4} & 0 & \frac{q_{2,\star}}{4} & \frac{q_{2,\star}}{4} & \frac{q_{2,\star}}{4} \\ 
\frac{q_{2,\star}}{2} & \frac{q_{2,\star}}{2} & \frac{q_{2,\star}}{2} & 0 & \frac{q_{2,\star}}{2} & \frac{q_{2,\star}}{4} & 0 & \frac{q_{2,\star}}{4} & \frac{q_{2,\star}}{4} & \frac{q_{2,\star}}{4} \\ 
\frac{q_{2,\star}}{2} & \frac{q_{2,\star}}{2} & \frac{q_{2,\star}}{2} & \frac{q_{2,\star}}{2} & 0 & \frac{q_{2,\star}}{4} & 0 & \frac{q_{2,\star}}{4} & \frac{q_{2,\star}}{4} & \frac{q_{2,\star}}{4} \\ \hdashline[2pt/2pt]
\frac{q_{3,\star}}{2} & \frac{q_{3,\star}}{2} & \frac{q_{3,\star}}{3} & \frac{q_{3,\star}}{3} & \frac{q_{3,\star}}{3} & 0 & 0 & \frac{q_{3,\star}}{3} & \frac{q_{3,\star}}{3} & \frac{q_{3,\star}}{3} \\ 
0 & 0 & 0 & 0 & 0 & 0 & 0 & 0 & 0 & 0 \\
\frac{q_{3,\star}}{2} & \frac{q_{3,\star}}{2} & \frac{q_{3,\star}}{3} & \frac{q_{3,\star}}{3} & \frac{q_{3,\star}}{3} & \frac{q_{3,\star}}{3} & 0 & 0 & \frac{q_{3,\star}}{3} & \frac{q_{3,\star}}{3} \\ 
\frac{q_{3,\star}}{2} & \frac{q_{3,\star}}{2} & \frac{q_{3,\star}}{3} & \frac{q_{3,\star}}{3} & \frac{q_{3,\star}}{3} & \frac{q_{3,\star}}{3} & 0 & \frac{q_{3,\star}}{3} & 0 & \frac{q_{3,\star}}{3} \\ 
\frac{q_{3,\star}}{2} & \frac{q_{3,\star}}{2} & \frac{q_{3,\star}}{3} & \frac{q_{3,\star}}{3} & \frac{q_{3,\star}}{3} & \frac{q_{3,\star}}{3} & 0 & \frac{q_{3,\star}}{3} & \frac{q_{3,\star}}{3} & 0 \\ 
\end{array}
\right).
\label{eqn:Example_Q_matrix}
\end{align}
Here, we have indicated the original cluster structure in dashed lines, and we have colored the row and column corresponding to the modified cluster behavior of state $\Vstar$. Comparing \refEquation{eqn:Example_Q_matrix} to \refEquation{eqn:Example_P_matrix} helps understanding how $Q$ is constructed and how $Q$ compares to $P$. Note in particular the minor changes to the normalizations of all entries.

\section{Log-likelihood ratio and its connection to the error rate}

To describe how the log-likelihood ratio $L$ relates to the error rate, we first introduce the following notations. For any $q\in {\cal Q}$, and any cluster index $a$, we define:
\begin{equation}
I_a(q||p) 
\triangleq \sum_{k=1}^K \Bigl( \bigl( \sum_{l=1}^K \pi_l q_{l,0} \bigr) q_{0,k} \ln{ \frac{ q_{0,k} }{ p_{a,k} } } + \pi_k q_{k,0} \ln{ \frac{ q_{k,0} \alpha_a }{ p_{k,a} } } \Bigr) + \Bigl( \frac{ \pi_a }{ \alpha_a } - \sum_{k=1}^K \pi_k q_{k,0} \Bigr).
\label{eqn:Definition_of_Icqp}
\end{equation}
As shown later in this section, $I_a(q||p)$ is the leading order of the expected log-likelihood ratio $L$ under $\Psi$ and given that $\Vstar$ is extracted from ${\cal V}_a$. We further define for any pair of cluster indexes $a$, $b$: 
$
\mathcal{Q}(a,b) 
\triangleq \bigl\{ q \in \mathcal{Q} \big| I_a(q||p) = I_b(q||p) \bigr\}.
$ 
These sets are not empty as stated below.

\begin{lemma}
\label{lem:Existence_of_point_qbar}
For any two cluster indexes $a \neq b$, $\mathcal{Q}(a,b)\neq\emptyset$. 
\end{lemma}

Next, \refProposition{prop:Lower_bound_on_the_number_of_misclassifications_for_a_class_of_Q_transition_matrices} states that choosing $q\in \mathcal{Q}(a,b)$, one may obtain lower bounds on the error rate by analyzing the log-likelihood ratio $L$.

\begin{proposition}
\label{prop:Lower_bound_on_the_number_of_misclassifications_for_a_class_of_Q_transition_matrices}
Assume that $\Vstar$ is chosen uniformly at random from two different clusters ${\cal V}_a$ and ${\cal V}_b$, that $Q$ is constructed from $q\in {\cal Q}(a,b)$, \jaron{and that there exists a $(\varepsilon,1)$-locally good clustering algorithm at $(\alpha,p)$. Then:} \\
(i) There exists a constant $\delta>0$ independent of $n$ s.t.\ $\probabilityWrt{ \Vstar \in \mathcal{E} }{\Psi} \geq \delta > 0$.\\
(ii) There exists a constant $C>0$ independent of $n$ such that
\begin{equation}
\expectationWrt{ \cardinality{\mathcal{E}} }{\Phi} 
\geq C n \exp{ \Bigl( - \expectationWrt{L}{\Psi} - \sqrt{ \frac{2}{\delta} } \sqrt{ \varianceWrt{L}{\Psi} } \Bigr) }.
\label{eqn:Change_of_measure_lower_bound}
\end{equation}
\end{proposition}

\section{Analysis of the log-likelihood ratio}
\label{sec:Asymptotic_negligibilit_of_VarLwrtQ}

\refProposition{prop:Leading_order_behavior_of_expectation_of_L_wrt_Q} characterizes the leading term of $\expectationWrt{L}{\Psi}$ when the cluster from which $\Vstar$ originates is fixed. 

\begin{proposition}
\label{prop:Leading_order_behavior_of_expectation_of_L_wrt_Q}
For any given cluster index $a$, and $q \in \mathcal{Q}$, 
$
\expectationWrt{ L | \sigma(\Vstar) = a }{\Psi} 
= (T/n) I_a(q||p) + \smallO{ T/n }
$.
\end{proposition}

With \refProposition{prop:Variance_is_small}, we establish that $\sqrt{ \varianceWrt{L|\sigma(\Vstar)=a}{\Psi} }$ is asymptotically negligible when compared to $\expectationWrt{ L | \sigma(\Vstar) }{\Psi}$ whenever $T = \omega(n)$. The proof relies on relating the covariances \emph{between} the $T$ transitions in the sample path $\{ X_0,X_1, X_2, \ldots, X_T \}$ to the mixing time of the underlying Markov chain. \refProposition{prop:Mixing_times_of_Markov_chains_with_transition_matrices_P_and_Q} then allows us to bound variance. 

\begin{proposition}
\label{prop:Variance_is_small}
For any given cluster index $a$, and $q \in \mathcal{Q}$, if $T = \omega(1)$, then 
$
\varianceWrt{ L | \sigma(\Vstar)=a }{\Psi} 
= \bigO{ T/n }
$.
As a consequence if $T = \omega(n)$, then $\varianceWrt{L | \sigma(\Vstar)=a }{\Psi} = \smallO{ T^2 / n^2 }$.
\end{proposition}

Combining Propositions \ref{prop:Lower_bound_on_the_number_of_misclassifications_for_a_class_of_Q_transition_matrices}, \ref{prop:Leading_order_behavior_of_expectation_of_L_wrt_Q}, and \ref{prop:Variance_is_small}, we obtain \refCorollary{lem:Specification_of_lower_bound_before_bound_optimization}.

\begin{corollary}
\label{lem:Specification_of_lower_bound_before_bound_optimization}
There exists a constant $C>0$ such that for any two different cluster indexes $a$ and $b$, and for any $q\in {\cal Q}(a,b)$, if $T = \omega(n)$, then   
$
\expectationWrt{ \cardinality{\mathcal{E}} }{\Psi} 
\geq C n \exp{} \bigl( - (T/n) I_{a}({q}||p) + \smallO{ T/n } \bigr)
$.
\end{corollary}

By varying $a$, $b$, and $q\in  \mathcal{Q}(a,b)$ in \refCorollary{lem:Specification_of_lower_bound_before_bound_optimization}, we conclude that if $T=\omega(n)$:
$
\expectationWrt{ \cardinality{\mathcal{E}} }{\Phi} 
\geq C n \exp{} \big( - (T/n) J(\alpha,p) + \smallO{ T/n } \bigr)
$,
where $J(\alpha,p) \triangleq \min_{a\neq b} \min_{ q \in \mathcal{Q}(a,b) } I_a(q||p)$. 

\section{Connecting $I(\alpha,p)$ and $J(\alpha,p)$}

Finally, to complete the proof of \refTheorem{thm:Information_bound}, we relate $I(\alpha,p)$ to $J(\alpha,p)$ by proving that:
\begin{lemma}
\label{lem:Jalphap_is_bounded_by_Ialphap}
For any BMC, we have $J(\alpha,p) \leq I(\alpha,p)$. 
\end{lemma}


\chapter{Performance of the Spectral Clustering Algorithm}
\label{sec:The_SVD_clustering_algorithm}

This section is devoted to the proof of \refTheorem{thm:Upper_bound_on_SVDCAs_performance}. The Spectral Clustering Algorithm relies on a spectral decomposition of the trimmed matrix $\hat{N}_\Gamma$, gathering the empirical transition rates between states. The proof of \refTheorem{thm:Upper_bound_on_SVDCAs_performance} hence leverages concentration inequalities for Markov chains, provided in Appendix \refAppendixSection{sec:Concentration_inequalities_for_BMCs}, and a spectral analysis of $\hat{N}$ or $\hat{N}_\Gamma$. 

\section{Spectral analysis}

The main ingredient in the proof of \refTheorem{thm:Upper_bound_on_SVDCAs_performance} is a concentration bound on the spectral norm of the matrix $\hat{N}_\Gamma$ centered around its mean, or more precisely on $\pnorm{ \hat{N}_\Gamma - N }{}$. The tighter such a bound is, the tighter our performance analysis of the Spectral Clustering Algorithm will be. Note that the concentration of the spectral norm holds for the trimmed matrix, i.e., \jaron{a matrix based on $\hat{N}$ in which the rows and columns that correspond to states that are visited too often are set to $0$}.



\begin{proposition}
\label{prop:Spectral_concentration_bound_for_BMCs}
For any \gls{BMC}, $\pnorm{ \hat{N}_\Gamma - N }{} = \bigOPbig{ \sqrt{ \frac{T}{n} \ln{ \frac{T}{n} } } }$.
\end{proposition}

The above concentration result, proved in \refAppendixSection{sec:Spectral_analysis_for_BMCs}, is sufficiently tight for the purposes of this paper, but can be improved up to logarithmic terms. The primary challenge one encounters in establishing this bound is that $\hat{N}$ is a random matrix with stochastically \emph{dependent} entries, as explained in the introduction. The concentration of the entire spectrum of $\hat{N}_\Gamma$ would be an intriguing topic for future study.

\section{Proof of \refTheorem{thm:Upper_bound_on_SVDCAs_performance}}

Throughout the proof, we use the following notation: \jaron{$N_{x,y} \triangleq \expectationWrt{ \hat{N}_{x,y} }{\Phi}= T \Pi_x P_{x,y}$} for $x,y \in \mathcal{V}$, $\hat{R}^{\star} = [\hat{R} , \hat{R}^\top]$, $N^\star = [N, N^\top]$, and $\hat{N}^\star = [\hat{N}_\Gamma,\hat{N}^\top_\Gamma]$. We further introduce the quantity 
\begin{equation}
D(\alpha,p) 
\triangleq \min_{a,b: a\neq b} \sum_{k=1}^K \Bigl( \bigl( \frac{ \pi_{a} p_{a,k} }{\alpha_k \alpha_{a} } - \frac{ \pi_{b} p_{b,k} }{\alpha_k \alpha_{b} } \bigr)^2 + \bigl( \frac{ \pi_{k} p_{k,a} }{\alpha_k \alpha_{a} } - \frac{ \pi_{k} p_{k,b} }{\alpha_k \alpha_{b} } \bigl)^2 \Bigr).
\end{equation}
Note that $D(\alpha,p) = 0$ if and only if there exist $a, b$ such that 
(C1) $p_{a,k} = p_{b,k}$ for all $k$,
(C2) $p_{k,a} / \alpha_a = p_{k,b} / \alpha_{b}$ for all $k$, and
(C3) $\pi_{a} / \alpha_a = \pi_{b} / \alpha_{b}$.
Under (C1)--(C3), $I(\alpha, p) = 0$. Thus, $D(\alpha,p) > 0$ when $I(\alpha, p) >0$.

The proof of \refTheorem{thm:Upper_bound_on_SVDCAs_performance} consists of four steps. 

\begin{itemize}
\item[Step 1.] We show that $N^\star$ satisfies a \emph{separability property}: i.e., if two states $x, y \in \mathcal{V}$ do not belong to the same cluster, the $l_2$-distance between their respective rows ${N}^\star_{x,\cdot}$, ${N}^\star_{y,\cdot}$ is at least $\Omega( T \sqrt{D(\alpha,p)} / n^{3/2} )$. 
\item[Step 2.] We upper bound the error $\pnorm{ \hat{R}^\star - N^\star }{\mathrm{F}}$ using $\pnorm{ \hat{N}_\Gamma - N }{}$.
\item[Step 3.] We prove that if $\pnorm{ \hat{N}_\Gamma - N }{}$ is small enough as $n$ grows large, then for any state $x$ misclassified under the Spectral Clustering Algorithm, $\pnorm{ {\hat{R}}^\star_{x,\cdot} - {N}^\star_{x,\cdot} }{2}$ is $\OmegaP{ { T \sqrt{D(\alpha,p)} } / { n^{3/2} } }$. In other words, ${\hat{R}}^\star$ inherites the separability property.
\item[Step 4.] Proposition \ref{prop:Spectral_concentration_bound_for_BMCs} ensures that the statement derived in Step 3 holds. From there and using the result of Step 2, we conclude that the number of misclassified states satisfies \refTheorem{thm:Upper_bound_on_SVDCAs_performance}. 
\end{itemize}

\noindent
Step 1. This is formalized in \refLemma{lem:N_vector_error_scales_minimally_with_this_bound_if_x_y_differ_clusters}, and proven in \refSupplementaryMaterial{suppl:Proof__Separability_property_of_Nhat}. It is a consequence of the block structure of matrix $N$.

\begin{lemma}
\label{lem:N_vector_error_scales_minimally_with_this_bound_if_x_y_differ_clusters}
For $x,y \in \mathcal{V}$ s.t.\ $\sigma(x) \neq \sigma(y)$,
$
\pnorm{ {N}^\star_{x,\cdot} - {N}^\star_{y,\cdot} }{2}
= \Omega \Bigl( \frac{ T \sqrt{D(\alpha,p)} }{ n^{3/2} } \Bigr).
$
\end{lemma}

\noindent
Step 2. \refLemma{lem:Concentration_between_Rhat_and_P} shows that the error $\pnorm{\hat{R}^\star - N^\star}{\mathrm{F}}$ is asymptotically bounded by $\pnorm{\hat{N}_\Gamma-N}{}$, and is proven in \refSupplementaryMaterial{suppl:Proof__Concentration_bound_relating_Rhat_and_P}. The proof relies on a powerful bound relating decompositions of random matrices and their spectra \cite{halko_finding_2011}.

\begin{lemma}
\label{lem:Concentration_between_Rhat_and_P}
$\pnorm{ \hat{R}^\star - N^\star }{\mathrm{F}} \leq \sqrt{16K} \pnorm{\hat{N}_\Gamma-N}{}$.
\end{lemma}

\noindent
Step 3. The result obtained in this step is stated in \refLemma{lem:Misclassification_separation_in_Rhat_and_P}. Its proof, presented  in \refSupplementaryMaterial{suppl:Proof__Separability_property_of_Rhat}, requires a bound on the spectral concentration rate of $\hat{N}_\Gamma$'s noise matrix, relies on Lemmas~\ref{lem:N_vector_error_scales_minimally_with_this_bound_if_x_y_differ_clusters}, \ref{lem:Concentration_between_Rhat_and_P}, and exploits the design of the $K$-means algorithm used after the spectral decomposition.  

\begin{lemma}
\label{lem:Misclassification_separation_in_Rhat_and_P}
If $\pnorm{ \hat{N}_\Gamma - N }{} = \smallOP{ f_n }$ for some sequence $f_n = \smallO{ T/n }$ and there exists a sequence $h_n$ such that $\omega\bigl( f_n / \sqrt{n} \bigr) = h_n = \smallO{ T \sqrt{D(\alpha,p)} / n^{3/2} }$, then
\begin{equation}
\pnorm{ {\hat{R}}^\star_{x,\cdot} - {N}^\star_{x,\cdot} }{2} 
= \OmegaPbig{ \frac{ T D^{1/2}(\alpha,p) }{ n^{3/2} } }
\quad
\textrm{for any misclassified state}
\quad
x \in \mathcal{E}.
\end{equation}
\end{lemma}

\noindent
Step 4. In view of \refProposition{prop:Spectral_concentration_bound_for_BMCs}, the conditions of \refLemma{lem:Misclassification_separation_in_Rhat_and_P} are satisfied for e.g.\ $f_n = \bigl( (T/n) \ln{ ( T/n ) } \bigr)^{ 1/2 + \zeta }$ and $h_n = ( f_n / \sqrt{n} ) (T/n)^\zeta$ with $0 < \zeta < 1/4$. The final step is then almost immediate. Indeed, we have because of \refLemma{lem:Misclassification_separation_in_Rhat_and_P}
$
\| \hat{R}^\star - N^\star \|_{\mathrm{F}}^2 \ge \sum_{x\in \mathcal{E} } \pnorm{ {\hat{R}}^\star_{x,\cdot} - {N}^\star_{x,\cdot} }{2}^2 = \cardinality{ \mathcal{E} } \OmegaP{ { T^2 D(\alpha,p) } / { n^{3} } }.
$
Hence, using Lemma \ref{lem:Concentration_between_Rhat_and_P},
$
( 1 / \cardinality{ \mathcal{E} } ) \pnorm{\hat{N}_\Gamma-N}{}^2 
= \OmegaP{ { T^2 D(\alpha,p) } / { n^{3} } }
$,
or equivalently, since $I(\alpha,p)>0$ and hence $D(\alpha,p)>0$,  
$
( n / \cardinality{ \mathcal{E} } ) \pnorm{\hat{N}_\Gamma-N}{}^2 
= \OmegaP{ { T^2 } / { n^{2} } }
$.
We also have from \refProposition{prop:Spectral_concentration_bound_for_BMCs} that $\pnorm{ \hat{N}_\Gamma - N }{}^2 = \bigOP{ \frac{T}{n} \ln{ \frac{T}{n} } }$. Combining the two previous equalities, we conclude that:
$
\cardinality{ \mathcal{E} } / n = \bigOP{ \frac{n}{T} \ln{ \frac{T}{n} } }
$.
This is obtained by applying Lemma \ref{lem:BigOP_ratio} presented in \refAppendixSection{sec:Stochastic_boundedness_properties} (with $X_n = \pnorm{ \hat{N}_\Gamma - N }{}^2$, $Y_n=\cardinality{ \mathcal{E} }/n$, $y_n = T^2/n^2$, and $x_n = (T/n) \ln{ (T/n) }$).

\chapter{Performance of the Cluster Improvement Algorithm}
\label{sec:The_cluster_improvement_algorithm}

\revisedPartBegin

\section{Intuition behind the algorithm}

Let us first briefly explain the intuition behind the Cluster Improvement Algorithm: given a cluster assignment, the algorithm inspects for each state the cluster assignment that makes the observed sample path the most likely. It then assigns the state accordingly.

Note that such a greedy assignment based on maximizing a likelihood function is generally not guaranteed to converge to a good minimum. Indeed; first, if the initial cluster assignment is far from the true assignment, the approximated \gls{BMC} parameters $\hat{p}$, $\hat{\pi}$, and $\hat{\alpha}$ will also be far from the true parameters. The cluster improvement algorithm will then weigh the sample paths according to an \emph{incorrect model} and may assign states to incorrect clusters. Second, the induced \emph{dependencies} between consecutive updates of our improvement algorithm threaten our chances of deriving a performance upper bound that gets tighter with the number of iterations.

\section{Proof of \refTheorem{thm:Conditions_for_improvement_algorithm}}

Our proof tackles both concerns -- the concerns of assigning vertices according to an incorrect model, and of having strong dependencies between consecutive updates -- by considering the precise asymptotic concentration rates of the \gls{BMC}. Tying the first concern to concentration is straightforward: as long as the initial cluster assignment provides sufficiently good estimates of the \gls{BMC} parameters, the sample path will be weighted according to a model that is close to the ground truth. The method with which we overcome the second concern is more refined. Specifically, in the proof of \refTheorem{thm:Conditions_for_improvement_algorithm}, we split the set of all states into a set of \emph{well-behaved} states $\mathcal{H}$ and a set of \emph{forlorn} states $\mathcal{H}^{\mathrm{c}}$. 
\begin{definition}
The set of well-behaved states $\mathcal{H}$ is the largest set of states $x \in \Gamma$ that satisfy the following two properties:
\begin{itemize}
\item[$\mathrm{(H1)}$] When $x\in \mathcal{V}_i$, for all $j\neq i$, 
\begin{equation} 
\sum_{k=1}^K \Bigl( \hat{N}_{x, \mathcal{V}_k} \ln{ \frac{p_{i,k}}{p_{j,k}} } + \hat{N}_{\mathcal{V}_k, x} \ln{ \frac{p_{k,i} \alpha_j}{p_{k,j} \alpha_i} } \Bigr) 
+  \Bigl(\frac{\hat{N}_{ \mathcal{V}_j, \mathcal{V}}}{\alpha_j n} - \frac{\hat{N}_{ \mathcal{V}_i, \mathcal{V}}}{\alpha_i n} \Bigr) 
\geq \frac{T}{2n} I(\alpha,p).
\end{equation}
\item[$\mathrm{(H2)}$] $\hat{N}_{ x, \mathcal{V} \setminus \mathcal{H} } + \hat{N}_{ \mathcal{V} \setminus \mathcal{H}, x } \leq 2 \ln{ ( (T/n)^2 ) }$.
\end{itemize}
\end{definition}

States in $\mathcal{H}$ satisfy properties (H1) and (H2), which guarantee that these states will likely be assigned to their true cluster in our greedy repeated local maximization of the log-likelihood function. To be precise, the set is designed such that the cardinality of the intersection $\itr{\mathcal{E}_{\mathcal{H}}}{t}\triangleq \itr{\mathcal{E}}{t} \cap \mathcal{H}$ of incorrectly classified states and well-behaved states shrinks at each iteration with high probability. Furthermore, as we will show this set converges to the empty set with high probability after $t \approx \ln{n}$ iterations. Because we cannot guarantee that states in $\mathcal{H}^{\mathrm{c}}$ are classified correctly as the algorithm greedily allocates vertices to clusters, we simply treat all of these vertices as being misclassified -- a worst-case upper bound. By then estimating the size of $\mathcal{H}^{\mathrm{c}}$, we are able to bound the total number of misclassified vertices after $t \in \naturalNumbersZero$ improvement steps of the Cluster Improvement Algorithm.

\paragraph{How we prove that $\cardinality{ \itr{\mathcal{E}_{\mathcal{H}}}{t} }$ shrinks} 
\refProposition{prop:Error_on_the_set_of_misclassified_wellbehaved_states} quantifies how $\cardinality{ \itr{\mathcal{E}_{\mathcal{H}}}{t} }$ is reduced in a single iteration. 

\begin{proposition}
\label{prop:Error_on_the_set_of_misclassified_wellbehaved_states}
If $I(\alpha,p) > 0$ and $T = \omega(n)$, and $\cardinality{ \itr{\EHset}{t} } = \bigOP{ \itr{e_n}{t} }$ for some $0 < \itr{e_n}{t} = \smallO{\frac{n}{\ln (T/n)}}$, then 
\begin{equation}
\cardinality{ \itr{\EHset}{t+1} } 
\asymp_{\mathbb{P}} \itr{e_n}{t+1} 
= \bigObig{ \itr{e_n}{t} \frac{n}{T}\ln({T\over n}) }
= \smallO{ \itr{e_n}{t} }.
\label{eqn:Improvement_of_misclassification_statistics_per_time_step}
\end{equation}
\end{proposition}

To establish \refProposition{prop:Error_on_the_set_of_misclassified_wellbehaved_states}, observe that after the $(t+1)$-th iteration, for any misclassified state $x$, its true cluster $\sigma(x)$ does not maximize the objective function $\itr{u_x}{t}(c)$. Hence, summing over all misclassified states that also belong to $\mathcal{H}$, we obtain
\begin{equation}
E \triangleq \sum_{ x \in \itr{\EHset}{t+1} } \bigl( \itr{u_x}{t}( \itr{\sigma}{t+1}(x) ) - \itr{u_x}{t}( \sigma(x) ) \bigr) \geq 0.
\label{eq:eee}
\end{equation}
We prove \refProposition{prop:Error_on_the_set_of_misclassified_wellbehaved_states} by analyzing $E$. After substituting $\itr{u_x}{t}$'s definition \refEquation{eqn:Objective_function_for_the_improvement_algorithm} into \refEquation{eq:eee}, we decompose $E$ as $E = E_1 + E_2 + U$, where  
\begin{align}
&
E_1
= 
\sum_{ x \in \itr{\EHset}{t+1} } 
\Bigl\{ 
\sum_{k=1}^K 
\Bigl(
\hat{N}_{x,\mathcal{V}_k} \ln{ \frac{ p_{ \itr{\sigma}{t+1}(x), k } }{ p_{ \sigma(x), k } } }
+ \hat{N}_{\mathcal{V}_k,x} \ln{ \frac{ p_{ k, \itr{\sigma}{t+1}(x) }\alpha_{\sigma(x)} }{ p_{ k, \sigma(x)}\alpha_ {\itr{\sigma}{t+1}(x) } } }
\Bigl)
\nonumber \\ &
\quad \quad \quad  \quad \quad \quad\quad \quad \quad
+ \Bigl( \frac{ \hat{N}_{ \mathcal{V}_{\sigma(x)}, \mathcal{V} } }{ \alpha_{\sigma(x)} n  } - \frac{ \hat{N}_{ \mathcal{V}_{ \itr{\sigma}{t+1}(x) }, \mathcal{V} } }{  \alpha_{ \itr{\sigma}{t+1}(x) } n}  \Bigr)
\Bigr\}
\nonumber \\ &
E_2 
= 
\sum_{ x \in \itr{\EHset}{t+1} } \sum_{k=1}^K 
\Bigl( 
\bigl( \hat{N}_{x,\itr{\hat{\mathcal{V}}}{t}_k} - \hat{N}_{x,\mathcal{V}_k} \bigr) \ln{ \frac{ p_{ \itr{\sigma}{t+1}(x), k } }{ p_{ \sigma(x), k } } }
\nonumber \\ &
\quad \quad \quad  \quad \quad \quad\quad \quad \quad
+ \bigl( \hat{N}_{\itr{\hat{\mathcal{V}}}{t}_k,x} - \hat{N}_{\mathcal{V}_k,x} \bigr) \ln{ \frac{ p_{ k, \itr{\sigma}{t+1}(x) } }{  p_{ k, \sigma(x) } } }
\Bigr), 
\nonumber \\ &
U = E-E_1-E_2.
\end{align}
Importantly, note that $E_1$ and $E_2$ account for the true model parameters $p$ and $\alpha$, whereas in the functions $\itr{u_x}{t}$ used in the algorithm, we replace these parameters by their estimates $\hat{p}$ and $\hat\alpha$. The term $U$ hence captures issues due to estimation errors.
By construction of $\mathcal{H}$, $E_1 \le - (T/n) I(\alpha,p) \cardinality{ \itr{\EHset}{t+1} }$. Using  concentration results for the \gls{BMC}, we can show that $E_2 \approx \pnorm{ \hat{N}_\Gamma - N }{} ( \cardinality{ \itr{\EHset}{t+1} } \cardinality{ \itr{\EHset}{t} } )^{1/2}$ and that $U \approx 0$,  asymptotically. These observations are formalized in  \refLemma{lem:Leading_behavior_of_E1_E2_E3_E4}, proved in \refSupplementaryMaterial{supple:Proof_EEE}. 

\begin{lemma}
\label{lem:Leading_behavior_of_E1_E2_E3_E4}
If $T = \omega(n)$, $\cardinality{ \itr{\EHset}{t} } = \bigOP{ \itr{e_n}{t} }$ for some $0 < \itr{e_n}{t} = \smallO{\frac{n}{\ln (T/n)}}$
and
$\cardinality{ \itr{\EHset}{t+1} } \asymp_{\mathbb{P}} \itr{e_n}{t+1}$,
then:
\begin{itemize}
\item[] $- E_1 = \Omega_{\mathbb{P}} ( I(\alpha,p) \frac{T}{n} \itr{e_n}{t+1} )$,
\item[] $| U | = \bigO{ \itr{e_n}{t+1} \sqrt{\frac{T}{n}} \ln{ \frac{T}{n} } + \itr{e_n}{t+1}\frac{\itr{e_n}{t}}{n}\frac{T}{n} \ln \frac{T}{n} }$, and 
\item[] $| E_2 | \leq F_1+F_2+F_3$ with $F_1 = \bigOPbig{ \frac{T}{n} \frac{ \itr{e_n}{t} }{n} \itr{e_n}{t+1} }$,
\item[] $F_2 = \bigOP{ \sqrt{{T\over n}\ln({T\over n}) \itr{e_n}{t} \itr{e_n}{t+1} } }$, and $F_3 = \bigOP{ \ln{ (T/n)^2 }  \itr{e_n}{t+1} }$.
\end{itemize}

\end{lemma}

Combining \refLemma{lem:Leading_behavior_of_E1_E2_E3_E4} with the fact that $E \geq 0$ yields \refEquation{eqn:Improvement_of_misclassification_statistics_per_time_step}. Indeed, we have $- E_1 \leq | E_2 | + | U |$ almost surely. Now observe that since $\sqrt{T\over n} \ln{T\over n}=o({T\over n})$ and by assumption $\itr{e_n}{t}=o(n/\ln(T/n))$, $|U|$ is negligible compared to $-E_1$, which implies that $-E_1-|U| =   \Omega_{\mathbb{P}} \bigl( I(\alpha,p) \frac{T}{n} \itr{e_n}{t+1} \bigr)$. Similarly, since by assumption $\itr{e_n}{t}=o(n)$ and $\ln{ ((T/n)^2) } = o({T\over n})$, $F_1$ and $F_3$ are also negligible compared to $-E_1$, which implies that $-E_1-|U| -F_1-F_3 =   \Omega_{\mathbb{P}} \bigl( I(\alpha,p) \frac{T}{n} \itr{e_n}{t+1} \bigr)$. Finally, in view of \refLemma{lem:Leading_behavior_of_E1_E2_E3_E4} 
$
\Omega_{\mathbb{P}} \bigl( I(\alpha,p) \frac{T}{n} \itr{e_n}{t+1} \bigr)
= -E_1 - |U| - F_1 - F_3
\leq F_2 
= \bigOP{ \sqrt{ \frac{T}{n} \ln{ \frac{T}{n} } \itr{e_n}{t} \itr{e_n}{t+1} } }
$.
We deduce that $I(\alpha,p) \itr{e_n}{t+1} = \bigO{  \sqrt{ \frac{n}{T} \ln{ \frac{T}{n} } \itr{e_n}{t} \itr{e_n}{t+1} } }$; see \refLemma{lem:BigOP_consistency_argument} for the precise justification. Since $I(\alpha,p) \allowbreak > 0$, we obtain $\itr{e_n}{t+1} = \bigO{ \itr{e_n}{t} \frac{n}{T} \ln{ \frac{T}{n} } }$, which concludes the proof of \refProposition{prop:Error_on_the_set_of_misclassified_wellbehaved_states}.

\paragraph{How we bound the size of $\mathcal{H}^{\mathrm{c}}$} 

\refProposition{prop:Set_of_forlorn_states_shrinks} provides an upper bound of the number of states in $\mathcal{H}^{\mathrm{c}}$, and is proved in \refAppendixSection{suppl:Proof__Size_of_complement_of_H}. 

\begin{proposition}
\label{prop:Set_of_forlorn_states_shrinks}
If $I(\alpha,p) > 0$ and $T = \omega(n)$, and $\cardinality{ \itr{\EHset}{t} } = \bigOP{ \itr{e_n}{t} }$ for some $0 < \itr{e_n}{t} = \smallO{n}$, then 
$
\cardinality{ \itr{ \mathcal{E}_{ \mathcal{H}^{\mathrm{c}} } }{t} }
\leq \cardinality{ \mathcal{H}^{\mathrm{c}} } 
= \bigOP{ 
n \exp{} \bigl( - C \frac{T}{n} I(\alpha, p) \bigr) } 
$,
where $C = {\alpha_{\min}^2} / ( {720 \eta^3 \alpha_{\max}^2} )$, $\alpha_{\max} = \max_{i} \alpha_i$, and $\alpha_{\min} = \min_{i} \alpha_i$.
\end{proposition}

We now sketch the proof of \refProposition{prop:Set_of_forlorn_states_shrinks}. First note that the number of states not in $\Gamma$ (obtained after the trimming process) is negligible, i.e., $n\exp{ (- \frac{T}{n} \ln{ \frac{T}{n} } ) }$. We then upper bound the number of states that do not satisfy (H1). Let $x\in \mathcal{V}_i$. If $x$ does not satisfy (H1), there exists $j\neq i$ such that $\hat{I}_{i,j}(x)<{T\over 2n}I(\alpha,p)$, where 
\begin{equation}
\hat{I}_{i,j}(x) \triangleq \sum_{k=1}^K \Bigl( \hat{N}_{x, \mathcal{V}_k} \ln{ \frac{p_{i,k}}{p_{j,k}} } + \hat{N}_{\mathcal{V}_k, x} \ln{ \frac{p_{k,i} \alpha_j}{p_{k,j} \alpha_{i}} } \Bigr) 
+  \Bigl(\frac{\hat{N}_{ \mathcal{V}_j, \mathcal{V}}}{\alpha_j n} - \frac{\hat{N}_{ \mathcal{V}_{i}, \mathcal{V}}}{\alpha_{i} n} \Bigr).
\end{equation}
Observe that $\expectation{ \hat{I}_{i,j}(x) } = \frac{T}{n} I_{i,j}(\alpha,p)$ where $I_{i,j}(\alpha,p)$ is the quantity involved in the definition of $I(\alpha,p)$; see \refEquation{eqn:Ialphabetap}. In particular, $\expectation{ \hat{I}_{i,j}(x) } \geq {T\over n}I(\alpha,p)$. Hence, $x$ does not satisfy (H1) implies that for some $j\neq i$, $\hat{I}_{i,j}(x)<{T\over 2n}I(\alpha,p)$ and $\expectation{ \hat{I}_{i,j}(x) } \geq \frac{T}{n} I(\alpha,p)$. Using concentration results for the \gls{BMC} \refAppendixSection{sec:Concentration_inequalities_for_BMCs}, this event happens with probability at most $\exp{} ( - C \frac{T}{n} I(\alpha,p) )$. We next deduce a bound on the expected number of states not satisfying (H1). From there, using Markov's inequality, we obtain that the number of states not satisfying (H1) does not exceed $n \exp{}( - C \frac{T}{n} I(\alpha, p) )$ with high probability. Note that the constant $C = {\alpha_{\min}^2} / { ( 720 \eta^3 \alpha_{\max}^2 ) }$ stems from the precise application of concentration results. 

We then complete the proof through the following argument. Consider the following iterative construction: start with the set $Z(0)$ of all states that do not satisfy (H1). The $t$-th iteration consists of adding to $Z(t-1)$ a state $v$ not satisfying (H2) written w.r.t. $Z(t-1)$, i.e., $\hat{N}_{v,Z(t-1)} +\hat{N}_{Z(t-1),v} > 2 \ln{ ( (T/n)^2 ) }$. If such a state does not exist, the construction stops. Let $Z(t^*)$ be the final set: $t^*$ is the number of iterations before the construction stops. Then $Z(t^*)$ is such that for all $x\notin Z(t^*)$, $x$ satisfies (H1) and (H2) written w.r.t. $Z(t^*)$. Hence by definition of ${\cal H}$, the size of ${\cal V}\setminus Z(t^*)$ is smaller than that of ${\cal H}$, and thus $|{\cal H}^{\mathrm{c}}| \le |Z(t^*)|$. To get an upper bound on $|Z(t^*)|$, we just establish an upper bound on $t^*$ using concentration results (at each iteration $t$, a \emph{large} number (specifically $2\ln{( (T/n)^2 )}$) of observed transitions inside $Z(t)$ is added, which rapidly becomes impossible). In summary: $|{\cal H}^{\mathrm{c}}| \le |Z(t^*)|\le t^* + |Z(0)|\le t^* + n \exp{} ( - C \frac{T}{n} I(\alpha, p) )$ with high probability.

\paragraph{Iterating the bound} 
If we initiate the Cluster Improvement Algorithm using the cluster assignment provided by the Spectral Clustering Algorithm when $T = \omega( n )$, from Theorem~\ref{thm:Upper_bound_on_SVDCAs_performance}, we satisfy the initial condition $\cardinality{ \itr{\EHset}{0} } = \smallOP{\frac{n}{\ln (T/n)}}$ of Propositions~\ref{prop:Error_on_the_set_of_misclassified_wellbehaved_states}, \ref{prop:Set_of_forlorn_states_shrinks}. Furthermore, since $\cardinality{ \itr{\mathcal{E}}{t} } = \cardinality{ \itr{\EHset}{t} } + \cardinality{ \itr{ \mathcal{E}_{\mathcal{H}^{\mathrm{c}}} }{t} }$, we conclude by iterating the bound in \refEquation{eqn:Improvement_of_misclassification_statistics_per_time_step} that after $t \in \naturalNumbersZero$ improvement steps the Cluster Improvement Algorithm misclassifies at most
\begin{align}
\cardinality{ \itr{\mathcal{E}}{t} }
&
= \bigOPbig{ 
\e{ \ln{n} - t \bigl( \ln{ \frac{T}{n} } - \ln{ \ln{ \frac{T}{n} } } \bigr) } 
+ \e{ \ln{n} - \frac{\alpha_{\min}^2}{720 \eta^3 \alpha_{\max}^2} \frac{T}{n} I(\alpha,p) } 
+ \e{ \ln{n} - \frac{T}{n} \ln{ \frac{T}{n} } } 
}
\end{align}
states. This completes the proof of \refTheorem{thm:Conditions_for_improvement_algorithm}.

\revisedPartEnd

\chapter{Acknowledgments} 

\jaron{We would like to thank our anonymous referees: their careful reading and suggestions have led to improved revisions of this work. We also} thank Pascal Lagerweij for having conducted the numerical experiment in \refSection{sec:Numerical_feasibility_region}.

\chapter*{Bibliography}

\AtNextBibliography{\normalsize} 
\setlength\bibitemsep{0pt}
\printbibliography[heading=none,notcategory=fullcited]

\newpage

\begin{supplement}[id=suppA]
\slink[url]{This concludes the supplementary material}
\setcounter{section}{0}
\renewcommand{\thesection}{SM\arabic{section}}

\chapter{Concentration inequalities for \glspl{BMC}}
\label{sec:Concentration_inequalities_for_BMCs}

Recall the notation $\hat{N}_{\mathcal{A},\mathcal{B}} = \sum_{ x \in \mathcal{A} } \sum_{ y \in \mathcal{B} } \hat{N}_{x,y}$ for any subsets $\mathcal{A}, \mathcal{B} \subseteq \mathcal{V}$. 

\begin{proposition}
\label{prop:Collection_of_relevant_concentration_inequalities_for_BMCs}
The following concentration inequalities hold for \glspl{BMC} for $T/n$ large enough (larger than a constant that does not depend on $n$, but on the BMC parameters):

\noindent
$-$ There exists an absolute constant $c_1>0$ such that for $k = 1, \ldots, K$,
\begin{equation}
\probabilityBig{ | \hat{N}_{\mathcal{V},\mathcal{V}_k} - N_{\mathcal{V},\mathcal{V}_k} | \ge c_{1} \sqrt{ T \ln{ \frac{T}{n} } } } 
\leq \frac{n^2}{T^2}.
\label{eq:Nseti}
\end{equation}

\noindent
$-$ There exists an absolute constant $c_2>0$ such that  for any $x \in \mathcal{V}$,
\begin{gather}
\probabilityBig{ \Bigl| \hat{N}_{\mathcal{V},x} - N_{\mathcal{V},x} \Bigr| \ge c_2 \frac{T}{n}\ln \frac{T}{n} } 
\leq \e{ - 2\frac{T}{n} \ln{ \frac{T}{n} } }, 
\label{eq:prop10-2}
\\
\probabilityBig{ \Bigl| \hat{N}_{\mathcal{V},x} - N_{\mathcal{V},x} \Bigr| \ge c_2 \sqrt{\frac{T}{n}} \ln{n} } 
\leq \frac{1}{n^2}. \label{eq:trim}
\end{gather}

\noindent
$-$ 
There exists an absolute constant $d_3 > 0$ such that for any $c_3 \geq 1$ and any subset $\mathcal{S} \subset \mathcal{V}$ of size $\cardinality{ \mathcal{S} } = \lfloor n \exp{ \bigl( - (T/n) \ln{(T/n)} \bigr) } \rfloor$
\begin{equation}
\probability{ | \hat{N}_{\mathcal{V},\mathcal{S}} - N_{\mathcal{V},\mathcal{S}} | \geq c_3 n } 
\leq \e{ - d_3 c_3 n }. 
\label{eq:trim2}
\end{equation}

\noindent
$-$ There exists an absolute constant $c_4>0$ such that for any $i, j \in \{ 1, \ldots, K \}$, and $x\in {\cal V}_i$,
\begin{equation}
\probabilityBig{  \sum_{t=1}^T f(\tilde{X}_{t})  \leq \frac{T}{2n}I(\alpha,p) } 
\leq \e{ - c_4 \frac{T}{n}I(\alpha,p) } 
\label{eq:KLsim}
\end{equation}
where $\tilde{X}_t = (X_t,X_{t+1})$ and
\begin{align*}
f(\tilde{X}_{t}) 
= & \sum_{k=1}^K \bigl( \indicator{ X_{t-1} = x , X_t \in \mathcal{V}_k } \ln{ \frac{ p_{i,k} }{ p_{j,k} } } + \indicator{ X_{t-1} \in \mathcal{V}_k, X_{t} = x } \ln{ \frac{ p_{k,i} \alpha_j }{ p_{k,j} \alpha_i } } \bigr)\\
& + {1\over n} \Bigl( \frac{ \indicator{ X_{t-1} \in \mathcal{V}_j } }{ \alpha_j } - \frac{ \indicator{ X_{t-1} \in \mathcal{V}_i } }{ \alpha_i } \Bigr).
\end{align*}
The constant $c_4$ can be chosen as $c_4 = \alpha_{\min}^2/ (720 \eta^3\alpha_{\max}^2)$. 

\noindent
$-$ 
There exists an absolute constant $c_5>0$ such that
\begin{equation}
\max_{ y \in \Gamma } \bigl\{ \hat{N}_{\Gamma,y} \vee \hat{N}_{y,\Gamma} \bigr\}
\leq c_5 \frac{T}{n} \ln{ \frac{T}{n} }
\quad
\textrm{with probability}
\quad
1 - n \e{ - \frac{T}{n} \ln{ \frac{T}{n} } }.
\label{eqn:Concentration_result_on_the_maximum_row_and_column_sum_over_Nhat_Gamma}
\end{equation}

\noindent
$-$ There exists an absolute constant $c_6>0$ such that 
\begin{equation}
\max_{\mathcal{A},\mathcal{B}\subset {\cal V}} \bigl| \hat{N}_{\mathcal{A},\mathcal{B}} - N_{\mathcal{A},\mathcal{B}} \bigr| 
< c_6 \sqrt{nT} 
\quad
\textrm{with probability}
\quad
1 - \e{-(4 - \ln{2} ) n}.
\label{eq:maxAB}
\end{equation}
\end{proposition}

The concentration inequalities in \refProposition{prop:Collection_of_relevant_concentration_inequalities_for_BMCs} can all be shown using \cite[Thm.~3.4]{paulin_concentration_2015}, which we reproduce here for your convenience. \refTheorem{thm:Paulins_concentration_result} concerns a (non-reversible) stationary Markov chain $\{ X_t\}_{t\ge 0}$ with state space $\Omega$ and stationary distribution $\Pi$. Its pseudo spectral gap $\gamma_{\mathrm{ps}}$ is defined in \cite[Eq.~(3.3)]{paulin_concentration_2015}, and we give it below the theorem. 

\begin{theorem}[Paulin, 2015]
\label{thm:Paulins_concentration_result}
Let $f\in L^2 (\Pi)$ with $|f(x) - \mathbb{E}_\Pi (f)| \le C$ for every $x\in \Omega$ (for some constant $C>0$). Let $V_f$ be the variance of $f(X)$ when $X$ follows the stationary distribution $\Pi$. Then, for any $z>0$,
\begin{equation}
\probabilityBigWrt{ \Bigl| \sum_{t=1}^T f(X_t) - \expectationBigWrt{ \sum_{t=1}^T f(X_t) }{\Pi} \Bigr| \geq z }{\Pi}
\leq 2 \exp{ \Bigl( -\frac{z^2 \gamma_{ps}}{8(T + 1/\gamma_{ps})V_f + 20z C} \Bigr) }.
\label{thm:bernstein}
\end{equation}
\end{theorem}

In order to apply \refTheorem{thm:Paulins_concentration_result}, we also use \cite[Prop.~3.4]{paulin_concentration_2015}, which relates the pseudo spectral gap to the mixing time of the Markov chain: for any $\varepsilon \in [0,1)$,
\begin{equation}
\gamma_{\mathrm{ps}} 
\triangleq \max_{ i \geq 1 } \frac{ 1 - \lambda \bigl( ( P^* )^i P^i \bigr) }{i}
\geq \frac{1 - \varepsilon}{ t_{\mathrm{mix}}(\varepsilon/2) }
\quad
\textrm{with}
\quad 
P^*(x,y) 
\triangleq \frac{P(x,y)}{ \Pi(x) } \Pi(y).
\end{equation}
Next it is important to note that \refProposition{prop:Mixing_times_of_Markov_chains_with_transition_matrices_P_and_Q} also holds for the Markov chain induced by the \emph{transitions} of the \gls{BMC}. To see this, define $\tilde{X}_t \triangleq (X_{t-1}, X_{t})$ such that $\process{ \tilde{X}_t }{ t \geq 0}$ denotes a Markov process describing the transitions of $\process{ X_t }{ t \geq 0 }$. Let $\tilde{P}$ and $\tilde{\Pi}$ be the transition kernel and the stationary distribution of $\tilde{X}_t$, respectively. Note now that $d_{\mathrm{TV}}(P_{x,\cdot}^t, \Pi) = d_{\mathrm{TV}}(\tilde{P}_{(x,y),\cdot}^{t+1}, \tilde{\Pi})$ for all $x,y \in \mathcal{V}$, $t \in \naturalNumbersPlus$. The mixing time of $\process{ \tilde{X}_t }{ t \geq 0 }$, therefore, requires one more transition than the mixing time of  $\process{ X_t }{ t \geq 0 }$. As a consequence of \refProposition{prop:Mixing_times_of_Markov_chains_with_transition_matrices_P_and_Q}, more precisely from \eqref{eqn:Upper_bound_on_dTV_Pt_Pi_mixing}, we have for both of our Markov processes $\process{ X_t }{ t \geq 0 }$, $\process{ \tilde{X}_t }{ t \geq 0 }$ that 
\begin{align}
\gamma_{\mathrm{ps}} \ge \frac{1}{2 (t_{\mathrm{mix}}(1/4) +1)} \ge \frac{1}{2(4\eta+1)}. \label{eq:gammaps}
\end{align}
The last inequality is obtained by observing that $t_{\mathrm{mix}}(1/4)\le - \ln(4)/\ln(1-1/2\eta)\le 2\eta\ln(4)\le 4\eta$.

\revisedPartBegin

\paragraph{Proof of \refEquation{eq:Nseti}} Apply \refTheorem{thm:Paulins_concentration_result} with $f(x) = \indicator{ x \in \mathcal{V}_k }$. Hence $\sum_{t=1}^T f(X_{t}) \allowbreak = \hat{N}_{\mathcal{V},\mathcal{V}_k}$, $C \leq 1$, and $V_f \le \pi_k$. Then, for any constant $c>0$,
\begin{align}
\probability{ | \hat{N}_{\mathcal{V},\mathcal{V}_k} - N_{\mathcal{V},\mathcal{V}_k} | \geq c \sqrt{ T \ln{{T\over n}} } } 
\leq & 2\exp \left( -\frac{c^2}{16 (4\eta+1)}\ln \frac{T}{n} (1+o(1))\right).
\end{align}
The desired inequality is obtained by choosing, in the above inequality, $c=c_1$ such that $c_1^2 \ge 32(4\eta+1)$.

\paragraph{Proof of \refEquation{eq:prop10-2}} Let $f(\cdot) = \indicator{ \cdot  = x }$ such that $\sum_{t=1}^T f(X_{t}) = \hat{N}_{\mathcal{V},x}$, $C \leq 1$, and $V_f \le \pi_k / \alpha_k n$. From \refTheorem{thm:Paulins_concentration_result}, it follows that for any constant $c>0$ 
\begin{equation}
\probabilityBig{ \Bigl| \hat{N}_{\mathcal{V},x} - N_{\mathcal{V},x} \Bigr| \ge c \frac{T}{n}\ln \frac{T}{n} } 
\leq 2 \exp{ \Bigl( - \frac{c}{40(4\eta +1)}\frac{T}{n} \ln{ \frac{T}{n} } (1+o(1))\Bigr) }.
\end{equation}
The desired inequality is obtained by choosing, in the above inequality, $c=c_2$ such that $c_2 \ge 40(4\eta+1)$.

\paragraph{Proof of \refEquation{eq:trim}} We use the same function $f$ as that used in the proof of \refEquation{eq:prop10-2}. \refTheorem{thm:Paulins_concentration_result} yields: for any $c>0$,
\begin{equation}
\probabilityBig{ \Bigl| \hat{N}_{\mathcal{V},x} - N_{\mathcal{V},x} \Bigr| \ge c \sqrt{\frac{T}{n}} \ln{n} } 
\leq 2\exp{ \Bigl( -\frac{c^2 {T\over n}(\ln{n})^2\gamma_{\mathrm{ps}}}{{8T\over \alpha_{\min}n}+20c \sqrt{\frac{T}{n}} \ln{n} }(1+o(1)) \Bigr) }.
\label{eqn:Proof_of_Eq_trim}
\end{equation}
If $8 T / ( n \alpha_{\min} ) \geq 20c \sqrt{T/n} \ln{n}$, the r.h.s.\ in \refEquation{eqn:Proof_of_Eq_trim} is less than $2\exp{} \bigl( -\frac{c^2 \alpha_{\min}}{32(4\eta+1)} \allowbreak ( \ln{n} )^2( 1 + \smallO{1} )\bigr)$. If $8 T / (n \alpha_{\min} ) < 20c \sqrt{T/n} \ln{n}$, the r.h.s.\ in \refEquation{eqn:Proof_of_Eq_trim} is smaller than $2\exp{} \bigl( -\frac{c}{80(4\eta+1)} ( \ln{n} ) (1+ \smallO{1} \bigr)$. The desired inequality is obtained by choosing $c = c_2 > \max \{ \sqrt{32(4\eta+1)/\alpha_{\min}}, \allowbreak 160(4\eta+1) \}$.

\paragraph{Proof of \refEquation{eq:trim2}} Let $\mathcal{S} \subset \mathcal{V}$ be such that $\cardinality{S} = \lfloor n \exp(-\frac{T}{n} \ln{ \frac{T}{n}) } \rfloor$. Let $f(\cdot) = \indicator{ \cdot \in \mathcal{S} }$ such that $\sum_{t=1}^T f(X_{t}) = \hat{N}_{\mathcal{V},S}$, $C\leq 1$, and  $V_f = \bigO{ \exp{ ( - \frac{T}{n} \ln{ \frac{T}{n} } ) } }$. \refTheorem{thm:Paulins_concentration_result} implies that for any constant $c_3> 0$,
\begin{equation}
\probabilityBig{ \Bigl| \hat{N}_{\mathcal{V},\mathcal{S}} - N_{\mathcal{V},\mathcal{S}} \Bigr| \geq c_3 n } 
\leq 2 \exp{ \Bigl( - \frac{c_3}{40(4\eta +1)} n(1+o(1)) \Bigr) } .
\end{equation}
The result is proved for $d_3=1/(40(4\eta+1))$.

\paragraph{Proof of \refEquation{eq:KLsim}} We apply \refTheorem{thm:Paulins_concentration_result} to the Markov chain $\process{ \tilde{X}_t }{ t \geq 0}$ and the function $f$ defined in \refEquation{eq:KLsim} for some $i,j$ and $x\in {\cal V}_i$. Observe that we have $\mathbb{E}_{\Pi}[f(\tilde{X}_t)]= (1/n) I_{i,j}(\alpha,p)$ by definition of $f$, where $I_{i,j}(\alpha,p)$ has been introduced when defining $I(\alpha,p)$ in \refEquation{eqn:Ialphabetap}.

When applying \refTheorem{thm:Paulins_concentration_result}, one can easily check that 
\begin{equation}
C \le \frac{1}{n}I_{i,j}(\alpha,p) + \ln \frac{\eta \alpha_{\max}}{\alpha_{\min}}  + \frac{1}{\alpha_{\min}n}. \label{eq:Clower}
\end{equation}

Let us analyze $V_f$. Since $\mathrm{Var}( f(\tilde{X}_{t}) ) \le \expectation{f(\tilde{X}_{t})^2}$,
\begin{equation}
V_f 
\leq \sum_{k=1}^K \frac{1}{\alpha_i n } \Bigl( \pi_i p_{i,k} \Bigl( \ln{ \frac{ p_{i,k} }{ p_{j,k} } } \Bigr)^2 + \pi_k p_{k,i} \Bigl( \ln{ \frac{ p_{k,i} \alpha_j }{ p_{k,j} \alpha_i } } \Bigr)^2 \Bigr) + \bigObig{ \frac{1}{n^2} }.  \label{eq:Vfexp1} 
\end{equation}
Now we can express $I_{i,j}(\alpha,p)$ using KL divergences, and show that:
\begin{equation}
I_{i,j}(\alpha,p) = \frac{\pi_i}{\alpha_i} \mathrm{KL}(p_{i,\cdot } \| p_{j,\cdot }) +  n \mathrm{KL}( \tilde{p}_{\cdot ,i} \| \tilde{p}_{\cdot ,j} ) + \bigObig{ \frac{1}{n} }, \label{eq:Iab}
\end{equation}
where $\tilde{p}_{\cdot ,i} \triangleq [\pi_1\frac{ p_{1,i} }{\alpha_i n},\pi_1 (1- \frac{ p_{1,i} }{\alpha_i n} ),\dots ,\pi_K\frac{ p_{K,i} }{\alpha_i n},\pi_K (1- \frac{ p_{K,i} }{\alpha_i n} )]$ and we have used the fact that $| \ln{(1+x)} - x | \leq x^2$ for $x \in [-1/2,1/2]$. Using $\tilde{p}_{\cdot ,i}$, \eqref{eq:Vfexp1} is reformulated as follows:
\begin{equation}
V_f 
\leq \frac{\pi_i}{\alpha_i n}\sum_{k=1}^K p_{i,k} \Bigl( \ln{ \frac{ p_{i,k} }{ p_{j,k} } } \Bigr)^2 + \sum_{k'=1}^{2K} \tilde{p}_{k',i} \Bigl( \ln{ \frac{ \tilde{p}_{k',i} }{ \tilde{p}_{k',j} } } \Bigr)^2 + \bigObig{ \frac{1}{n^2} }. 
\label{eqn:Intermediate_equation_1_in_the_proof_of_Eq_KLsim}
\end{equation}
Next, in view of \refEquation{eq:Iab}, \refLemma{lem:KL-var} in \refSupplementaryMaterial{sec:The_KL_divergence_and_the_log_square_expression} allows us to upper bound the r.h.s. of \refEquation{eqn:Intermediate_equation_1_in_the_proof_of_Eq_KLsim} using $I_{i,j}(\alpha,p)$. Specifically, we deduce that:
\begin{equation}
V_f 
\leq \Bigl( \eta\frac{\alpha_{\max}}{\alpha_{\min}} \Bigr)^2 \frac{1}{n}I_{i,j}(\alpha,p) + \bigObig{ \frac{1}{n^2} }. \label{eq:Vflower}
\end{equation}

Putting \eqref{eq:gammaps}, \eqref{eq:Clower}, and \eqref{eq:Vflower} into Theorem~\ref{thm:Paulins_concentration_result}, we have
\begin{align}
& \probabilityBig{ \Bigl| \sum_{t=1}^T f(\tilde{X}_{t}) - \mathbb{E}_\pi (f(\tilde{X}_{t})) \Bigr| \geq \frac{T}{2n}I_{i,j}(\alpha,p) } 
\nonumber \\ &  
\leq 2 \exp{ \Bigl( - \frac{1}{ 16(4\eta+1) \bigl( 4\bigl({\eta \alpha_{\max}\over \alpha_{\min}}\bigr)^2 + 5 \ln{ \frac{ \eta \alpha_{\max} }{ \alpha_{\min} } } \bigr) } \frac{T}{n}I_{i,j}(\alpha,p)(1+o(1)) \Bigr) }.
\label{eqn:Intermediate_equation_2_in_the_proof_of_Eq_KLsim}
\end{align}

Note that \refEquation{eq:KLsim} is directly deduced from \refEquation{eqn:Intermediate_equation_2_in_the_proof_of_Eq_KLsim} and the fact that $I_{i,j}(\alpha,p)\ge I(\alpha,p)$. The constant $c_4$ can be chosen as $c_4 = \alpha_{\min}^2/ (720 \eta^3\alpha_{\max}^2)$ (observe that since $\eta\ge 1$, $\ln(\eta\alpha_{\max}/\alpha_{\min}) \le (\eta\alpha_{\max}/\alpha_{\min})^2$).

\revisedPartEnd

\paragraph{Proof of \refEquation{eqn:Concentration_result_on_the_maximum_row_and_column_sum_over_Nhat_Gamma}} This concentration result is a direct consequence of \refEquation{eq:prop10-2}.

\paragraph{Proof of \refEquation{eq:maxAB}} Consider any sets $\mathcal{A}, \mathcal{B} \subset \mathcal{V}$ and let $f( \tilde{X}_t ) = \indicator{ X_{t-1} \in \mathcal{A}, X_t \in \mathcal{B} }$. For this function, $C \leq 1$ and $V_f =  \bigO{ 1  }$. It follows from \refTheorem{thm:Paulins_concentration_result} that for any $c>0$:
\begin{equation*}
\probabilityBig{ \Bigl| \hat{N}_{\mathcal{A},\mathcal{B}} - N_{\mathcal{A},\mathcal{B}} \Bigr| \geq c \sqrt{nT} } 
\leq 2\exp{ ( -{c^2\over 16(4\eta+1)} n(1+o(1)) ) }.
\end{equation*}
The result is obtained by selecting $c=c_6$ such that ${c_6^2\over 16(4\eta+1)}\ge 4$, and by applying a union bound over all possible subsets of ${\cal V}$ (there are $2^n$ such subsets).

\chapter{Proofs of \refChapter{sec:The_block_Markov_chain}}

\section{Proof of Proposition~\ref{prop:Equilibrium_behavior_of_pi}}
\label{suppl:Asymptotic_equilibrium_behavior_of_a_BMC}

\begin{proof}
We first prove that $\vect{\pi}$ is a probability distribution. This follows by (i) definition of $\vect{\pi}$, (ii) symmetry of all states in the same cluster, and (iii) because $\vect{\Pi}$ is a probability distribution:
\begin{equation}
\sum_{k=1}^K \pi_k
\eqcom{i}= \sum_{k=1}^K \lim_{n \to \infty} \bar{\Pi}_k \cardinality{ \mathcal{V}_k }
\eqcom{ii}= \lim_{n \to \infty} \sum_{k=1}^K \sum_{ x \in \mathcal{V}_k } \Pi_x
= \lim_{n \to \infty} \sum_{ x \in \mathcal{V} } \Pi_x
\eqcom{iii}= 1.
\end{equation}

Next, we show that the balance equations hold. For $k = 1, \ldots, K$ it follows by symmetry of any two states $x,z \in \mathcal{V}_k$ that $\Pi_x = \Pi_z = \bar{\Pi}_k$. Hence for any $y \in \mathcal{V}_l$, by (iv) global balance
\begin{equation}
\Pi_y
= \bar{\Pi}_{l} 
\eqcom{iv}= \sum_{k=1}^K \sum_{ x \in \mathcal{V}_k } \Pi_x P_{x,y}
= \sum_{k=1}^K \bar{\Pi}_k ( \cardinality{ \mathcal{V}_k } - \indicator{ k = l } ) \frac{ p_{k,l} }{ \cardinality{ \mathcal{V}_{l} } - \indicator{ k = l } }. 
\end{equation}
Letting $n \to \infty$, we find that $\pi_{l} = \sum_{k=1}^K \pi_k p_{k,l}$ for all $k,l$. This completes the proof.
\end{proof}

\section{Proof of Proposition~\ref{prop:Mixing_times_of_Markov_chains_with_transition_matrices_P_and_Q}}
\label{sec:Proof_of_Mixing_times_of_Markov_chains_with_transition_matrices_P_and_Q}

\revisedPartBegin
\begin{proof}
We will use Dobrushin's ergodic coefficient, which is defined for any stochastic matrix $P$ by \cite[Definition~7.1]{bremaud_markov_1999}
\begin{equation}
\delta(P) 
\triangleq \tfrac{1}{2} \sup_{ x,y \in \mathcal{V} } \sum_{ z \in \mathcal{V} } | P_{x,z} - P_{y,z} |.
\end{equation}
Moreover, Dobrushin's coefficient satisfies 
$
\delta(P)
= 1 - \inf_{ x,y \in \mathcal{V} } \sum_{ z \in \mathcal{V} } \bigl( P_{x,z} \wedge P_{y,z} \bigr)
$ \cite[Eq.~(7.3)]{bremaud_markov_1999}. Now recall our assumption $\exists_{0 < \eta \neq 1} : \max_{a,b,c} \{ p_{b,a} / p_{c,a}, \allowbreak p_{a,b} / p_{a,c} \} \leq \eta$, which implies that 
\begin{align}
\delta(P) 
&
< 1 - \inf_{ x,y \in \mathcal{V} } \sum_{ z \in \mathcal{V} } \bigl( P_{x,z} \wedge P_{y,z} \bigr)
\nonumber \\ &
\leq 1 - \inf_{ x \in \mathcal{V} }  \sum_{ z \in \mathcal{V} } \frac{1}{\eta} P_{x,z}   \frac{\alpha_{\min}n -2}{\alpha_{\min}n}
= 1 - \frac{1}{\eta}\frac{\alpha_{\min}n -2}{\alpha_{\min}n}.
\end{align}
We deduce that when $n\ge 4/\alpha_{\min}$, $\delta(P) < 1 - 1 / (2\eta)$ (recall that $\eta > 1$ by definition).

Next \cite[Thm.~7.2]{bremaud_markov_1999} gives us the convergence rate in terms of Dobrushin's coefficient. Specifically,
\begin{equation}
d_{\mathrm{TV}}( P_{x,\cdot}^t, \Pi )
\leq \bigr( \delta(P) \bigr)^t d_{\mathrm{TV}}( P_{x,\cdot}^0, \Pi )
\quad
\textrm{for}
\quad
x \in \mathcal{V}.
\label{eqn:Upper_bound_on_dTV_Pt_Pi_distance_measure_as_function_of_t}
\end{equation}
As a consequence
\begin{equation}
d_{\mathrm{TV}}( {P}^t_{x,\cdot}, \vect{\Pi} )
\leq \varepsilon
\quad
\textrm{whenever}
\quad 
t \geq \frac{ \ln{\varepsilon} }{ \ln{ \bigl(1-\frac{1}{2\eta} \bigr) } }.
\label{eqn:Upper_bound_on_dTV_Pt_Pi_distance_measure_as_function_of_varepsilon}
\end{equation}
Therefore,
\begin{equation}
t_{\mathrm{mix}}(\varepsilon) \le \frac{ \ln{\varepsilon} }{ \ln{ \bigl(1-\frac{1}{2\eta} \bigr) } }.
\label{eqn:Upper_bound_on_dTV_Pt_Pi_mixing}
\end{equation}
This completes the proof.
\end{proof}
\revisedPartEnd

\chapter{Proofs of \refChapter{sec:The_information_bound_and_the_change_of_measure}}\label{ch5:proofs}

\section{$Q$ is a stochastic matrix}
\label{sec:Q_is_a_stochastic_matrix}

Observe that for $x \in \mathcal{V} \backslash \{ \Vstar \}$,
\begin{align}
\sum_{y \in \mathcal{V}} Q_{x,y} 
&
= \frac{ q_{\omega(x),\star} }{ n } + \sum_{ y \in \mathcal{W}_{\omega(x)} \backslash \{ x \} } \frac{ q_{\omega(x),\omega(x)} }{ \cardinality{ \mathcal{W}_{ \omega(x) } } - 1 } + \sum_{k=1}^K \indicator{ k \neq \omega(x) } \sum_{ y \in \mathcal{W}_{k} } \frac{ q_{\omega(x),k} }{ \cardinality{ \mathcal{W}_{k} } } 
\nonumber \\ &
= \frac{ q_{\omega(x),\star} }{ n } + \sum_{k=1}^K q_{\omega(x),k}
\eqcom{\ref{eqn:Coupling_of_the_parameters_within_q_to_p}}
= \frac{ q_{\omega(x),\star} }{ n } + \sum_{k=1}^K \Bigl( p_{\omega(x),k} - \frac{q_{\omega(x),\star}}{ K n } \Bigr)
= 1.
\end{align}
Similarly for $x = \Vstar$:
$
\sum_{y \in \mathcal{V}} Q_{\Vstar,y} 
= \sum_{k=1}^K \sum_{ y \in \mathcal{W}_{k} } \frac{ q_{\star,k} }{ \cardinality{ \mathcal{W}_{k} } } 
= \sum_{k=1}^K q_{\star,k}= 1
$.

\section{Proof of Proposition~\ref{prop:Equilibrium_behavior_of_piQ}}
\label{sec:Proof_of_Equilibrium_behavior_of_piQ}

We first show that $(\itr{\gamma_\star}{0}, \itr{\gamma_1}{0}, \ldots, \itr{\gamma_K}{0} )$ is a probability distribution. We have: 
\begin{align}
\itr{\gamma_{\star}}{0}+\sum_{k=1}^K \itr{\gamma_k}{0} 
&
= \lim_{n \to \infty} \Bigl( \sum_{k=1}^K \cardinality{ \mathcal{W}_k } \bar{\Pi}_k^{(Q)} + \Pi_{\Vstar}^{(Q)} \Bigr)
= \lim_{n \to \infty} \Bigl( \sum_{k=1}^K \sum_{ x \in \mathcal{W}_k } \Pi_{x}^{(Q)} + \Pi_{\Vstar}^{(Q)} \Bigr)
\nonumber \\ &
= \lim_{n \to \infty} \sum_{ x \in \mathcal{V} } \Pi_x^{(Q)} = 1,
\end{align}
where the last equality stems from the fact that $\Pi^{(Q)}$ is a probability distribution. Next, we have:
\begin{align}
\itr{\gamma_{\star}}{0}
&
= \lim_{n \to \infty} \Pi_{\Vstar}^{(Q)} = \lim_{n \to \infty} \sum_{ x \in \mathcal{V} } \Pi_{x}^{(Q)} Q_{x,\Vstar}
\eqcom{\ref{eqn:Definition_of_Qs_entries}}= \lim_{n \to \infty} \sum_{k=1}^K \sum_{ x \in \mathcal{W}_k } \bar{\Pi}_{k}^{(Q)} \frac{ q_{k,\star} }{ n }
\nonumber \\ &
= \lim_{n \to \infty} \sum_{k=1}^K \itr{\gamma_k}{0} \frac{ q_{k,\star} }{ n } 
= 0,
\label{eqn:Limit_of_gamma_star_0}
\end{align}
where the second equality stems from the global balance equations for $\Pi^{(Q)}$. Now we establish that the vector $\transpose{ ( \itr{\gamma_1}{0}, \ldots, \itr{\gamma_K}{0} ) }$ satisfies the balance equations $\vectInLineTransposed{ \itr{\gamma_1}{0}, \ldots, \itr{\gamma_K}{0} } p = \vectInLineTransposed{ \itr{\gamma_1}{0}, \ldots, \itr{\gamma_K}{0} }$. For $l = 1, \ldots, K$
\begin{align}
\itr{\gamma_l}{0}
&
= \lim_{n \to \infty} \cardinality{ \mathcal{W}_l } \bar{\Pi}_l^{(Q)}
= \lim_{n \to \infty} \sum_{ y \in \mathcal{W}_l } \Pi_y^{(Q)}
\eqcom{ii}= \lim_{n \to \infty} \sum_{ y \in \mathcal{W}_l } \sum_{ x \in \mathcal{V} } \Pi_x^{(Q)} Q_{x,y}
\label{eqn:Global_balance_for_gammal0}
\\ &
\eqcom{\ref{eqn:Definition_of_Qs_entries}}= \lim_{n \to \infty} \sum_{ y \in \mathcal{W}_l } \Bigl( \sum_{k = 1}^K \sum_{ x \in \mathcal{W}_k \backslash \{ y \} } \bar{\Pi}_k^{(Q)} \frac{ q_{k,l} }{ \cardinality{ \mathcal{W}_l } - \indicator{ k = l } } + \Pi_{\Vstar}^{(Q)} \frac{ q_{\star,l} }{ \cardinality{ \mathcal{W}_l } } \Bigr)
\nonumber \\ &
= \lim_{n \to \infty} \Bigl( \sum_{k = 1}^K \bigl( \cardinality{ \mathcal{W}_k} - \indicator{ k = l } \bigr) \bar{\Pi}_k^{(Q)} \frac{ \cardinality{ \mathcal{W}_l } }{ \cardinality{ \mathcal{W}_l } - \indicator{ k = l } } q_{k,l} + \Pi_{\Vstar}^{(Q)} q_{\star,l} \Bigr)
\nonumber \\ &
\eqcom{\ref{eqn:Coupling_of_the_parameters_within_q_to_p}}= \sum_{k=1}^K \itr{\gamma_k}{0} p_{k,l}.
\nonumber
\end{align}
This proves the first two assertions of the proposition. The proof of the third assertion is similar to that of \refEquation{eqn:Limit_of_gamma_star_0}. More precisely: 
\begin{align}
\itr{\gamma_{\star}}{1}
&
= \lim_{n \to \infty} n\Pi_{\Vstar}^{(Q)} = \lim_{n \to \infty} n\sum_{ x \in \mathcal{V} } \Pi_{x}^{(Q)} Q_{x,\Vstar}
\eqcom{\ref{eqn:Definition_of_Qs_entries}}= \lim_{n \to \infty} n \sum_{k=1}^K \sum_{ x \in \mathcal{W}_k } \bar{\Pi}_{k}^{(Q)} \frac{ q_{k,\star} }{ n }
\nonumber \\ &
=\sum_{k=1}^K \itr{\gamma_k}{0} q_{k,\star}.
\label{eqn:Global_balance_for_gammaStar1}
\end{align}
Together with the first assertion, this completes the proof.

\section{Proof of Lemma~\ref{lem:Existence_of_point_qbar}}
\label{suppl:Existence_of_point_qbar}

Define for $c \in \{ 1, \ldots, K \}$:
$
q_c 
= \bigl( 0, {p_{1,c}} / {\alpha_c}, \allowbreak \ldots, {p_{K,c}} / {\alpha_c}; \allowbreak 0, p_{c,1}, \ldots, p_{c,K} \bigr) 
\in \mathcal{Q}
$. Let $a \neq b$. The points $q_a$, $q_b$ satisfy: $I_a( q_a || p ) = I_b( q_b || p ) = 0$, and $0< I(\alpha,p) \leq I_a( q_b || p ) < \infty$ and $0< I(\alpha,p) \leq I_b( q_a || p ) < \infty$ (by definition of $I(\alpha,p)$). Consider the function $g:[0,1]\mapsto \mathbb{R}$ defined by 
$
g(\lambda) = I_a( \lambda q_a+(1-\lambda) q_b  || p ) - I_b( \lambda q_a+(1-\lambda) q_b  || p )
$.
This function $g$ is continuous and from the aforementioned properties of $q_a$ and $q_b$, we have $g(0)>0$ and $g(1)<0$. The intermediate value theorem implies that there exists $\lambda \in (0,1)$ such that $I_a( \lambda q_a + (1-\lambda) q_b || p )
= I_b( \lambda q_a + (1-\lambda) q_b || p )$.
Hence ${\cal Q}(a,b)\neq \emptyset$.

\section{Proof of Proposition~\ref{prop:Lower_bound_on_the_number_of_misclassifications_for_a_class_of_Q_transition_matrices}}
\label{suppl:Proof__Lower_bound_on_the_number_of_misclassifications_for_a_class_of_change_of_measures}

\revisedPartBegin

\paragraph{Proof of Proposition \ref{prop:Lower_bound_on_the_number_of_misclassifications_for_a_class_of_Q_transition_matrices} (i)}
Consider a fixed $n$. We have a \gls{BMC} with true parameters $\{ p_{k,l} \}$, and $\{ \cardinality{ \mathcal{V}_k } \}$. We pick $V^*$ uniformly at random from clusters $a \neq b$ and construct $Q$. Recall that $Q$ is therefore random and depends on $V^*$, $\{ p_{k,l} \}$, and $\{ \cardinality{ \mathcal{V}_k } \}$. We now give a lower bound on the probability that a $(\varepsilon,1)$-locally good algorithm at $(\alpha,p)$ that does not have access to  $\{ p_{k,l} \}$ and $\{ \cardinality{ \mathcal{V}_k } \}$ misclassifies $V^*$.

Consider such algorithm. Suppose we give the algorithm an infinitely long sample path generated under $Q$, i.e., $T \to \infty$. As a consequence, this algorithm has access to the elements $Q_{x,y}$ for all $x,y \in \mathcal{V}$ (from this it can determine which state is $V^*$). We furthermore give the algorithm the information that the two clusters from which we randomly picked are $a, b$. The algorithm must now consider the following two hypotheses:
\begin{enumerate}
\item[--] $\mathcal{H}_a = 
\bigl\{ 
\hat{\sigma}(V^*) = a, 
\hat{p}_{\sigma(x),\sigma(y)} = ( \cardinality{ \mathcal{V}_{\sigma(y)} } - \indicator{ \sigma(y) = a } - \indicator{ x \neq V^*, \sigma(x) = \sigma(y) } ) Q_{x,y} + {Q_{x,V^*}} / {K} \, \forall_{ x \neq y, y \neq V^* }
\bigr\}$, 
\item[--] $\mathcal{H}_b = 
\bigl\{ 
\hat{\sigma}(V^*) = b, 
\hat{p}_{\sigma(x),\sigma(y)} = ( \cardinality{ \mathcal{V}_{\sigma(y)} } - \indicator{ \sigma(y) = b } - \indicator{ x \neq V^*, \sigma(x) = \sigma(y) } ) Q_{x,y} + {Q_{x,V^*}} / {K} \, \forall_{ x \neq y, y \neq V^* }
\bigr\}$.
\end{enumerate}
Recall that because the algorithm is $(\varepsilon,1)$-locally good at $(\alpha,p)$, it is able to cluster well in both (slightly different) \glspl{BMC} in hypotheses $\mathcal{H}_a$ and $\mathcal{H}_b$. Note furthermore that both of these \glspl{BMC} give the exact same $Q$-matrix, regardless of whether truly $\sigma(V^*) = a$ or $b$. Therefore, since $\mathcal{H}_a$ held w.p.\ $\alpha_a / (\alpha_a + \alpha_b)$ and $\mathcal{H}_b$ otherwise, this $(\varepsilon,1)$-locally good algorithm at $(\alpha,p)$ will misclassify $V^*$ with probability at least
$
\probabilityWrt{ V^* \in \mathcal{E} }{ \Psi }
\geq ( \alpha_a \wedge \alpha_b ) / ( \alpha_a + \alpha_b )
\triangleq \delta 
> 0
$.
\revisedPartEnd

\paragraph{Proof of Proposition \ref{prop:Lower_bound_on_the_number_of_misclassifications_for_a_class_of_Q_transition_matrices} (ii)} 
Select a state $\Vstar$ uniformly at random from any two specific clusters $a, b \in \{ 1, \ldots, K \}$, $a \neq b$. We are going to bound
\begin{equation}
\probabilityWrt{ L \leq f(n,T) }{\Psi}
= \probabilityWrt{ L \leq f(n,T), \Vstar \in \mathcal{E} }{\Psi} + \probabilityWrt{ L \leq f(n,T), \Vstar \not\in \mathcal{E} }{\Psi}.
\label{eqn:Probability_of_L_leq_f_wrt_Q}
\end{equation}
for any function $f : \naturalNumbersPlus^2 \to \realNumbers$.

The first term of \refEquation{eqn:Probability_of_L_leq_f_wrt_Q} can be bounded using our change of measure formula \refEquation{eqn:Definition_of_log_likelihood_L}. Namely,
\begin{align}
\probabilityWrt{ L \leq f(n,T), \Vstar \in \mathcal{E} }{\Psi}
&
\eqcom{\ref{eqn:Definition_of_log_likelihood_L}}\leq \e{f(n,T)} \probabilityWrt{ L \leq f(n,T), \Vstar \in \mathcal{E} }{\Phi}
\nonumber \\ &
\leq \e{f(n,T)} \probabilityWrt{ \Vstar \in \mathcal{E} }{\Phi}.
\label{eqn:Intermediate_bound_on_event_L_less_f_and_vStar_being_classified_wrt_Q}
\end{align}
Because $\Vstar$ is selected from $\mathcal{V}_a \cup \mathcal{V}_b$ uniformly at random, we have by \refLemma{lem:Relation_between_unif_selection_from_two_clusters_vs_all_clusters}, see \refAppendixSection{sec:Properties_of_uniform_vertex_selection}, that for any $V$ selected uniformly at random from \emph{all} vertices $\mathcal{V}$,
\begin{equation}
\probabilityWrt{ \Vstar \in \mathcal{E} }{\Phi}
= \probabilityWrt{ V \in \mathcal{E} | V \in \mathcal{V}_a \cup \mathcal{V}_b }{\Phi}
= \frac{ \probabilityWrt{ V \in \mathcal{E}, V \in \mathcal{V}_a \cup \mathcal{V}_b }{\Phi} }{ \probabilityWrt{ V \in \mathcal{V}_a \cup \mathcal{V}_b }{\Phi} }
\leq \frac{ \probabilityWrt{ V \in \mathcal{E} }{\Phi} }{ \alpha_a + \alpha_b }.
\end{equation}
Subsequently by \refLemma{lem:Relation_between_expected_nr_of_misclassified_vertices_and_probability_of_a_uniformly_selected_vertex_being_misclassified}, see \refAppendixSection{sec:Properties_of_uniform_vertex_selection},
\begin{equation}
\probabilityWrt{ \Vstar \in \mathcal{E} }{\Phi} 
\leq \frac{ \expectationWrt{ \cardinality{\mathcal{E}} }{\Phi} }{ ( \alpha_a + \alpha_b ) n }.
\label{eqn:Bound_on_vStar_being_classified_wrt_P}
\end{equation}
Substituting \refEquation{eqn:Bound_on_vStar_being_classified_wrt_P} into \refEquation{eqn:Intermediate_bound_on_event_L_less_f_and_vStar_being_classified_wrt_Q}, we obtain
\begin{equation}
\probabilityWrt{ L \leq f(n,T), \Vstar \in \mathcal{E} }{\Psi}
\leq \e{f(n,T)} \frac{ \expectationWrt{ \cardinality{\mathcal{E}} }{\Phi} }{ ( \alpha_a + \alpha_b ) n }.
\label{eqn:Intermediate_bound_on_the_first_term}
\end{equation}
The second term of \refEquation{eqn:Probability_of_L_leq_f_wrt_Q} can be bounded using Proposition \ref{prop:Lower_bound_on_the_number_of_misclassifications_for_a_class_of_Q_transition_matrices} (i):
\begin{equation}
\probabilityWrt{ L \leq f(n,T), \Vstar \not\in \mathcal{E} }{\Psi}
\leq \probabilityWrt{ \Vstar \not\in \mathcal{E} }{\Psi}
= 1 - \probabilityWrt{ \Vstar \in \mathcal{E} }{\Psi} 
\leq 1 - \delta < 1. 
\label{eqn:Intermediate_bound_on_the_second_term} 
\end{equation}
Now using \refEquation{eqn:Intermediate_bound_on_the_first_term} and \refEquation{eqn:Intermediate_bound_on_the_second_term} to bound \refEquation{eqn:Probability_of_L_leq_f_wrt_Q}, we arrive at
\begin{equation}
\probabilityWrt{ L \leq f(n,T) }{\Psi}
\leq \e{f(n,T)} \frac{ \expectationWrt{ \cardinality{\mathcal{E}} }{\Phi} }{ ( \alpha_a + \alpha_b ) n } + 1 - \delta.
\label{eqn:Bound_on_probability_of_L_leq_f_wrt_Q}
\end{equation}

We now prepare for an application of Chebyshev's inequality. First note using \refEquation{eqn:Bound_on_probability_of_L_leq_f_wrt_Q} that
\begin{equation}
\probabilityWrt{ L \geq f(n,T) }{\Psi}
= 1 - \probabilityWrt{ L \leq f(n,T) }{\Psi}
\geq \delta - \e{f(n,T)} \frac{ \expectationWrt{ \cardinality{\mathcal{E}} }{\Phi} }{ ( \alpha_a + \alpha_b ) n }.
\label{eqn:Lower_bound_on_the_probability_of_L_geq_f_wrt_Q}
\end{equation}
Specify $f(n,T) = \ln{\bigl( \delta/2 \bigr)} + \ln{ \bigl( (\alpha_a + \alpha_b) n / \expectationWrt{ \cardinality{\mathcal{E}} }{\Phi} \bigr) }$, so that
\begin{equation}
\probabilityBigWrt{ L \geq \ln{ \frac{ \delta }{2} } + \ln{ \frac{(\alpha_a + \alpha_b) n }{ \expectationWrt{ \cardinality{\mathcal{E}} }{\Phi} } } }{\Psi}
\geq \frac{\delta}{2}.
\end{equation}
Since $\delta > 0$, we can apply Chebyshev's inequality and \refEquation{eqn:Lower_bound_on_the_probability_of_L_geq_f_wrt_Q} to conclude
\begin{equation}
\probabilityBigWrt{ L \geq \expectationWrt{L}{\Psi} + \sqrt{ \frac{2}{\delta} } \sqrt{ \varianceWrt{L}{\Psi} } }{\Psi}
\leq \frac{\delta}{2} 
\leq \probabilityBigWrt{ L \geq \ln{ \frac{ \delta }{2} } + \ln{ \frac{(\alpha_a + \alpha_b) n }{ \expectationWrt{ \cardinality{\mathcal{E}} }{\Phi} } } }{\Psi}.
\end{equation}
Comparing the events in the l.h.s. and r.h.s. of the above inequalities, we then must have
\begin{equation}
\ln{ \frac{ \delta }{2} } + \ln{ \frac{(\alpha_a + \alpha_b) n }{ \expectationWrt{ \cardinality{\mathcal{E}} }{\Phi} } }
\leq \expectationWrt{L}{\Psi} + \sqrt{ \frac{2}{\delta} } \sqrt{ \varianceWrt{L}{\Psi} }.
\end{equation}
Rearranging gives \refEquation{eqn:Change_of_measure_lower_bound} with $C = (\alpha_a + \alpha_b ) \delta / 2 > 0$. This completes the proof.

\section{Proof of Proposition~\ref{prop:Leading_order_behavior_of_expectation_of_L_wrt_Q}}
\label{suppl:Proof__Leading_order_behavior_of_expectationLsigmaVstarWrtQ}

Define $R_{x,y} \triangleq \ln{ ( Q_{x,y} / P_{x,y} ) }$ for notational convenience: we refer to \refAppendixSection{sec:Asymptotic_comparisons_between_P_and_Qs_entries} for its asymptotic behavior. Since the Markov chain is started at equilibrium,
\begin{equation}
\expectationWrt{ L | \sigma(\Vstar) }{\Psi}
= T \sum_{x \in \mathcal{V}} \sum_{y \in \mathcal{V}} \Pi_{x}^{(Q)} Q_{x,y} \ln{R_{x,y}}.
\label{eqn:Expectation_of_L_in_terms_of_equilibrium_jumps}
\end{equation}
The largest individual contributions to the expectation in \refEquation{eqn:Expectation_of_L_in_terms_of_equilibrium_jumps} are by jumps to and from $\Vstar$, since this is where the change of measure is modified most. Jumps not involving $\Vstar$ contribute less individually, but there are many of such jumps. We therefore separate out the jumps to and from $\Vstar$, i.e.,
\begin{align}
\frac{ \expectationWrt{L | \sigma(\Vstar) }{\Psi} }{T}
&
= \sum_{ y \neq \Vstar } \Pi^{(Q)}_{\Vstar} Q_{\Vstar,y} \ln{ R_{\Vstar,y} }
+ \sum_{ x \neq \Vstar } \Pi^{(Q)}_x Q_{x,\Vstar} \ln{ R_{x,\Vstar} } 
\nonumber \\ &
\phantom{=} + \sum_{ x,y \neq \Vstar } \Pi_x^{(Q)} Q_{x,y} \ln{ R_{x,y} }.
\label{eqn:EQL_given_Vstar_asymptotic_calculation__Termsplit}
\end{align}
We now calculate the leading order behavior of each term.

For the first term in \refEquation{eqn:EQL_given_Vstar_asymptotic_calculation__Termsplit} we have by (i) \refLemma{lem:Leading_order_behavior_of_Qxy_Pxy} in \refAppendixSection{sec:Asymptotic_comparisons_between_P_and_Qs_entries} and $Q$'s definition, see \refEquation{eqn:Definition_of_Qs_entries}, and (ii) \refProposition{prop:Equilibrium_behavior_of_piQ},
\begin{align}
\sum_{ y \neq \Vstar } \Pi^{(Q)}_{\Vstar} Q_{\Vstar,y} \ln{ R_{\Vstar,y} }
&
\eqcom{i}\sim \sum_{k=1}^K \sum_{ y \in \mathcal{W}_k } \Pi^{(Q)}_{\Vstar} \frac{ q_{\star,k} }{ \cardinality{ \mathcal{W}_k } } \ln{ \frac{ q_{\star,\omega(y)} }{ p_{\sigma(\Vstar),\omega(y)} } }
\nonumber \\ &
\eqcom{ii}\sim \frac{1}{n} \sum_{k=1}^K \itr{\gamma_\star}{1} q_{\star,k} \ln{ \frac{ q_{\star,k} }{ p_{\sigma(\Vstar),k} } }.
\label{eqn:EQL_given_Vstar_asymptotic_calculation__Term1}
\end{align}
The second term in \refEquation{eqn:EQL_given_Vstar_asymptotic_calculation__Termsplit} handles similarly:
\begin{align}
\sum_{ x \neq \Vstar } \Pi^{(Q)}_x Q_{x,\Vstar} \ln{ R_{x,\Vstar} } 
&
\eqcom{i}\sim \sum_{k=1}^K \sum_{ x \in \mathcal{W}_k \backslash \{ \Vstar \} } \bar{\Pi}_k^{(Q)} \frac{ q_{k,\star} }{ n } \ln{ \frac{ q_{k,\star} \alpha_{\sigma(\Vstar)} }{ p_{k,\sigma(\Vstar)} } }
\nonumber \\ &
\eqcom{ii}\sim \frac{1}{n} \sum_{k=1}^K \itr{\gamma_k}{0} q_{k,\star} \ln{ \frac{ q_{k,\star} \alpha_{\sigma(\Vstar)} }{ p_{k,\sigma(\Vstar)} } }.
\label{eqn:EQL_given_Vstar_asymptotic_calculation__Term2}
\end{align}
The third term in \refEquation{eqn:EQL_given_Vstar_asymptotic_calculation__Termsplit} requires (iii) a Taylor expansion of $\ln{(1+x)} = x + \bigO{x^2}$ for $x \approx 0$ and (iv) the balance equations \refEquation{eqn:Global_balance_for_gammal0}--\refEquation{eqn:Global_balance_for_gammaStar1}, so that
\begin{align}
\sum_{ x,y \neq \Vstar } \Pi_x^{(Q)} Q_{x,y} \ln{ R_{x,y} }
&
\eqcom{iii}\sim \sum_{k,l \neq \star} \sum_{ x \in \mathcal{W}_k } \sum_{ y \in \mathcal{W}_l \backslash \{ x \} } \bar{\Pi}_k^{(Q)} \frac{ q_{k,l} }{ \cardinality{ \mathcal{W}_{l} } - \indicator{ k = l } } \times
\nonumber \\
\cdots \times \frac{1}{n} \Bigl( \frac{ \indicator{ l = \sigma(\Vstar) } }{ \alpha_{l} } - \frac{ q_{k,\star} }{ p_{k,l} K } \Bigr)
&
\eqcom{\ref{eqn:Coupling_of_the_parameters_within_q_to_p}} \sim \frac{1}{n} \sum_{k=1}^K \itr{\gamma_k}{0} \Bigl( \frac{ q_{k,\sigma(\Vstar)} }{ \alpha_{\sigma(\Vstar)} } - \sum_{l=1}^K \frac{1}{K} q_{k,\star} \Bigr)
\nonumber \\ &
\eqcom{iv}= \frac{1}{n} \Bigl( \frac{ \itr{\gamma_{\sigma(\Vstar)}}{0} }{ \alpha_{\sigma(\Vstar)} } - \itr{\gamma_\star}{1} \Bigr).
\label{eqn:EQL_given_Vstar_asymptotic_calculation__Term3}
\end{align}

Substituting \refEquation{eqn:EQL_given_Vstar_asymptotic_calculation__Term1}--\refEquation{eqn:EQL_given_Vstar_asymptotic_calculation__Term3} into \refEquation{eqn:EQL_given_Vstar_asymptotic_calculation__Termsplit} gives
\begin{align}
\expectationWrt{ L | \sigma(\Vstar) }{\Psi}
&
\sim \frac{T}{n} \sum_{k=1}^K \Bigl( \itr{\gamma_\star}{1} q_{\star,k} \ln{ \frac{ q_{\star,k} }{ p_{\sigma(\Vstar),k} } } + \itr{\gamma_k}{0} q_{k,\star} \ln{ \frac{ q_{k,\star} \alpha_{\sigma(\Vstar)} }{ p_{k,\sigma(\Vstar)} } } \Bigr) 
\nonumber \\ &
\phantom{\sim} + \frac{T}{n} \Bigl( \frac{ \itr{\gamma_{\sigma(\Vstar)}}{0} }{ \alpha_{\sigma(\Vstar)} } - \itr{\gamma_\star}{1} \Bigr).
\end{align}
By now applying \refProposition{prop:Equilibrium_behavior_of_piQ}, we complete the proof.

\section{Proof of Proposition~\ref{prop:Variance_is_small}}
\label{suppl:Proof__Asymptotic_negligibilit_of_VarLwrtQ}

Define $L_t \triangleq \ln{ ( Q_{X_{t-1},X_t} / P_{X_{t-1},X_t} ) }$. Expanding, we obtain 
\begin{equation}
\varianceWrt{L | \sigma(\Vstar) }{\Psi} 
= \varianceBigWrt{ \sum_{t=1}^T L_t \Big| \sigma(\Vstar) }{\Psi}
= \sum_{t=1}^T \sum_{s=1}^T \covarianceWrt{L_t}{L_s | \sigma(\Vstar) }{\Psi} .
\label{eqn:Variance_of_L_wrt_VstarQ_split_into_time_covariances}
\end{equation}

We now consider the cases $|t-s| \geq 2$ and $|t-s| \leq 1$, in that order. Since there are only $\bigO{T}$ terms corresponding to $|t-s| \leq 1$ that contribute to the sum, we only provide crude bounds on these terms. On the contrary, there are as many as $\bigO{T^2}$ terms corresponding to $|t-s| \geq 2$, we will need sharper bounds for these terms. As we will show for the  cases where $|t-s| \geq 2$, we can derive a sharper bound when $|t-s| \gg t_{\mathrm{mix}}(\varepsilon)$ because \refProposition{prop:Mixing_times_of_Markov_chains_with_transition_matrices_P_and_Q} implies that the covariances decay quickly.

First note that since (i) the process is started from equilibrium, we have for any $t, s \in \{ 1, \ldots, T \}$ that
\begin{align}
\covarianceWrt{ L_t }{ L_s | \sigma(\Vstar) }{\Psi} 
&
= \expectationWrt{L_t L_s | \sigma(\Vstar) }{\Psi} - \expectationWrt{L_t | \sigma(\Vstar) }{\Psi} \expectationWrt{L_s | \sigma(\Vstar) }{\Psi}
\nonumber \\ &
\eqcom{i}= \expectationWrt{ L_t L_s | \sigma(\Vstar) }{\Psi} - \expectationWrt{ L_t | \sigma(\Vstar) }{\Psi}^2.
\label{eqn:Covariance_of_LtLs_when_the_process_is_started_from_equilibrium}
\end{align}

Consider the case $|t-s| \geq 2$. Define $S_{x,y,u,v} \triangleq ( \ln{ R_{x,y} } ) ( \ln{ R_{u,v} } )$ for notational convenience. Refer to \refAppendixSection{sec:Asymptotic_comparisons_between_P_and_Qs_entries} for its asymptotic behavior. In this case the first term of \refEquation{eqn:Covariance_of_LtLs_when_the_process_is_started_from_equilibrium} evaluates as
\begin{align}
&
\expectationWrt{ L_t L_s | \sigma(\Vstar) }{\Psi}
\label{eqn:Expectation_of_LtLs_when_the_process_is_started_from_equilibrium}
\\
= & \sum_{x,y,u,v} \probabilityWrt{ X_{t \wedge s - 1} = x, X_{t \wedge s} = y, X_{t \vee s - 1 } = u, X_{t \vee s} = v | \sigma(\Vstar) }{\Psi} S_{x,y,u,v}
\nonumber \\
= & \sum_{x,y,u,v} \Pi_x^{(Q)} Q_{x,y} \Bigl( \sum_{ z_{ t \wedge s + 1 }, \ldots, z_{ t \vee s - 2 } } Q_{ y, z_{ t \wedge s + 1 } } Q_{ z_{ t \wedge s + 2 }, z_{ t \wedge s + 2} } \cdots Q_{z_{t \vee s - 2},u} \Bigr) Q_{u,v} S_{x,y,u,v}
\nonumber \\
= & \sum_{x,y,u,v} \Pi_x^{(Q)} Q_{x,y} Q^{ |t-s| - 1 }_{y,u} Q_{u,v} S_{x,y,u,v}.
\end{align}
The second term of \refEquation{eqn:Covariance_of_LtLs_when_the_process_is_started_from_equilibrium} expands as
\begin{equation}
\expectationWrt{ L_t | \sigma(\Vstar) }{\Psi}^2
= \Bigl( \sum_{x,y} \Pi_x^{(Q)} Q_{x,y} \ln{ \frac{ Q_{x,y} }{ P_{x,y} } } \Bigr)^2
= \sum_{x,y,u,v} \Pi_x^{(Q)} Q_{x,y} \Pi_u^{(Q)} Q_{u,v} S_{x,y,u,v}.
\label{eqn:Squared_expectation_of_Lt__when_the_process_is_started_from_equilibrium}
\end{equation}
Substituting \refEquation{eqn:Expectation_of_LtLs_when_the_process_is_started_from_equilibrium} and \refEquation{eqn:Squared_expectation_of_Lt__when_the_process_is_started_from_equilibrium} into the last member of \refEquation{eqn:Covariance_of_LtLs_when_the_process_is_started_from_equilibrium} gives
\begin{equation}
\covarianceWrt{L_t}{L_s | \sigma(\Vstar) }{\Psi}
= \sum_{x,y,u,v} \Pi_x^{(Q)} Q_{x,y} \bigl( Q^{ |t-s| - 1 }_{y,u} - \Pi_u^{(Q)} \bigr) Q_{u,v} S_{x,y,u,v}.
\label{eqn:Centered_covariance_term}
\end{equation}

In order to bound \refEquation{eqn:Centered_covariance_term}, we need to take two effects into consideration: a filter effect that happens because the transition matrix $Q$ is similar to the transition matrix $P$, and a concentration effect because the Markov chain moves closer to equilibrium as time progresses. The filter effect is quantified by \refCorollary{cor:Leading_order_behavior_of_ln_Qxy_Pxy_products} in \refAppendixSection{sec:Asymptotic_comparisons_between_P_and_Qs_entries}. The latter implies that $\sum_{x,y,u,v} S_{x,y,u,v} \leq c_1 n^2$ for some absolute constant $c_1$ (even though $\sum_{x,y,u,v} 1 = n^4$). We can use the effect by for example bounding $\Pi_x^{(Q)} Q_{x,y} \bigl( Q^{m}_{y,u} - \Pi_u^{(Q)} \bigr) Q_{u,v} \leq c_2 / n^4$ uniformly using another absolute constant, and then concluding that $\covarianceWrt{L_t}{L_s | \sigma(\Vstar) }{\Psi} \leq c_2 (T^2/n^4) \sum_{x,y,u,v} S_{x,y,u,v} \leq c_1 c_2 T^2 / n^2$. However, this bound is not sufficiently sharp for our purposes: we need to provide a bound that is at least $\smallO{ T^2 / n^2 }$.

To arrive at a sharper bound, we use the concentration of the Markov chain. Apply the triangle inequality first, and then bound $\Pi_x^{(Q)} Q_{x,y} Q_{u,v} \leq c_1 / n^3$ uniformly using an absolute constant $c_1$ to obtain
\begin{align}
&
\Bigl| \sum_{t=1}^T \sum_{s=1}^T \indicator{ |t-s| \geq 2 } \covarianceWrt{L_t}{L_s | \sigma(\Vstar) }{\Psi} \Bigr|
\nonumber \\ &
\leq \frac{2c_1}{n^3} \sum_{t=1}^T \sum_{s=t+2}^T \sum_{x,y,u,v} \bigl| Q^{ |t-s| - 1 }_{y,u} - \Pi_u^{(Q)} \bigr| | S_{x,y,u,v} |.
\label{eqn:First_bound_on_sum_st_of_covariances}
\end{align}
Now let $m \in \naturalNumbersPlus$. By nonnegativity of the summands and \refEquation{eqn:Definition_of_total_variation_distance}, $| Q^m_{x,y} - \Pi_y^{(Q)} | \leq \sum_{y} | Q^m_{x,y} - \Pi_y^{(Q)} | = 2 d_{\mathrm{TV}}( {Q}^m_{x,\cdot}, \vect{\Pi}^{(Q)} )$. 
Recall furthermore from \refEquation{eqn:Upper_bound_on_dTV_Pt_Pi_distance_measure_as_function_of_t} and \refEquation{eqn:Upper_bound_on_dTV_Pt_Pi_distance_measure_as_function_of_varepsilon} combined that there exists a $\delta(Q) \in (0,1)$ such that $d_{\mathrm{TV}}( Q_{u,\cdot}^m, \Pi ) \leq ( \delta(Q) )^{m} d_{\mathrm{TV}}( Q_{u,\cdot}^0, \Pi )$ for $u \in \mathcal{V}$. We therefore have that there exists an absolute constant $c_3$ such that
\begin{align}
&
\Bigl| \sum_{t=1}^T \sum_{s=1}^T \indicator{ |t-s| \geq 2 } \covarianceWrt{L_t}{L_s | \sigma(\Vstar) }{\Psi} \Bigr|
\nonumber \\ &
\leq \frac{4 c_1 c_2 \max_{ u \in \mathcal{V} } \{ d_{\mathrm{TV}}( Q_{u,\cdot}^0, \Pi ) \} }{n^3} \sum_{t=1}^T \sum_{s=t+2}^T \bigl( \delta(Q) \bigr)^{ |t-s| - 1 } \sum_{x,y,u,v} | S_{x,y,u,v} |
\nonumber \\ &
\eqcom{i}\leq \frac{c_3 }{n} \sum_{t=1}^T \sum_{s=t+2}^T \bigl( \delta(Q) \bigr)^{ |t-s| - 1 },
\label{eqn:Second_bound_on_sum_st_of_covariances}
\end{align}
due to (i) the filter effect. By a continuous extension of the sum,
\begin{align}
\sum_{t=1}^T \sum_{s=t+2}^T \bigl( \delta(Q) \bigr)^{ |t-s| - 1 }
&
\leq \int_0^T \int_{t+1}^T \bigl( \delta(Q) \bigr)^{ |t-s| - 1 } \d{s} \d{t}
\sim - \frac{T}{ \ln{ \bigl( \delta(Q) \bigr) } }.
\end{align}
Since $\delta(Q) \in (0,1)$, there thus exists an absolute constant $c_4 > 0$ such that
\begin{equation}
\Bigl| \sum_{t=1}^T \sum_{s=1}^T \indicator{ |t-s| \geq 2 } \covarianceWrt{L_t}{L_s | \sigma(\Vstar) }{\Psi} \Bigr|
\leq c_4 \frac{T}{n}.
\label{eqn:Obound_on_covariance_terms_case_t_min_s_geq_2}
\end{equation}

Finally we deal with the cases $|t-s| \leq 1$. When $|t-s| = 0$, or equivalently $t=s$, we have that (iv) because of \refLemma{lem:Leading_order_behavior_of_Qxy_Pxy} and \refCorollary{cor:Leading_order_behavior_of_ln_Qxy_Pxy_products} that there exist absolute constants $c_5, \ldots, c_8$ such that
\begin{align}
&
\covarianceWrt{L_t}{L_t | \sigma(\Vstar) }{\Psi}
\eqcom{\ref{eqn:Covariance_of_LtLs_when_the_process_is_started_from_equilibrium}}\leq \expectationWrt{L_t^2 | \sigma(\Vstar) }{\Psi}
= \sum_{x \in \mathcal{V}} \sum_{y \in \mathcal{V}} \Pi_x^{(Q)} Q_{x,y} ( \ln{R_{x,y}} )^2
\\ &
\eqcom{iv}\leq \Pi_{\Vstar}^{(Q)} \sum_{ y \neq \Vstar } Q_{\Vstar,y} c_5 + \sum_{ x \neq \Vstar } \Pi_x^{(Q)} Q_{x,\Vstar} c_6 + \sum_{x \neq \Vstar} \sum_{y \neq \Vstar} \Pi_x^{(Q)} Q_{x,y} \frac{c_7}{n^2}
\leq \frac{c_8}{n}
\nonumber
\end{align}
for all $t = 1, \ldots, T$. Therefore
\begin{equation}
\Bigl| \sum_{t=1}^T \sum_{s=1}^T \indicator{ |t-s| = 0 } \covarianceWrt{L_t}{L_s | \sigma(\Vstar) }{\Psi} \Bigr| 
= \sum_{t=1}^T \varianceWrt{L_t | \sigma(\Vstar) }{\Psi} 
= \bigObig{ \frac{T}{n} }.
\label{eqn:Obound_on_covariance_terms_case_t_min_s_eq_0}
\end{equation}
When $|t-s|=1$, there exists an absolute constant $c_{9} > 0$ such that
\begin{align}
&
\covarianceWrt{ L_{ t \wedge s } }{ L_{ t \wedge s + 1 } | \sigma(\Vstar) }{\Psi}
\eqcom{\ref{eqn:Covariance_of_LtLs_when_the_process_is_started_from_equilibrium}}\leq \expectationWrt{ L_{ t \wedge s } L_{ t \wedge s + 1 } | \sigma(\Vstar) }{\Psi}
\nonumber \\ &
\leq \sum_{x,y,z} \Pi_x^{(Q)} Q_{x,y} Q_{y,z} S_{x,y,y,z}
\leq \frac{ c_{9} }{n^3} \sum_{x,y,z} S_{x,y,y,z}.
\end{align}
Invoking \refCorollary{cor:Leading_order_behavior_of_ln_Qxy_Pxy_products}'s filter effect implies that $\sum_{x,y,z} S_{x,y,y,z} = \bigO{n^2}$. Therefore
\begin{equation}
\Bigl| \sum_{s,t=1}^T \indicator{ |t-s| = 1 } \covarianceWrt{L_t}{L_s | \sigma(\Vstar) }{\Psi} \Bigr| 
\leq 2 \sum_{t=1}^T | \covarianceWrt{L_t}{ L_{t+1} | \sigma(\Vstar) }{\Psi} |
= \bigObig{ \frac{T}{n} }.
\label{eqn:Obound_on_covariance_terms_case_t_min_s_eq_1}
\end{equation}

Splitting \refEquation{eqn:Variance_of_L_wrt_VstarQ_split_into_time_covariances} into the respective cases and then (v) substituting \refEquation{eqn:Obound_on_covariance_terms_case_t_min_s_geq_2}, \refEquation{eqn:Obound_on_covariance_terms_case_t_min_s_eq_0}, and \refEquation{eqn:Obound_on_covariance_terms_case_t_min_s_eq_1} gives
\begin{align}
\varianceWrt{ L | \sigma(\Vstar) }{\Psi} 
= \sum_{t=1}^T \sum_{s=1}^T \bigl( & \indicator{ t=s } + \indicator{ |t-s| = 1 } 
\nonumber \\ &
+ \indicator{ |t-s| \geq 2 } \bigr) \covarianceWrt{L_t}{L_s | \sigma(\Vstar) }{\Psi}
\eqcom{v}= \jaron{ \bigObig{ \frac{T}{n} } },
\end{align}
which completes the proof.

\section{Proof of Corollary~\ref{lem:Specification_of_lower_bound_before_bound_optimization}}
\label{suppl:Proof__Deconditioning}

Let $a \neq b$ be any two distinct clusters, and let $q \in Q(a,b)$ (i.e., such that $I_a({q}||p) = I_b({q}||p)$. Select $\Vstar$ uniformly at random in $\mathcal{V}_a \cup \mathcal{V}_b$. Then we have: 
\begin{align}
\expectationWrt{L}{\Psi} 
&
= \frac{\alpha_a}{ \alpha_a + \alpha_b } \expectationWrt{L | \sigma(\Vstar) = a }{\Psi} + \frac{\alpha_b}{ \alpha_a + \alpha_b } \expectationWrt{L | \sigma(\Vstar) = b }{\Psi}
\nonumber \\ &
= \frac{T}{n} I_{a,b}(p) + \smallObig{ \frac{T}{n} },
\end{align}
where the last equality results from \refProposition{prop:Leading_order_behavior_of_expectation_of_L_wrt_Q}.
Moreover, since for any random variables $X$ and $Y$, $\variance{X} = \expectationWrt{ \variance{X|Y} }{Y} + \varianceWrt{ \expectation{X|Y} }{Y}$, we also have:
\begin{align}
\varianceWrt{L}{\Psi}
= & \sum_{k \in \{a,b\}} \frac{\alpha_k}{ \alpha_a + \alpha_b } \Bigl( \varianceWrt{L | \sigma(\Vstar) = k }{\Psi} + \bigl( \expectationWrt{ L | \sigma(\Vstar) = k }{\Psi} - \expectationWrt{L}{\Psi} \bigr)^2 \Bigr).
\nonumber
\end{align}

We deduce that $\varianceWrt{L}{\Psi}= \smallO{ T^2 / n^2 }$ from  \refProposition{prop:Variance_is_small} for the first two terms, and from the fact that ${q} \in Q(a,b)$ for the last two terms.

\section{Proof of Lemma~\ref{lem:Jalphap_is_bounded_by_Ialphap}}
\label{sec:Proof__Upper_bounding_Jalphap_by_Ialphap}

\def\sigmaVstar{*}
\let\star\sigmaVstar

Recall the definition of $q_c$ for $c = 1, \ldots, K$:
$$
q_c  = \Bigl( 0,\frac{p_{1,c}}{\alpha_c}, \ldots, \frac{p_{K,c}}{\alpha_c}; 0,p_{c,1}, \ldots, p_{c,K} \Bigr) 
\in \mathcal{Q}.
$$

Let $a^\star$ and $b^\star$ be the cluster indices such that
\begin{equation}
I(\alpha,p) 
= \sum_{k=1}^K \frac{1}{\alpha_{a^\star}} \Bigl( \pi_{a^\star} p_{a^\star , k}\ln \frac{p_{a^\star , k}}{p_{b^\star , k}} + \pi_k p_{k, a^\star } \ln{ \frac{p_{k, a^\star }}{p_{k,b^\star }} } 
+ \Bigl( \frac{\pi_{b^\star}}{\alpha_{b^\star}} - \frac{\pi_{a^\star}}{\alpha_{a^\star}} \Bigr) \Bigr).
\end{equation}
From the definitions of $I(\alpha, p)$, $I_c(q\|p)$, and $q_c$, we have that (i) $I_{a^\star} (q_{a^\star} \| p) = 0$, $I_{a^\star}(q_{b^\star}\| p) = I(\alpha,p)$, and (ii) $I_{b^\star}( q_{b^\star} \| p ) = 0$, $I_{b^\star} (q_{a^\star} \| p) \geq I(\alpha,p)$.

We are now going to show that there exists a path along which $I_{a^\star} (q\| p)$ monotonically decreases from $I(\alpha, p)$ to $0$, while at the same time the $I_{b^\star} (q\| p)$ moves from initially $0$ to eventually $I_{b^\star} (q_{a^\star}\| p) \geq I(\alpha,p)$. Since $I_c(q \| p)$ is continuous in $q$, this implies the existence of atleast one point $\bar{q}$ such that $0\le I_{a^\star} (\bar{q}\| p) = I_{b^\star} (\bar{q}\| p) \leq I(\alpha,p)$.

First, we will walk along the path
\begin{align}
q^{(1)}(\lambda) 
&
= (1-\lambda) \Bigl( \frac{p_{1,b^\star}}{\alpha_{b^\star}},\dots,\frac{p_{K,b^\star}}{\alpha_{b^\star}};p_{b^\star,1},\dots,p_{b^\star,K};0 \Bigr)
\nonumber \\ &
\phantom{=} + \lambda \Bigl( \frac{p_{1,b^\star}}{\alpha_{b^\star}},\dots,\frac{p_{K,b^\star}}{\alpha_{b^\star}};p_{a^\star,1},\dots,p_{a^\star,K};0 \Bigr). 
\end{align} 
parameterized by $\lambda \in [0,1]$. Specifically note that $I_{a^\star} (q^{(1)}(\lambda)\| p)$ is convex and monotonically decreasing in $\lambda$. This is because $\lambda$ only changes the convex summands $(\sum_{k=1}^K \pi_l q_{l,0} ) \mathrm{KL}( q_{0,\cdot} \| p_{c,\cdot})$ in \refEquation{eqn:Definition_of_Icqp}, and additionally, $\mathrm{KL}( q_{0,\cdot} \| p_{c,\cdot})$ is minimized at $\lambda = 1$.

Next, starting from the end point $q^{(1)}(1)$, we will walk along the path
\begin{align*}
q^{(2)}(\eta) 
&
= (1-\eta) \Bigl( \frac{p_{1,b^\star}}{\alpha_{b^\star}},\dots,\frac{p_{K,b^\star}}{\alpha_{b^\star}};p_{a^\star,1},\dots,p_{a^\star,K};0 \Bigr)
\nonumber \\ &
\phantom{=} + \eta \Bigl( \frac{p_{1,a^\star}}{\alpha_{a^\star}},\dots,\frac{p_{K,a^\star}}{\alpha_{a^\star}};p_{a^\star,1},\dots,p_{a^\star,K};0 \Bigr)
\end{align*} 
parameterized by $\eta \in [0,1]$. Similar to before $I_{a^\star} (q^{(2)}(\lambda)\| p)$ is convex and monotonically decreasing in $\eta$, and tends to $0$ as $\eta \to 1$. Note that we have that $I_{b^\star}(q^{(1)}(0)\| p) = 0$ and $I_{b^\star} (q^{(2)}(1)\| p) = I_{b^\star} (q_{a^\star}\| p) \geq I(\alpha,p)$.

\section{Proof of Lemma~\ref{lem:noinformation}}
\label{sec:Proof_lem1}

Recall that
\begin{equation}
\expectationWrt{ L | \sigma(\Vstar) }{\Psi} 
= \sum_{ \textrm{all sample paths } \chi } \probabilityWrt{ \chi | \sigma(\Vstar) }{\Psi} \ln{ \frac{ \probabilityWrt{ \chi | \sigma(\Vstar) }{\Psi} }{ \probabilityWrt{\chi}{P} } }
\end{equation}
is a KL-divergence. As a consequence, $\expectationWrt{L}{\Psi} = 0$ if and only if
\begin{equation}
\probabilityWrt{ \chi | \sigma(\Vstar) }{\Psi} 
= \prod_{t=1}^T Q_{x_{t-1},x_t} 
= \prod_{t=1}^T P_{x_{t-1},x_t} 
= \probabilityWrt{\chi}{P}
\quad
\textrm{for all sample paths }
\chi.
\end{equation}
Equivalently $\expectationWrt{ L | \sigma(\Vstar) }{\Psi} = 0$ if and only if $Q_{x,y} = P_{x,y}$ for all $x,y \in \mathcal{V}$, which can be seen by considering the set of paths that disagree only on the last jump. Since
\begin{equation}
I_{\sigma(\Vstar)}(q||p) 
= \lim_{n \to \infty} \frac{n}{T} \Bigl( \expectationWrt{ L | \sigma(\Vstar) }{\Psi} + \smallO{1} \Bigr),
\end{equation}
we obtain that $I_{\sigma(\Vstar)}(q||p) = 0$ if and only if $q = q_{\sigma(\Vstar)}$. Since there exists $\bar{q}$ such that $I_{a^\star} (\bar{q}\| p) = I_{b^\star} (\bar{q}\| p) = 0$, $q_{a^\star} = q_{b^\star}$. This completes the proof. \qed

\def\sigmaVstar{0}
\let\star\sigmaVstar

\chapter{Proofs of \refChapter{sec:The_SVD_clustering_algorithm}}

\section{Proof of Proposition~\ref{prop:Spectral_concentration_bound_for_BMCs}}
\label{sec:Spectral_analysis_for_BMCs}

Let $\mathrm{diag}(\Pi) \in [0,1]^{n \times n}$ denote the matrix whose diagonal entries correspond to the entries of $\Pi$. Then by \refEquation{eqn:Definitions_of_Nhatxy_and_Nxy}, \jaron{$N \triangleq \expectation{ \hat{N} } = T \mathrm{diag}(\Pi) P$}.

The Markov process $\process{ X_t }{ 0 \leq t \leq T}$ generates $\hat{N}$. We can think of $\hat{N}$ as a sum of random matrices $\hat{N} = \sum_{t=0}^{T-1} \hat{N}(t)$, where all the elements of each $t$-th matrix $\hat{N}(t) \in \{0,1\}^{n \times n}$ are zero except for the one element $( \hat{N}(t) )_{X_{t},X_{t+1}} = 1$. It is important to note that matrices $\hat{N}(t)$ and $\hat{N}(t-1)$ are dependent. In particular, only the $X_{t}$-th row of $\hat{N}(t)$ can contain a nonzero value. These dependencies are what make the analysis challenging.

To circumvent the difficulties associated with these dependencies, we use the following trick. We split $\hat{N}$ into two parts, specifically,
\begin{equation}
\hat{N}^{(\textrm{even})} 
= \sum_{t=0}^{\lceil T/2 \rceil -1} \hat{N}(2t) 
\quad 
\textrm{and}
\quad 
\hat{N}^{(\textrm{odd})} = \sum_{t=0}^{\lfloor T/2 \rfloor -1} \hat{N}(2t+1).
\end{equation}
This particular split ensures that the elements of $\hat{N}^{(\textrm{even})}$ and $\hat{N}^{(\textrm{odd})}$ are almost independent. Consider for instance the dependency between a matrix $\hat{N}(t)$ and the matrix $\hat{N}(t-2)$. Note that $\hat{N}(t)$ can contain a nonzero element almost anywhere when comparing to $\hat{N}(t-2)$, and that the only exceptions are diagonal entries.

Let us formalize this more precisely. Define $\bar{N}(t) = \expectationWrt{ \hat{N}(t) | \allowbreak \hat{N}(t-2) }{P}$ for $t \geq 2$ and $\bar{N}(t) = \mathrm{diag}(\Pi) P$ for $t = 1, 2$. Note that $\expectationWrt{ \hat{N}(t) | \allowbreak \hat{N}(t-2) }{P} = \expectationWrt{ \hat{N}(t) | \hat{N}(t-2), \hat{N}(t-4), \dots }{P}$ for all $t \geq 2$ since $\process{X_t}{0 \leq t \leq T}$ is a Markov chain. Hence, we have $\bar{N}(t) = \mathrm{diag}( \jaron{ P_{X_{t-1},\cdot} } ) P$ \jaron{for $t \geq 2$}. This trick of separating the original process into two processes each of which skips one unit of time ensures that almost all elements of $\bar{N}(t)$ are of order $1/n^2$. Note that there exists an absolute constant $p_{\max} > 0$ such that $( \mathrm{diag}( P_{x,\cdot} ) P )_{y,z} \leq p_{\max} / n^2$ for all $x, y, z \in \mathcal{V}$.

We now explain how to show $\jaron{ \pnorm{ \hat{N}_\Gamma - N }{} = \bigOP{ \sqrt{ (T/n) \ln{(T/n)} } } }$. Using the triangle inequality, it follows that
$
\pnorm{ \hat{N}_\Gamma - N }{} 
\leq \pnorm{ \hat{N}_\Gamma - \sum_{t=0}^{T-1}\bar{N}(t) }{} 
+ \pnorm{ N - \sum_{t=0}^{T-1}\bar{N}(t) }{}
$.
To prove the proposition, we will first show that 
\begin{equation}
\pnormBig{ N - \sum_{t=0}^{T-1} \bar{N}(t) }{} 
= \bigOPbig{ \sqrt{ \frac{T}{n} \ln{ \frac{T}{n} } } }
\label{eq:PFNN2}
\end{equation}
and then that
\begin{equation}
\pnormBig{ \hat{N}_\Gamma - \sum_{t=0}^{T-1} \bar{N}(t) }{} 
= \bigOPbig{ \sqrt{ \frac{T}{n} \ln{ \frac{T}{n} } } }. 
\label{eq:PFNN1}
\end{equation}

\paragraph{\underline{Part 1. Proof of \refEquation{eq:PFNN2}}}
The Frobenius norm provides an upper bound of the spectral norm. Namely,
\begin{equation}
\pnormBig{ N -\sum_{t=0}^{T-1}\bar{N}(t) }{} 
\leq \pnormBig{ \Bigl( \sum_{t=2}^{T-1} \bigl( \mathrm{diag}(\Pi) - \mathrm{diag}( \jaron{ P_{X_{t-1},\cdot} } ) \bigr) \Bigr) P }{\mathrm{F}}.
\end{equation}
\jaron{Observe that for any matrix $A$, $\pnorm{ \mathrm{diag}(A) P }{\mathrm{F}} \leq \max_{x \in \mathcal{V}} \{ \pnorm{ P_{x,\cdot} }{2} \} \pnorm{A}{2}$ by the Cauchy--Schwarz inequality.} Also note that $\pnorm{ P_{x,\cdot} }{2} = \bigO{ 1/\sqrt{n} }$ for all $x \in \mathcal{V}$. Hence to establish (\ref{eq:PFNN2}), it is sufficient to show that
\begin{equation}
\pnormBig{ \sum_{t=2}^{T-1} \bigl( \Pi - \jaron{ P_{X_{t-1},\cdot} } \bigr) }{2}
= \bigOPbig{ \sqrt{ T \ln{ \frac{T}{n} } } }.
\label{eq:PFNN2s}
\end{equation}

\revisedPartBegin
In fact \refEquation{eq:PFNN2s} can readily be verified. Centering, and using the triangle inequality, we have
\begin{align}
\pnormBig{ \sum_{t=2}^{T-1} \bigl( \Pi- P_{X_{t-1},\cdot} \bigr) }{2}
& \leq \pnormBig{ \sum_{t=2}^{T-1}\Pi- \sum_{i=1}^K \frac{\hat{N}_{\mathcal{V},\mathcal{V}_k}}{|\mathcal{V}_k|}\sum_{v \in \mathcal{V}_i}P_{v,\cdot} }{2} 
\nonumber \\ &
\phantom{\leq} + \pnormBig{ \sum_{i=1}^K \frac{\hat{N}_{\mathcal{V},\mathcal{V}_k}}{|\mathcal{V}_k|} \sum_{x \in \mathcal{V}_k} P_{x,\cdot} - \sum_{t=2}^{T-1} P_{X_{t-1},\cdot} }{2}.
\label{eqn:Intermediate_bound_within_proofs_of_Chapter_6}
\end{align}
To bound the first term in the r.h.s. of \refEquation{eqn:Intermediate_bound_within_proofs_of_Chapter_6}, recall first that (i) $\Pi_y = \sum_{x \in \mathcal{V}} \Pi_x P_{x,y} $, $N_{x,y} = T \Pi_x P_{x,y}$, $\forall_{x,y}$, and (ii) $\Pi_x = \bar{\Pi}_{\sigma(x)} \, \forall_x$. Therefore (iii) $T \bar{\Pi}_k \cardinality{ \mathcal{V}_k } = N_{\mathcal{V},V_k}$. Thus
\begin{align}
&
\pnormBig{ \sum_{t=2}^{T-1}\Pi- \sum_{k=1}^K \frac{\hat{N}_{\mathcal{V},\mathcal{V}_k}}{ \cardinality{ \mathcal{V}_k } } \sum_{x \in \mathcal{V}_k} P_{x,\cdot} }{2}^2 
\eqcom{i}= \sum_{y \in \mathcal{V}} \Bigl| (T-2) \Pi_y - \sum_{k=1}^K \frac{\hat{N}_{\mathcal{V},\mathcal{V}_i}}{ \cardinality{ \mathcal{V}_i } } \sum_{x \in \mathcal{V}_k} P_{x,y} \Bigr|^2
\nonumber \\ &
\eqcom{ii}= \sum_{y \in \mathcal{V}} \Bigl| (T-2) \sum_{k=1}^K \sum_{x \in \mathcal{V}_k} \bar{\Pi}_k P_{x,y} - \sum_{k=1}^K \frac{\hat{N}_{\mathcal{V},\mathcal{V}_k}}{ \cardinality{ \mathcal{V}_k } } \sum_{x \in \mathcal{V}_k} P_{x,y} \Bigr|^2
\nonumber \\ &
\eqcom{iii}= \sum_{y \in \mathcal{V}} \Bigl| \sum_{k=1}^K \frac{ \sum_{x \in \mathcal{V}_k} P_{x,y} }{ \cardinality{ \mathcal{V}_k } } \Bigl[ \frac{T-2}{T} N_{\mathcal{V},\mathcal{V}_k} - \hat{N}_{\mathcal{V},\mathcal{V}_k} \Bigr] \Bigr|^2
\nonumber \\ &
= \bigObig{ \frac{1}{n} } \Bigl| \sum_{k=1}^K \Bigl[ \frac{T-2}{T} N_{\mathcal{V},\mathcal{V}_k} - \hat{N}_{\mathcal{V},\mathcal{V}_k} \Bigr] \Bigr|^2
= \bigOPbig{ \frac{T}{n} \ln{ \frac{T}{n} } }
\end{align}
where the last equality follows from \refEquation{eq:Nseti}. To bound the second term in the r.h.s.\ of \refEquation{eqn:Intermediate_bound_within_proofs_of_Chapter_6}, note that (iv) $\hat{N}_{\mathcal{V},x} = \sum_{y \in \mathcal{V}} \hat{N}_{y,x} = \sum_{y \in \mathcal{V}} \sum_{t=0}^{T-1} \indicator{ X_t \allowbreak = y, X_{t+1} = x } = \sum_{t=0}^{T-1} \indicator{ X_{t+1} = x }$ by definition, and write
\begin{align}
&
\pnormBig{ \sum_{k=1}^K \hat{N}_{\mathcal{V},\mathcal{V}_k} \frac{ \sum\limits_{x \in \mathcal{V}_k}P_{x,\cdot} }{ \cardinality{ \mathcal{V}_k } } - \sum_{t=2}^{T-1} P_{X_{t-1},\cdot} }{2}^2 
= \sum_{ y \in \mathcal{V} } \Bigl| \sum_{k=1}^K \hat{N}_{\mathcal{V},\mathcal{V}_k} \frac{ \sum\limits_{x \in \mathcal{V}_k} P_{x,y} }{ \cardinality{ \mathcal{V}_k } } - \sum_{t=2}^{T-1} P_{X_{t-1},y} \Bigr|^2
\nonumber \\ &
= \sum_{ y \in \mathcal{V} } \Bigl| \sum_{k=1}^K \hat{N}_{\mathcal{V},\mathcal{V}_k} \frac{ \sum_{x \in \mathcal{V}_k} P_{x,y} }{ \cardinality{ \mathcal{V}_k } } - \sum_{x \in \mathcal{V}} \sum_{t=2}^{T-1} P_{x,y} \indicator{ X_{t-1} = x } \Bigr|^2
\nonumber \\ &
\eqcom{iv}= \sum_{ y \in \mathcal{V} } \Bigl| \sum_{k=1}^K \Bigl( \hat{N}_{\mathcal{V},\mathcal{V}_k} \frac{ \sum\limits_{x \in \mathcal{V}_k} P_{x,y} }{ \cardinality{ \mathcal{V}_k } } - \sum_{x \in \mathcal{V}_k} P_{x,y} \bigl( \hat{N}_{\mathcal{V},x} - \indicator{ X_{T-1} = x } - \indicator{ X_T = x } \bigr) \Bigr) \Bigr|^2
\nonumber \\ &
\leq 2 \sum_{ y \in \mathcal{V} } \Bigl| \sum_{k=1}^K \hat{N}_{\mathcal{V},\mathcal{V}_k} \frac{ \sum_{x \in \mathcal{V}_k} P_{x,y} }{ \cardinality{ \mathcal{V}_k } } - \sum_{k=1}^K \sum_{x \in \mathcal{V}_k} P_{x,y}  \hat{N}_{\mathcal{V},x}  \Bigr|^2
\nonumber \\ &
\phantom{\leq} + 2\sum_{ y \in \mathcal{V} } \Bigl| \sum_{k=1}^K \sum_{x \in \mathcal{V}_k} P_{x,y} \bigl( \indicator{ X_{T-1} = x } + \indicator{ X_T = x } \bigr) \Bigr|^2\nonumber \\ &
\leq 2 \sum_{ y \in \mathcal{V} } \Bigl| \sum_{k=1}^K \max_{w,z\in \mathcal{V}_k} |\hat{N}_{\mathcal{V},w}-\hat{N}_{\mathcal{V},z}| \frac{ \sum_{x \in \mathcal{V}_k} P_{x,y} }{ \cardinality{ \mathcal{V}_k } } \Bigr|^2
\nonumber \\ &
\phantom{\leq} + 2\sum_{ y \in \mathcal{V} } \Bigl| \sum_{k=1}^K \sum_{x \in \mathcal{V}_k} P_{x,y} \bigl( \indicator{ X_{T-1} = x } + \indicator{ X_T = x } \bigr) \Bigr|^2\nonumber \\ &
\eqcom{v}= 2\sum_{ y \in \mathcal{V} } \Bigl| \sum_{k=1}^K \sum_{x \in \mathcal{V}_k} P_{x,y} \bigl( \indicator{ X_{T-1} = x } + \indicator{ X_T = x } \bigr) \Bigr|^2 + \bigOPbig{ \frac{T}{n^2} \ln{n}^2 }
\nonumber \\ &
= \bigOPbig{ \frac{T}{n} \ln{ \frac{T}{n}}},
\end{align}
where (v) stems from \eqref{eq:trim}.
This proves \refEquation{eq:PFNN2s} after taking the square root. \qed
\revisedPartEnd

\paragraph{\underline{Part 2. Proof of \refEquation{eq:PFNN1}}}
Using the triangle inequality, we obtain
\begin{equation}
\pnormBig{ \hat{N}_\Gamma - \sum_{t=0}^{T-1} \bar{N}(t) }{} 
\leq \pnormBig{ \sum_{t=0}^{ \bigl\lceil \tfrac{T}{2} \bigr\rceil -1} \bigl(\hat{N}_\Gamma(2t) - \bar{N}(2t) \bigr) }{} 
+ \pnormBig{ \sum_{t=0}^{ \bigl\lfloor \tfrac{T}{2} \bigr\rfloor -1} \bigl(\hat{N}_\Gamma(2t+1) - \bar{N}(2t+1) \bigr) }{}. 
\label{eq:NN1}
\end{equation}
We next show that the first term in the r.h.s.\ of (\ref{eq:NN1}) is in fact $\bigOP{ \sqrt{T/n} \allowbreak \ln{ (T/n) } }$. The second term can be bounded using an analogous argument. The proof classically consists in relating the spectral norm to the rectangular quotient, and then combining an $\varepsilon$-net argument with a sufficiently strong concentration inequality.

To simplify notation, define $A \triangleq \sum_{t=0}^{\lceil T/2 \rceil -1} \hat{N}(2t)$, $M \triangleq \sum_{t=0}^{\lceil T/2 \rceil -1} \bar{N}(2t)$, and $f_n=\sqrt{ (T/n) \ln{(T/n)}}$. We wish to bound $\| A_\Gamma -M\|$. To this aim, we use the {\it rectangular quotient relation} \cite{dax_eigenvalues_2013}:
\revisedPartBegin
\begin{equation}
\pnorm{A_\Gamma -M}{}
= \sigma_1 
= \max_{ x, y \in \mathbb{S}^{n-1} } | \transpose{x} (A_\Gamma -M) y |.
\end{equation}
where $\mathbb{S}^{n-1}$ denotes the unit sphere in $\mathbb{R}^n$.

\medskip
\noindent
\emph{Step 1: The $\epsilon$-net argument.} 
Let us first recall the definition of an $\epsilon$-net: let $(X,d)$ be a metric space and let $\epsilon > 0$. A subset $\mathcal{N}_\epsilon$ of $X$ is called an $\epsilon$-net of $X$ if every point $x \in X$ can be approximated to within $\epsilon$ by some point $y \in \mathcal{N}_\epsilon$, i.e., so that $d(x,y) \leq \epsilon$. To bound $\pnorm{A_\Gamma -M}{}$, we use the classical $\epsilon$-net argument formalized in \refLemma{lem:Relation_spectral_norm_and_rectangular_quotient_over_epsilon_net}. 

\begin{lemma}
\label{lem:Relation_spectral_norm_and_rectangular_quotient_over_epsilon_net}
Let $\mathcal{N}_\epsilon$ denote an $\epsilon$-net of $\mathbb{S}^{n-1}$ for some $\epsilon \in [0,1)$. Then
\begin{equation}
\pnorm{A_\Gamma -M}{} 
\leq \frac{1}{ 1 - 3 \epsilon } \max_{ x, y \in \mathcal{N}_\epsilon } | \transpose{x} (A_\Gamma-M) y |.
\end{equation}
\end{lemma}

\begin{proof} We use here the notation $B=A_\Gamma -M$. Adapting the strategy in \cite[Lemma 5.4]{vershynin_introduction_2010}: Choose $a, b \in \mathbb{S}^{n-1}$ such that $\pnorm{B}{} = | \transpose{a} B b |$ and choose $x, y \in \mathcal{N}_\epsilon$ such that $\pnorm{x-a}{2} \leq \epsilon$ and $\pnorm{y-b}{2} \leq \epsilon$. Then by the triangle inequality, 
\begin{align}
&
| \transpose{x} B y - \transpose{a} B b |
\leq \bigl| \transpose{(x-a)} B (y-b) \bigr| + \bigl| \transpose{a} B (y-b) \bigr| + \bigl| \transpose{(x-a)} B b \bigr|
\leq 
\nonumber \\ &
\pnorm{x-a}{2} \pnorm{B}{} \pnorm{y-b}{2} + \pnorm{a}{2} \pnorm{B}{} \pnorm{y-b}{2} + \pnorm{x-a}{2} \pnorm{B}{} \pnorm{b}{2}
= ( 2 \epsilon + \epsilon^2 ) \pnorm{B}{}.
\end{align}
Therefore $| \transpose{x} B y | \geq ( 1 - 2 \epsilon - \epsilon^2 ) \pnorm{B}{}$. By first maximizing over such $x, y$ and next extending the optimization range, we obtain
\begin{equation}
( 1 - 2 \epsilon - \epsilon^2 ) \pnorm{B}{} 
\leq \max_{ \{ x, y \in \mathcal{N}_\epsilon : \pnorm{x-a}{2} \leq \epsilon, \pnorm{y-a}{2} \leq \epsilon \} } | \transpose{x} B y |
\leq \max_{  x, y \in \mathcal{N}_\epsilon } | \transpose{x} B y |.
\end{equation} 
That completes the proof of \refLemma{lem:Relation_spectral_norm_and_rectangular_quotient_over_epsilon_net}.
\end{proof}

By a volume covering argument there exists an $\epsilon$-net of $\mathbb{S}^{n-1}$ satisfying $\cardinality{ \mathcal{N}_\epsilon } \leq \e{ \zeta_\epsilon n }$ with $\zeta_\epsilon = \ln{ ( 1 + 2 / \epsilon ) }$ \cite[Lemma 5.2]{vershynin_introduction_2010}. In the remainder of this proof, we use such an $\epsilon$-net.
\revisedPartEnd

\medskip
\noindent
\emph{Step 2: Splitting between light and heavy couples.} To bound $\max_{ x, y \in \mathcal{N}_\epsilon } | \transpose{x} (A_\Gamma \allowbreak - M) y |$, we adapt the proof strategy used in \cite{feige_spectral_2005}. Let us fix $x, y \in \mathcal{N}_\epsilon$. Define $\mathcal{L} = \bigl\{ (v, w) \in \mathcal{V} \times \mathcal{V} : |x_v y_{w} | \leq (1/n) \sqrt{ T / n } \bigr\}$ to be the set of \emph{light couples}. Its complement $\mathcal{L}^{\mathrm{c}} \triangleq \mathcal{V} \backslash \mathcal{L}$ will be called the set of \emph{heavy couples}. Furthermore define $\mathcal{K} \triangleq (\Gamma^{\mathrm{c}} \times \mathcal{V}) \cup (\mathcal{V} \times \Gamma^{\mathrm{c}})$. Using the triangle inequality, we bound
\begin{equation}
| \transpose{x} \left( A_\Gamma - M \right) y |
\leq F_1(x,y) + F_2(x,y) + F_3(x,y)
\end{equation}
where 
$F_1(x,y) \triangleq | \sum_{ (v, w) \in \mathcal{K} \cap \mathcal{L} } x_v A_{vw} y_{w} |$,
$F_2(x,y) \triangleq | \sum_{(v, w)\in \mathcal{L}} x_v A_{vw}y_{w} - \transpose{x} M y |$, and
$F_3(x,y) \triangleq | \sum_{(v, w)\in \mathcal{L}^{\mathrm{c}} } x_v (A_\Gamma)_{vw} y_{w} |$. Next we bound each of these terms individually.

\medskip
\noindent
\emph{Step 3a: Exponential concentration of $F_1(x,y)$.}

\revisedPartBegin
\begin{lemma}
\label{lem:Bound_for_F1xy}
$
\forall_{x, y \in \mathcal{N}_\epsilon, c \geq 1}, \exists_{ d_3>0, \delta_c, N_c } :$
$$
\probability{ F_1(x,y) \geq \delta_c \sqrt{T/n} }
\leq \e{ ( \ln{2} - d_3 c ) n }, 
\quad  \forall_{ n > N_c }.
$$
\end{lemma}

\begin{proof}
Let $x, y \in \mathcal{N}_\epsilon$. Define $\mathcal{M}_n \triangleq \{ \mathcal{S} \subseteq \mathcal{V} : \cardinality{ \mathcal{S} } = \lfloor n \exp{} ( -T/n \cdot \allowbreak \ln{(T/n)} ) \rfloor \}$, and bound
\begin{equation}
n \sqrt{ \frac{n}{T} } F_1(x,y) 
\eqcom{i}\leq \Bigl| \sum_{ (v, w) \in \mathcal{K} \cap \mathcal{L} } A_{vw} \Bigr|
\leq 2 \hat{N}_{\mathcal{V}, \Gamma^{\mathrm{c}} }
\leq 2 \max_{ \mathcal{S} \in \mathcal{M}_n } \bigl\{ \hat{N}_{\mathcal{V},\mathcal{S}} \bigr\}
\label{eqn:Intermediate_equation_for_the_exponential_concentration_of_F1}
\end{equation}
by using (i) $\mathcal{L}$'s definition. 

We next show that the r.h.s.\ of \refEquation{eqn:Intermediate_equation_for_the_exponential_concentration_of_F1} is $\bigOP{n}$. For some sufficiently large constant $c$, we obtain (i) by Boole's inequality and the fact that $N_{\mathcal{V},\mathcal{S}} = \bigO{ n ( n / T )^{T/n-1} } = \smallO{n}$ since $\mathcal{S} \in \mathcal{M}_n$ and $T = \omega(n)$ and (ii) by the observation that the number of subsets of $\mathcal{V}$ is less than $2^n$ and \refEquation{eq:trim2} that
\begin{equation}
\probability{ \max_{ \mathcal{S} \in \mathcal{M}_n } \bigl\{ \hat{N}_{\mathcal{V},\mathcal{S}} \bigr\} \geq c n }
\eqcom{i}\leq \sum_{ \mathcal{S} \in \mathcal{M}_n } \probability{ | \hat{N}_{\mathcal{V}, \mathcal{S}} - N_{ \mathcal{V}, \mathcal{S} } | \geq c n - \smallO{n} } 
\eqcom{ii}\leq 2^n \e{- d_3 c n } 
\rightarrow 0.
\end{equation}
This proves \refLemma{lem:Bound_for_F1xy}.
\end{proof}

\noindent
\emph{Step 3b: Exponential concentration of $F_2(x,y)$.}

\begin{lemma}
\label{lem:Bound_for_F2xy}
$
\forall_{x, y \in \mathcal{N}_\epsilon, \delta > 0} :
\probability{ F_2(x,y) \geq \delta f_n }
\leq \e{ ( p_{\max} - \frac{\delta}{2} \sqrt{ \ln{ \frac{T}{n} } } ) n }
\ , \forall_{n > 0 }
$.
\end{lemma}

\begin{proof}
Let $x,y \in \mathcal{N}_\varepsilon$, $a_n > 0$, and $\lambda > 0$. Markov's inequality implies
\begin{equation}
\probabilityBig{ \sum_{ (v, w) \in \mathcal{L} } x_v A_{vw} y_{w} - \transpose{x} M y \geq a_n \sqrt{\frac{T}{n}} }
\leq \frac{ \expectation{ \exp{ \bigl( \lambda \sum_{(v, w)\in \mathcal{L}} x_v A_{vw}y_{w} \bigr) } } }
{ \exp{ ( \lambda a_n\sqrt{T/n} + \transpose{x} M y ) } }.
\end{equation}

We start by bounding the numerator. Prepare for a folding argument by noting that $\hat{N}_{vw}(s) \in \{ 0,1 \}$ and $\sum_{v,w} \hat{N}_{vw}(s) = 1$ for $s \geq 0$. This implies that $\probability{ \hat{N}_{ru}(s) = 1 | \mathcal{F}_{s-2} } = \bar{N}_{ru}(s)$. Hence we calculate for $s \geq 0$ that
\begin{align}
&
\expectationBig{ \e{ \lambda \sum_{ (v, w) \in \mathcal{L} } x_v \hat{N}_{vw}(s) y_w } \Big| \mathcal{F}_{s-2} }
= \sum_{ (r,u) \in \mathcal{V}^2 } \bigl( \indicator{ r,u \not\in \mathcal{L} } + \indicator{ r,u \in \mathcal{L} } \e{ \lambda x_u y_r } \bigr) \bar{N}_{ru}(s)
\nonumber \\ &
= \sum_{ (r,u) \in \mathcal{L}^{\mathrm{c}} } \bar{N}_{ru}(s) + \sum_{ (r,u) \in \mathcal{L}^{\mathrm{c}} } \e{ \lambda x_r y_u } \bar{N}_{ru}(s)
= 1 + \sum_{ (v,w) \in \mathcal{L}^{\mathrm{c}} } \bar{N}_{vw}(s) \bigl( \e{ \lambda x_v y_w } - 1 \bigr).
\end{align}
We are now in place to use the tower property to fold up the numerator backwards through time. Starting from $t =  \lceil T/2 \rceil - 1$, we calculate
\begin{align}
&
\,\,
\expectationBig{ \exp{ \bigl( \lambda \sum_{ (v, w) \in \mathcal{L} } x_v A_{vw}y_w \bigr) } }
= \expectationBig{ \prod_{t=0}^{ \lceil T/2 \rceil - 1 } \exp{ \bigl( \lambda \sum_{ (v, w) \in \mathcal{L} } x_v \hat{N}_{vw}(2t) y_w \bigr) } }
=
\nonumber \\ & 
\expectationBig{ 
\expectationBig{ \e{ \lambda \sum_{ (v, w) \in \mathcal{L} } x_v \hat{N}_{vw}( 2 \lceil T/2 \rceil - 2 ) y_w } \Big| \mathcal{F}_{ 2 \lceil T/2 \rceil - 4 } } 
\prod_{t=0}^{ \lceil T/2 \rceil - 2 } \e{ \lambda \sum_{ (v, w) \in \mathcal{L} } x_v \hat{N}_{vw}(2t) y_w } 
},
\end{align}
and then repeat the argument for $t = \lceil T/2 \rceil - 2, \ldots$, et cetera. This brings us to (i) below, and we now use the elementary bounds (ii) $e^x \leq 1+ x+2x^2$ for $|x| \leq 1/2$ and (iii) $1 + x \leq e^x$ to obtain
\begin{align}
& \e{ \lambda a_n\sqrt{ \frac{T}{n} } } \probabilityBig{ \sum_{ (v, w) \in \mathcal{L} } x_v A_{vw} y_{w} - \transpose{x} M y \geq a_n \sqrt{\frac{T}{n}} }
\nonumber \\
\eqcom{i}\leq &\e{ - \lambda \transpose{x} M y } \expectationBig{ \prod_{t=0}^{ \lceil T/2\rceil - 1 } \Bigr( 1 + \sum_{ (v, w) \in \mathcal{L}}\bar{N}_{vw}(2t) \bigl( \e{ \lambda x_v y_{w} } - 1 \bigr) \Bigr) }
\nonumber \\
\eqcom{ii}\leq &\e{ - \lambda \transpose{x} M y } \expectationBig{ \prod_{t=0}^{ \lceil T/2\rceil - 1 } \Bigl( 1 + \sum_{ (v, w) \in \mathcal{L}}\bar{N}_{vw}(2t) \bigl( \lambda  x_v y_{w} + 2 \lambda^2  x_v^2 y_w^2 ) \bigr) \Bigr) }
\nonumber \\ 
\eqcom{iii}\leq &\e{ - \lambda \transpose{x} M y } \expectationBig{ \exp{ \Bigl( \sum_{ (v,w) \in \mathcal{L} } M_{vw} \bigl( \lambda  x_v y_{w} + 2 \lambda^2 x_v^2 y_w^2 \bigr) \Bigr) } }. \label{eq:step0}
\end{align}

We next bound the r.h.s. of \refEquation{eq:step0}. Use a contradiction argument to note that since $\sum_{v \in \mathcal{V}} \sum_{w \in \mathcal{V}} x_v^2 y_{w}^2 = 1$ and $|x_v y_{w}| > (1/n) \sqrt{ T/n }$ for all $(v,w)\in \mathcal{L}^{\mathrm{c}}$, we must have $\sum_{ (v,w) \in \mathcal{L}^{\mathrm{c}} } |x_v y_{w}| \le n\sqrt{ n/T }$. Specify $\lambda = \frac{1}{2} n \sqrt{ n / T }$. Together with the bound $M_{vw} \leq p_{\max} T / n^2$ for all $v,w$, we obtain
\revisedPartEnd
\begin{align} 
\sum_{ (v,w) \in \mathcal{L} } M_{vw} \lambda x_v y_w - \lambda \transpose{x} M y 
= -\sum_{(v,v)\in \mathcal{L}^{\mathrm{c}} } \lambda M_{vw} x_v y_w
\leq \frac{n}{2}p_{\max}. \label{eq:step1}
\end{align}
Additionally since $\sum_{v \in \mathcal{V}} \sum_{w \in \mathcal{V}} x_v^2 y_{w}^2 = 1$, it follows that
\begin{equation}
\sum_{(v, w)\in \mathcal{L}} M_{vw} 2 \lambda^2  x_v^2 y_w^2 
\leq \frac{n}{2} p_{\max}. \label{eq:step2}
\end{equation}

Finally, by combining \refEquation{eq:step0}--\refEquation{eq:step2}, we obtain
\begin{equation}
\probabilityBig{ \sum_{(v, w)\in \mathcal{L}} x_v A_{vw}y_{w} - \transpose{x} M y \geq a_n\sqrt{\frac{T}{n}} }
\leq \exp{ \Bigl( n p_{\max} - \frac{a_n}{2} n \Bigr) }. \label{eq:xyup}
\end{equation}
This proves \refLemma{lem:Bound_for_F2xy}.
\end{proof}

\noindent
\emph{Step 3c: Discrepancy property for $F_3(x,y)$.} We extend the arguments in \cite{feige_spectral_2005} as follows. First, we introduce the quantity $e(\mathcal{A},\mathcal{B}) \triangleq \sum_{x \in \mathcal{A}} \sum_{y \in \mathcal{B} } ( A_\Gamma )_{x,y}$. Next, we say that the random variable $A_\Gamma$ satisfies the \emph{discrepancy property} if there exist constants $c_2, c_3 > 0$ such that for every $\mathcal{A}, \mathcal{B} \subset \mathcal{V}$ one of the following holds:
\begin{itemize}
\item[(i)] $\frac{e(\mathcal{A},\mathcal{B})n^2}{\cardinality{\mathcal{A}}\cardinality{\mathcal{B}}T} \le c_2 \ln\frac{T}{n}$ 
\item[(ii)] $e(\mathcal{A},\mathcal{B}) \ln \frac{e(\mathcal{A},\mathcal{B})n^2}{\cardinality{\mathcal{A}}\cardinality{\mathcal{B}}T} \le c_3 \max{ \{ \cardinality{\mathcal{A}}, \cardinality{\mathcal{B}} \} } \ln{ \frac{n}{ \max{ \{ \cardinality{\mathcal{A}},\cardinality{\mathcal{B}} \} } } }.$
\end{itemize}
We now prove that the discrepancy property provides an absolute bound on $F_3(x,y)$, and that it holds with high probability.

\begin{lemma}
\label{lem:Discrepancy_property_gives_bound}
If $A_\Gamma$ satisfies the discrepancy property, it holds that
$$
\exists_{c>0} : F_3(x,y) \leq c f_n,\quad \forall_{x,y \in \mathcal{N}_\epsilon}.
$$
\end{lemma}

\begin{proof}
This is explained in \cite{feige_spectral_2005}. This is Remark 4.5 of \cite{keshavan2010matrix}.
\end{proof}

\begin{lemma}
\label{lem:Discrepancy_property_holds_whp}
The random variable $A_\Gamma$ satisfies the discrepancy property with probability of at least $1 - 1 / n$.
\end{lemma}
\revisedPartEnd

\begin{proof}
Let $\mathcal{A}, \mathcal{B} \subset \mathcal{V}$ be two subsets such that $\cardinality{\mathcal{A}} \leq b$ w.l.o.g. \jaron{In these next two paragraphs, we temporarily let $a = \cardinality{\mathcal{A}}$, $b = b$ to declutter notation.} Also let $c_2, c_3$ be two large constants (how large will be sufficient will become clear in a moment). We can now distinguish two cases:

\noindent
\emph{Case 1:} $b \geq n /5$. The trimming step ensures that $e(v,\mathcal{V}) = \bigO{ T/n }$ for all $v\in \mathcal{V}$. We therefore have in this case that $e(\mathcal{A},\mathcal{B}) \leq c_2 abT / n^2$ \jaron{for a sufficiently large constant $c_2$}.

\noindent
\emph{Case 2:} $b \leq n/5$. For this case, define the quantity $\eta( a, b ) = \max \{ \eta^{\star}, \allowbreak (c_2 a b T \ln{ (T/n) } ) / n^2 \}$ with $\eta^{\star}$ the constant that satisfies the relation $\eta^\star \cdot \ln{ \bigl( ( \eta^\star n^2 ) / \allowbreak ( a b T ) \bigr) } = c_3 b \ln{ ( n / b ) }$. If all pairs of subsets $\mathcal{A}, \mathcal{B} \subset \mathcal{V}$ satisfy $e(\mathcal{A},\mathcal{B}) \leq \eta( a, b )$, the discrepancy property holds. It therefore suffices to show that $e(\mathcal{A},\mathcal{B}) \le \eta( a, b )$ with high probability for all \jaron{$\mathcal{A}, \mathcal{B} \subset \mathcal{V}$}.

We will first quantify the probability that $e(\mathcal{A},\mathcal{B}) \leq \eta( a, b )$ for any arbitrary subsets \jaron{$\mathcal{A}, \mathcal{B} \subset \mathcal{V}$}. Using Markov's inequality, we obtain
\begin{align}
&
\probability{ e(\mathcal{A},\mathcal{B}) > \eta(a,b) }
\leq \inf_{ h \geq 0 } \frac{ \expectation{ \exp{ ( h \cdot e(\mathcal{A},\mathcal{B}) ) } } }{ \exp{ ( h \cdot \eta(a, b ) ) } } 
\nonumber \\ &
\jaron{
\leq \inf_{ h \geq 0 } \frac{ \prod_{t=1}^{\lceil T/2\rceil - 1} \bigl( 1 + \frac{ a b p_{\max} }{ n^2 } \e{h} \bigr) }{ \exp{ ( h \cdot \eta( a, b ) ) } } 
\leq \inf_{ h \geq 0 } \frac{ \prod_{t=1}^{\lceil T/2\rceil - 1} \exp{ \bigl( \frac{ a b p_{\max} }{ n^2 } \e{h} \bigr) } }{ \exp{ ( h \cdot \eta( a, b ) ) } } 
}
\nonumber \\ &
\leq \inf_{ h \geq 0 } \exp{ \Bigl(  \frac{ a b p_{\max} T }{ 2 n^2 } \e{h} - h \eta( a, b ) \Bigr) } 
\nonumber \\ &
\leq \exp{ \Bigl( - \eta( a,b ) \Bigl(\ln \frac{ 2 n^2 \eta( a, b ) }{ a b p_{\max} T } - 1 \Bigr) \Bigr) },\label{eq:c2pre}
\end{align}
where, for the last inequality, we specify $h = \ln{ \bigl( 2 n^2 \eta( a, b ) ) / ( a b p_{\max} T ) \bigr) }.$

As a last step, we compute the expected number of pairs $\mathcal{A}, \mathcal{B} \subset \mathcal{V}$ such that $e(\mathcal{A},\mathcal{B}) > \eta( \cardinality{\mathcal{A}}, \cardinality{\mathcal{B}} )$. The number of possible pairs of sets $\mathcal{A}$ and $\mathcal{B}$ such that $\cardinality{\mathcal{A}} = a$ and $\cardinality{\mathcal{B}} = b$ is ${{n}\choose{a}} {{n}\choose{b}}$. Hence using \refEquation{eq:c2pre}, 
\begin{align}
&\expectationBig{ \Bigl| \Bigl\{ (\mathcal{A},\mathcal{B}) \Big| e(\mathcal{A},\mathcal{B}) > \eta( \cardinality{\mathcal{A}}, \cardinality{\mathcal{B}} ), \cardinality{\mathcal{A}}=a, \cardinality{\mathcal{B}}=b, \mathcal{A}, \mathcal{B} \subset \mathcal{V} \Bigr\} \Bigr| } 
\nonumber \\ &  
\leq \binom{n}{a} \binom{n}{b} \max_{ \mathcal{A}, \mathcal{B} \subset \mathcal{V} \textrm{ s.t. } \cardinality{\mathcal{A}} = a, \cardinality{\mathcal{B}} = b } \probability{ e(\mathcal{A},\mathcal{B}) > \eta(a,b) } 
\nonumber \\ &
\eqcom{i}\leq \Bigl( \frac{n \e{} }{b}\Bigr)^{2b} \max_{ \mathcal{A}, \mathcal{B} \subset \mathcal{V} \textrm{ s.t. } \cardinality{\mathcal{A}} = a, \cardinality{\mathcal{B}} = b } \probability{ e(\mathcal{A},\mathcal{B}) > \eta(a,b) } 
\nonumber \\ &
\eqcom{ii}\leq \exp{ \Bigl( 4b \ln{ \frac{n}{b} } - \eta(a,b) \Bigl( \ln{ \frac{ 2 n^2 \eta(a,b) }{ ab p_{\max} T} } - 1 \Bigr) \Bigr) } 
\nonumber \\ &
\eqcom{iii}\leq \exp{ \Bigl( - 3 \ln n + 7b \ln{ \frac{n}{b} } - \eta(a,b) \Bigl( \ln{ \frac{ 2 n^2 \eta(a,b) }{ a b p_{\max} T } } - 1 ) \Bigr) }
\nonumber \\ &
\eqcom{iv}\leq \exp{ \Bigl( - 3 \ln n + 7 b \ln{ \frac{n}{b} } - \frac{\eta(a,b)}{2} \ln{ \frac{ 2 n^2 \eta(a,b) }{ a b p_{\max} T} } \Bigr) } 
\eqcom{v}\leq \frac{1}{n^3}.\label{eq:etabnd}
\end{align}
\revisedPartBegin
Here, we have used that (i,ii) $a \leq b \leq n/5$, and (iii) that $-3b \ln{(n/b)} \allowbreak \leq -3 \ln{n}$ on the interval $b \in [1,n/5]$. Inequality (iv) follows from $n^2 \eta(a,b) / ( a b T ) \allowbreak \geq c_2 \ln{ T/n }$ and the fact that $\ln{x} - 1 \geq \tfrac{1}{2} \ln{x}$ for sufficiently large $x$, and (v) follows from the definition of $\eta(a,b)$ since
\begin{equation}
7 b \ln{ \frac{n}{b} } - \frac{\eta(a,b)}{2} \ln{ \frac{ 2 n^2 \eta(a,b) }{ a b p_{\max} T} }
\leq 7 b \ln{ \frac{n}{b} } - \frac{c_3 b}{2}  \ln{ \frac{n}{b} } - \frac{\eta^0}{2} \ln{ \frac{2}{p_{\max}} }.
\end{equation}
Here, we have used specifically that $\eta(a,b) \geq \eta^0 \geq 0$, that $\eta^0$ satisfies $\eta^0 \ln{ ( \eta^0 n2 / (abT) ) } = c_3 b \ln{(n/b)}$, and that $2 / p_{\max} \geq 1$. We have therefore shown that for sufficiently large $c_3$, when we sum the above inequality for all possible cardinalities $a,b$,
\revisedPartEnd
\begin{equation}
\expectationBig{ \Bigl| \Bigl\{ (\mathcal{A},\mathcal{B}) \Big| e(\mathcal{A},\mathcal{B})> \eta(\cardinality{\mathcal{A}},\cardinality{\mathcal{B}}), \mathcal{A}, \mathcal{B} \subset V_1 \cap \Gamma \Bigr\} \Bigr| } 
\leq \frac{1}{n}.
\end{equation}
We can thus conclude that the discrepancy property holds with probability $1 - 1/n$. 
\end{proof}

\revisedPartBegin
\noindent
\emph{Step 4: Summary.}
Before putting the results obtained in the previous steps together, we make the following observation. For generic positive random variables $X, X_1, \ldots, X_m$ with $m < \infty$ satisfying $X \leq \sum_{i=1}^m X_i$,
$
\probability{ X \geq x } \leq \sum_{i=1}^m \probability{ X_i \geq x/m }
$
since
\begin{align}
\probability{ X < x }
&
\geq \probability{ \sum_{i=1}^m X_i < x }
\geq \probability{ \cap_{i=1}^m \{ X_i < x/m \} }
\nonumber \\ &
\geq 1 - \probability{ \cup_{i=1}^m \{ X_i \geq x/m \} }
\geq 1 - \sum_{i=1}^m \probability{ X_i \geq x/m }.
\end{align}
Therefore for any $\epsilon, \delta > 0$
\begin{align}
\probability{ \pnorm{ A_\Gamma - M }{} \geq \delta f_n }
&
\eqcom{i}\leq \probability{ \max_{ x,y \in \mathcal{N}_\epsilon } | \transpose{x} (A_\Gamma - M) y | \geq \delta ( 1 - 3\epsilon ) f_n }
\nonumber \\ &
\leq \sum_{i=1}^3 \probabilityBig{ \max_{ x,y \in \mathcal{N}_\epsilon } F_i(x,y) \geq \frac{ \delta ( 1 - 3\epsilon ) f_n }{3} }
\label{eqn:Intermediate_bound_for_spectral_norm_before_using_epsilon_net_and_discrepancy}
\end{align}
where (i) we have used \refLemma{lem:Relation_spectral_norm_and_rectangular_quotient_over_epsilon_net}. We now bound the r.h.s.\ of \refEquation{eqn:Intermediate_bound_for_spectral_norm_before_using_epsilon_net_and_discrepancy}:

\noindent
\emph{First and second terms (corresponding to $F_1$ and $F_2$).} Since we have exponential concentration, we can use the union bound on the $\epsilon$-net for these terms. Applying \refLemma{lem:Bound_for_F1xy} with $c > ( \zeta_\epsilon + \ln{2} ) / d_3$ when $\delta > 3 \delta_c / ( 1 - 3 \epsilon )$, we obtain
\begin{enumerate}
\item[(i)] $\e{ \zeta_\epsilon n } \probability{ F_1(x,y) \geq (\delta / 3) (1 - 3 \epsilon) f_n } \leq \e{ \zeta_\epsilon n } \probability{ F_1(x,y) \geq \delta_c \sqrt{T/n} } = \bigO{ \e{ ( \zeta_\epsilon + \ln{2} - d_3 c )n } } \to 0$ as $n \to \infty$.
\end{enumerate}
Using \refLemma{lem:Bound_for_F2xy}, we find
\begin{enumerate}
\item[(ii)] $\e{ \zeta_\epsilon n } \probability{ F_2(x,y) \geq (\delta / 3) (1 - 3 \epsilon) f_n } \leq \e{ (\zeta_\epsilon + p_{\max} - \frac{ \delta }{6} \sqrt{\ln{ \frac{T}{n} } } ) n } \to 0$ as $n \to \infty$.
\end{enumerate}

\noindent
\emph{Third term (corresponding to $F_3$).} The third term cannot be bounded in a sufficiently tight manner using the union bound on the $\epsilon$-net. We instead rely on the discrepancy property. Write $\mathcal{D}$ for the event that $A_\Gamma$ satisfies the discrepancy property. Using \refLemma{lem:Discrepancy_property_gives_bound} for sufficiently large $C$, and \refLemma{lem:Discrepancy_property_holds_whp}, we arrive at
\begin{enumerate}
\item[(iii)] $
\probability{ \max_{ x,y \in \mathcal{N}_\epsilon } F_3(x,y) \geq C f_n }
= \probability{ \max_{ x,y \in \mathcal{N}_\epsilon } F_3(x,y) \geq C f_n | \mathcal{D} } \probability{ \mathcal{D} } 
+ \probability{  \max_{ x,y \in \mathcal{N}_\epsilon } F_3(x,y) \geq C f_n | \mathcal{D}^{\mathrm{c}} } \probability{ \mathcal{D}^{\mathrm{c}} }
\leq 
0 \cdot \probability{ \mathcal{D} } + 1 \cdot \probability{ \mathcal{D}^{\mathrm{c}} } 
\leq \frac{1}{n}
$.
\end{enumerate}

\noindent
Finally, by bounding \refEquation{eqn:Intermediate_bound_for_spectral_norm_before_using_epsilon_net_and_discrepancy} with (i)--(iii), we obtain the desired conclusion
\begin{equation}
\textrm{if } T = \omega(n) \textrm{, then }
\exists_{ \delta } : 
\lim_{n \to \infty}
\probabilityBig{ \pnormBig{ \sum_{t=0}^{\lceil T/2 \rceil -1} \bigl(\hat{N}_\Gamma(2t) - \bar{N}(2t) \bigr) }{} \geq \delta f_n }
= 0.
\end{equation}
Together with \refEquation{eq:NN1} this implies \refEquation{eq:PFNN1}.\qed
\revisedPartEnd

\section{Proof of Lemma~\ref{lem:N_vector_error_scales_minimally_with_this_bound_if_x_y_differ_clusters}}
\label{suppl:Proof__Separability_property_of_Nhat}

Because (i) $N_{x,y} = T \Pi_x P_{x,y}$, and by $P_{x,y}$'s definition in \refEquation{eqn:Definition_of_P},
\begin{align}
&
\pnorm{{N}_{x,\cdot} - {N}_{y,\cdot} }{2}^2
= \sum_{ z \in \mathcal{V} } | N_{x,z} - N_{y,z} |^2
\eqcom{i}= \sum_{ z \in \mathcal{V} } | T \Pi_x P_{x,z} - T \Pi_y P_{y,z} |^2
\nonumber \\ &
\eqcom{\ref{eqn:Definition_of_P}}= T^2 \sum_{k=1}^K \sum_{z \in \mathcal{V}_k} \Bigl| \bar{\Pi}_{\sigma(x)} \frac{ p_{\sigma(x),k} }{ \cardinality{\mathcal{V}_k} - \indicator{ \sigma(x) = k } } - \bar{\Pi}_{\sigma(y)} \frac{ p_{\sigma(y),k} }{ \cardinality{\mathcal{V}_k}  - \indicator{ \sigma(y) = k } } \Bigr|^2
\end{align}
and
\begin{align}
&
\pnorm{{N}_{\cdot,x} - {N}_{\cdot,y} }{2}^2
= \sum_{ z \in \mathcal{V} } | N_{z,x} - N_{z,y} |^2
\eqcom{i}= \sum_{ z \in \mathcal{V} } | T \Pi_z P_{z,x} - T \Pi_z P_{z,y} |^2
\nonumber \\ &
\eqcom{\ref{eqn:Definition_of_P}}= T^2 \sum_{k=1}^K \sum_{z \in \mathcal{V}_k} \Bigl| \bar{\Pi}_{k} \frac{ p_{k,\sigma(x)} }{ \cardinality{\mathcal{V}_{\sigma(x)}} - \indicator{ \sigma(x) = k } } - \bar{\Pi}_{k} \frac{ p_{k,\sigma(y)} }{ \cardinality{\mathcal{V}_{\sigma(y)}}  - \indicator{ \sigma(y) = k } } \Bigr|^2.
\end{align}
Then,
we have that
\begin{align}
&\pnorm{ {N}^\star_{x,\cdot} - {N}^\star_{y,\cdot} }{2}^2 =  \pnorm{ {N}_{x,\cdot} - {N}_{y,\cdot} }{2}^2 +\pnorm{ {N}_{\cdot,x} - {N}_{\cdot,y} }{2}^2 \cr
&\sim \frac{T^2}{n^3} \sum_{k=1}^K \Bigl( \Bigl( \frac{ \pi_{\sigma(x)} p_{\sigma(x),k} }{\alpha_k \alpha_{\sigma(x)} } - \frac{ \pi_{\sigma(y)} p_{\sigma(y),k} }{\alpha_k \alpha_{\sigma(y)} } \Bigr)^2 +\Bigl( \frac{ \pi_{k} p_{k,\sigma(x)} }{\alpha_k \alpha_{\sigma(x)} } - \frac{ \pi_{k} p_{k,\sigma(y)} }{\alpha_k \alpha_{\sigma(y)} } \Bigr)^2 \Bigr)\cr
&\geq \frac{T^2}{n^3} D(\alpha,p).
\end{align}
That completes this proof. \qed

\section{Proof of Lemma~\ref{lem:Concentration_between_Rhat_and_P}}
\label{suppl:Proof__Concentration_bound_relating_Rhat_and_P}

Recall that for any matrix $A \in \realNumbers^{n \times n}$ that $\pnorm{ A }{\mathrm{F}}^2 = \sum_{i=1}^n \sigma_i^2(A)$, and that for the spectral norm $\pnorm{A}{} = \max_{i=1,\ldots,n} \{ \sigma_i(A) \}$. Because both $\hat{R}$ and $N$ are of rank $K$, the matrix $\hat{R} - N$ is of rank at most $2K$, and therefore
\begin{equation}
\pnorm{ \hat{R}^\star - N^\star }{\mathrm{F}}^2  = 2 \pnorm{ \hat{R} - N }{\mathrm{F}}^2 
\leq 4K \pnorm{ \hat{R} - N }{}^2.
\end{equation}
By the triangle inequality it then follows that
\begin{equation}
\pnorm{ \hat{R}^\star - N^\star }{\mathrm{F}} 
\leq 2\sqrt{K} \bigl( \pnorm{ \hat{R} - \hat{N}_\Gamma }{} + \pnorm{ \hat{N}_\Gamma - N }{} \bigr).
\label{eqn:Upper_bound_on__Error_Rhat_P_after_triangle_inequality}
\end{equation}
Since $K$ is independent of $n$, we just need to bound $\pnorm{ \hat{R} - \hat{N}_\Gamma }{}$ using $\pnorm{ \hat{N}_\Gamma - N }{}$. From the definition of $\hat{R}$,
\begin{equation}
\pnorm{ \hat{N}_\Gamma - \hat{R} }{}
=  \sigma_{K+1}(\hat{N}_\Gamma).
\label{eqn:Upper_bound_on__Error_Rhat_Phat}
\end{equation}
Since the rank of $N$ is at most $K$, Weyl's theorem gives
\begin{align}\label{eqn:Upper_bound_on__Next_singular_value_of_Phat}
&
\sigma_{K+1}( \hat{N}_\Gamma )
\leq \pnorm{ \hat{N}_\Gamma - N }{}.
\end{align}
The proof is completed after bounding \refEquation{eqn:Upper_bound_on__Error_Rhat_P_after_triangle_inequality} by \refEquation{eqn:Upper_bound_on__Error_Rhat_Phat} and \refEquation{eqn:Upper_bound_on__Next_singular_value_of_Phat}. \qed

\section{Proof of Lemma~\ref{lem:Misclassification_separation_in_Rhat_and_P}}
\label{suppl:Proof__Separability_property_of_Rhat}

\revisedPartBegin
\paragraph{Preliminaries}
For notational convenience, let ${\bar{N}}^\star_k \triangleq ( 1 / \cardinality{ \mathcal{V}_k } ) \sum_{ z \in \mathcal{V}_k } {N}^\star_{z,\cdot}$ for $k = 1, \ldots, K$. Let $0 < a < 1/2$, $1 + a < b < \infty$ be two constants. Also recall the definitions of \emph{neighborhoods} in \refEquation{eqn:Definition_of_neighborhoods}:
\begin{equation}
\mathcal{N}_x 
\triangleq \Bigl\{ y \in \mathcal{V} \Big| \sqrt{ \pnorm{ {\hat{R}}_{x,\cdot} - {\hat{R}}_{y,\cdot} }{2}^2 + \pnorm{ {\hat{R}}_{\cdot,x} - {\hat{R}}_{\cdot,y} }{2}^2} \leq h_n \Bigr\}
\quad
\textrm{for}
\quad
x \in \mathcal{V}.
\end{equation}
Note that in \refEquation{eqn:Definition_of_neighborhoods} we specified $h_n = (1/n) \cdot ( T/n )^{3/2} \bigl( \ln{ (T/n) } \bigr)^{4/3}$, while in this proof we assume instead that $h_n$ satisfies $\omega( f_n^2 / n ) = h_n^2 = \smallO{ T^2 / n^3 }$. 

\paragraph{Approach}
We show that for any $0 < a < 1/2$ the recursive algorithm in \refEquation{eqn:SVD_centers} will (for sufficiently large $n,T$) give centers $z_1^\ast, \ldots, z_K^\ast$ satisfying
\begin{equation}
\pnorm{ {\hat{R}}^\star_{z_k^\ast,\cdot} - {\bar{N}}^\star_{\gamma(k)} }{2} 
< a h_n
\quad
\textrm{for}
\quad
k = 1, \ldots, K.
\label{eqn:Distance_of_the_centers_from_the_truths}
\end{equation}
for some permutation $\gamma$. Assuming \refEquation{eqn:Distance_of_the_centers_from_the_truths} holds, one finishes the proof by case checking: let $x \in \mathcal{E}$ be a misclassified state (necessarily $x \not\in \mathcal{N}_{z_{\sigma(x)}^\ast}$). 
\revisedPartEnd

\noindent
{Case 1:} If $x \in \mathcal{N}_{z_c^\ast}$ for some $c \neq \sigma(x)$, we have $\pnorm{ {\hat{R}}_{x,\cdot} - {\bar{N}}^\star_{c} }{2} \leq (1+a) h_n$ by \refEquation{eqn:Definition_of_neighborhoods} and \refEquation{eqn:Distance_of_the_centers_from_the_truths}. Together with \refLemma{lem:N_vector_error_scales_minimally_with_this_bound_if_x_y_differ_clusters}, this gives the lower bound
\begin{equation}
\pnorm{ {\hat{R}}^\star_{x,\cdot} - {\bar{N}}^\star_{\sigma(x)} }{2}
\eqcom{i}\geq \bigl| \pnorm{ {\hat{R}}^\star_{x,\cdot} - {\bar{N}}^\star_c }{2} - \pnorm{ {\bar{N}}^\star_c - {\bar{N}}^\star_{\sigma(x)} }{2} \bigr|
\geq \frac{T D^{1/2}(\alpha,p)}{n^{3/2}} - (1+a) h_n.
\end{equation}
By assumption $h_n = \smallO{ T / n^{3/2} }$, and the result in \refLemma{lem:Misclassification_separation_in_Rhat_and_P} follows. 

\noindent
{Case 2:} Otherwise $x \in ( \cup_{k=1}^K \mathcal{N}_{z_k^\ast} )^{\mathrm{c}}$, and the algorithm has associated $x$ to the closest (but incorrect) center via \refEquation{eqn:SVD_remainders}, i.e., to some cluster $c \neq \sigma(x)$ satisfying $\pnorm{ {\hat{R}}^\star_{z_c^\ast,\cdot} - {\hat{R}}^\star_{x,\cdot} }{2} \leq \pnorm{ {\hat{R}}^\star_{z_{\sigma(x)}^\ast,\cdot} - {\hat{R}}^\star_{x,\cdot} }{2}$. Because each center $z_k^\ast$ is $a h_n$ close to its truth ${\bar{N}}^\star_k$, which themselves are $\Omega( T / n^{3/2} )$ apart, it must be that $\pnorm{ {\hat{R}}^\star_{x,\cdot} - {\bar{N}}^\star_{\sigma(x)} }{2} = \Omega( T / n^{3/2} )$.

\revisedPartBegin
\paragraph{Proof of \refEquation{eqn:Distance_of_the_centers_from_the_truths}}
To prove \refEquation{eqn:Distance_of_the_centers_from_the_truths}, we will construct $K$ disjoint sets $\mathcal{C}_1, \ldots, \mathcal{C}_K$ that satisfy
\begin{equation}
\exists z \in \bigl( \cup_{l=1}^K \mathcal{C}_l \bigr) \backslash \cup_{l=0}^{k-1} C^{(l)}
: \cardinality{ \mathcal{N}_z } 
\geq \cardinality{ \mathcal{C}^{(k)} } 
\geq n \alpha_k ( 1 - \smallOP{1} ).
\label{eqn:Recursive_selection_property_of_the_core_sets}
\end{equation}
Here, the $\cardinality{ \mathcal{C}^{(1)} } \geq \ldots \geq \cardinality{ \mathcal{C}^{(K)} }$ denote the order statistic of the sets $\mathcal{C}_1, \ldots, \mathcal{C}_K$ based on their cardinalities, and $\alpha^{(1)} \geq \ldots \geq \alpha^{(K)}$ denote the order statistic for the cluster concentrations. The existence of sets $\mathcal{C}_1, \ldots, \mathcal{C}_K$ for which property \refEquation{eqn:Recursive_selection_property_of_the_core_sets} holds namely implies that it is impossible that any one of the centers $z_1^\ast, \ldots, z_K^\ast$ provided by the recursion in \refEquation{eqn:SVD_centers} is an \emph{outlier} when $n,T$ are sufficiently large.
\revisedPartEnd
Specifically, we define the sets of \emph{cores}:
\begin{equation}
\mathcal{C}_k
\triangleq \bigl\{ x \in \mathcal{V}_k \big| \pnorm{ {\hat{R}}^\star_{x,\cdot} - {\bar{N}}^\star_k }{2} \leq a h_n \bigr\}
\quad
\textrm{for}
\quad
k = 1, \ldots, K,
\label{eqn:Definition_of_cores_wrt_hn}
\end{equation}
i.e., states from cluster $k$ for which ${\hat{R}}^\star_{x,\cdot}$ is correctly close to cluster $k$'s center. We also define the set of \emph{outliers}:
\begin{equation}
\mathcal{O}
\triangleq \bigl\{ x \in \mathcal{V} \big| \pnorm{ {\hat{R}}^\star_{x,\cdot} - {\bar{N}}^\star_k }{2} \geq b h_n \textrm{ for all } k = 1, \ldots, K \bigr\},
\label{eqn:Definition_of_outliers_wrt_hn}
\end{equation}
so states for which ${\hat{R}}^\star_{x,\cdot}$ is far from \emph{any} cluster's center.  

Let $x \in \mathcal{O}$, $k \in \{ 1, \ldots, K \}$ and $y \in \mathcal{C}_k$. \jaron{The situation is schematically depicted in \refFigure{fig:Schematic_depiction_of_the_sets_of_cores_outliers_and_neighborhoods}.} By centering and then applying the reverse triangle inequality, we find
\begin{align}
\pnorm{ {\hat{R}}^\star_{x,\cdot} - {\hat{R}}^\star_{y,\cdot} }{2}
&
\geq \bigl| \pnorm{ {\hat{R}}^\star_{x,\cdot} - {\bar{N}}^\star_k }{2} - \pnorm{ {\hat{R}}^\star_{y,\cdot} - {\bar{N}}^\star_k }{2} \bigr|.
\label{eqn:Reverse_triangleq_ineqality_on_the_difference_between_rows_of_Rhat}
\end{align}
Since $x \in \mathcal{O}$ and $y \in \mathcal{C}_k$, it follows that $\pnorm{ {\hat{R}}^\star_{x,\cdot} - {\hat{R}}^\star_{y,\cdot} }{2} \geq (b-a) h_n$. Furthermore $b - a > 1$, implying that $y \not\in \mathcal{N}_x$. We have shown that $\mathcal{N}_x \cap \bigl( \cup_{k=1}^K \mathcal{C}_k \bigr) = \emptyset$ for all $x \in \mathcal{O}$. Consequentially:

\revisedPartBegin
\noindent
(a) for any $x \in \mathcal{O}$, $\cardinality{ \mathcal{N}_x } \leq \cardinality{ \bigl( \cup_{k=1}^K \mathcal{C}_k \bigr)^{\mathrm{c}} }$ since $\mathcal{N}_x \subseteq \bigl( \cup_{k=1}^K \mathcal{C}_k \bigr)^{\mathrm{c}}$.

\noindent
Furthermore:

\noindent
(b) for any $y \in \cup_{k=1}^K \mathcal{C}_k$, $\mathcal{C}_{\sigma(y)} \subseteq \mathcal{N}_y$ since $a < 1/2$, 

\noindent
(c) for $k \neq l$ and sufficiently large $n,T$, $\mathcal{C}_k \cap \mathcal{C}_l = \emptyset$ since $h_n = \smallO{ T / n^{3/2} }$.
\revisedPartEnd

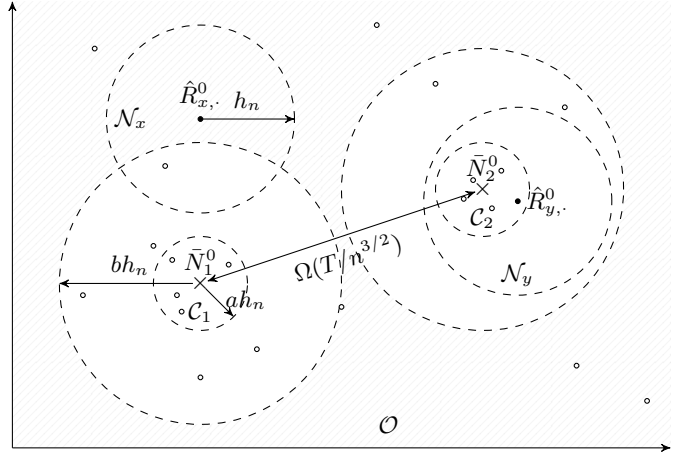
\begin{figure}[!hbtp]
\centering
\def\setCA{(-2cm,0cm) circle (0.5cm)}
\def\perCA{(-2cm,0cm) circle (1.5cm)}
\def\setCB{(1cm,1cm) circle (0.5cm)}
\def\perCB{(1cm,1cm) circle (1.5cm)}
\def\neighborhoodX{(-2cm,1.75cm) circle (1cm)}
\def\neighborhoodY{(1.375cm,0.875cm) circle (1cm)}
\def\canvas{(-4cm,-1.75cm) rectangle (3cm,3cm)}
\begin{tikzpicture}[scale=1.25]
    \begin{scope}
		\fill[draw=none, pattern=north east lines, opacity=0.2] \canvas;
    	\draw[->, thin, >=stealth'] (-4cm,-1.75cm) -- node [below] {} (3cm,-1.75cm);
    	\draw[->, thin, >=stealth'] (-4cm,-1.75cm) -- node [below] {} (-4cm,3cm);    		        	
    	
	    
	    \draw[dashed, fill=white] \perCA node [above] {};
		\draw[dashed, fill=white] \setCA node [above] {$\bar{N}^\star_1$};	    
	    \draw (-2cm,0cm) node {$\times$};	
    	\draw[->, thin, >=stealth', shorten <=0.1cm] (-2cm,0cm) -- node [right] {$a h_n$} (-1.65cm,-0.35cm);		
  
	    \draw[dashed, fill=white] \perCB node [above] {};
	    \draw[dashed, fill=white] \setCB node [above] {$\bar{N}^\star_2$};
	    \draw (1cm,1cm) node {$\times$};	    
    	\draw[->, thin, >=stealth', shorten <=0.1cm] (-2cm,0cm) -- node [above] {$b h_n$} (-3.5cm,0cm);	    	    
    
		\draw[dashed] \neighborhoodX node [above] {$\hat{R}^\star_{x,\cdot}$};    
	    \filldraw[draw=black, fill=black] (-2cm,1.75cm) circle (0.025cm);
    	\draw[->, thin, >=stealth'] (-2cm,1.75cm) -- node [above] {$h_n$} (-1cm,1.75cm);	    
		
    	\draw[<->, thin, >=stealth', shorten >=0.1cm, shorten <=0.1cm] (-2cm,0cm) -- node [below, rotate=17] {$\Omega(T/n^{3/2})$} (1cm,1cm);	    		
		
     	\draw (0cm,-1.5cm) node [] {$\mathcal{O}$};  		
     	\draw (-2cm,-0.5cm) node [above] {$\mathcal{C}_1$};  
     	\draw (1cm,0.5cm) node [above] {$\mathcal{C}_2$};    
     	\draw (-3cm,1.75cm) node [right] {$\mathcal{N}_x$};  	
     	\draw (1.375cm,-0.125cm) node [above] {$\mathcal{N}_y$};
     	
	    \filldraw[draw=black, fill=white] (-1.7cm,0.2cm) circle (0.025cm);
	    \filldraw[draw=black, fill=white] (-2.5cm,0.4cm) circle (0.025cm);
	    \filldraw[draw=black, fill=white] (-1.4cm,-0.7cm) circle (0.025cm);
	    \filldraw[draw=black, fill=white] (-2.3cm,0.25cm) circle (0.025cm);
	    \filldraw[draw=black, fill=white] (-2.2cm,-0.3cm) circle (0.025cm);	    	    	    

		\draw[dashed] \neighborhoodY;    	    
	    \filldraw[draw=black, fill=black] (1.375cm,0.875cm) circle (0.025cm) node[right] {$\hat{R}^\star_{y,\cdot}$}; 	    
	    \filldraw[draw=black, fill=white] (1.1cm,0.8cm) circle (0.025cm);
	    \filldraw[draw=black, fill=white] (1.2cm,1.2cm) circle (0.025cm);
	    \filldraw[draw=black, fill=white] (0.9cm,1.1cm) circle (0.025cm);
	    \filldraw[draw=black, fill=white] (0.8cm,0.9cm) circle (0.025cm);	    	    	    
		
	    \filldraw[draw=black, fill=white] (-2.25cm,-0.125cm) circle (0.025cm);  
	    \filldraw[draw=black, fill=white] (1.875cm,1.875cm) circle (0.025cm);
	    \filldraw[draw=black, fill=white] (0.5cm,2.125cm) circle (0.025cm);	        	     
	    \filldraw[draw=black, fill=white] (-0.5cm,-0.25cm) circle (0.025cm);
	    \filldraw[draw=black, fill=white] (-2.375cm,1.25cm) circle (0.025cm); 
	    \filldraw[draw=black, fill=white] (-3.125cm,2.5cm) circle (0.025cm);
	    \filldraw[draw=black, fill=white] (-0.125cm,2.75cm) circle (0.025cm);	    
	    \filldraw[draw=black, fill=white] (-2cm,-1cm) circle (0.025cm);
	    \filldraw[draw=black, fill=white] (2.75cm,-1.25cm) circle (0.025cm); 
	    \filldraw[draw=black, fill=white] (-3.25cm,-0.125cm) circle (0.025cm);
	    \filldraw[draw=black, fill=white] (2cm,-0.875cm) circle (0.025cm); 	    
    \end{scope}
\end{tikzpicture}
\caption{Schematic depictions of cores $\mathcal{C}_1, \mathcal{C}_2$ for $K=2$, the set of outliers $\mathcal{O}$ (shaded area), and neighborhoods $\mathcal{N}_x,\mathcal{N}_y$ for some $x \in \mathcal{O}$ and $y \in \mathcal{C}_2$.}
\label{fig:Schematic_depiction_of_the_sets_of_cores_outliers_and_neighborhoods}
\end{figure}

\jaron{We now provide estimates for the sizes of the sets involved in (a--c).} By (ii) \refLemma{lem:Concentration_between_Rhat_and_P}, and \refEquation{eqn:Definition_of_cores_wrt_hn},
\begin{align}
&
16 K \pnorm{ \hat{N}_\Gamma - N }{}^2
\eqcom{ii}\geq \pnorm{ \hat{R}^\star - N^\star }{\mathrm{F}}^2
= \sum_{ x \in \mathcal{V} } \pnorm{ {\hat{R}}^\star_{x,\cdot} - {\bar{N}}^\star_{\sigma(x)} }{2}^2
\\ &
\geq \cardinality{ \bigl( \cup_{k=1}^K \mathcal{C}_k \bigr)^{\mathrm{c}} } \min_{ x \in ( \cup_{k=1}^K \mathcal{C}_k )^{\mathrm{c}} } \bigl\{ \pnorm{ {\hat{R}}^\star_{x,\cdot} - {\bar{N}}^\star_{\sigma(x)} }{2}^2 \bigr\}
\eqcom{\ref{eqn:Definition_of_cores_wrt_hn}}\geq \cardinality{ \bigl( \cup_{k=1}^K \mathcal{C}_k \bigr)^{\mathrm{c}} } a^2 h_n^2.
\nonumber 
\end{align}
Rearrange to conclude that $\cardinality{ \bigl( \cup_{k=1}^K \mathcal{C}_k \bigr)^{\mathrm{c}} } = \bigOP{ f_n^2 / h_n^2 } = \smallOP{n}$. Similarly for any $k \in \{ 1, \ldots, K \}$,
$
16K  \pnorm{ \hat{N}_\Gamma - N }{}^2
\geq \cardinality{ \mathcal{C}_k^{\mathrm{c}} \cap \mathcal{V}_k } a^2 h_n^2 
$
such that $\cardinality{ \mathcal{C}_k } = \cardinality{ \mathcal{V}_k } - \cardinality{ \mathcal{C}_k^{\mathrm{c}} \cap \mathcal{V}_k } \geq n \alpha_k - \bigOP{ f_n^2 / h_n^2 } = n \alpha_k ( 1 - \smallOP{1} )$. \jaron{Together with (a--c), this establishes the existence of $\mathcal{C}_1, \ldots, \mathcal{C}_K$ such that \refEquation{eqn:Recursive_selection_property_of_the_core_sets} holds.} \qed

\revisedPartEnd

\chapter{Proofs of \refChapter{sec:The_cluster_improvement_algorithm}}

\section{Bounding the size of $\mathcal{H}^{\mathrm{c}} = \mathcal{V} \backslash \mathcal{H}$ (Proof of  \refProposition{prop:Set_of_forlorn_states_shrinks})}
\label{suppl:Proof__Size_of_complement_of_H}

\revisedPartBegin
First note that the number of states not in $\Gamma$ (obtained after the trimming process) is negligible, i.e., $n\exp(-{T\over n}\ln({T\over n}))$. We then upper bound the number of states that do not satisfy (H1). Let $x\in \mathcal{V}_i$. If $x$ does not satisfy (H1), there exists $j\neq i$ such that $\hat{I}_{i,j}(x)<{T\over 2n}I(\alpha,p)$, where 
$$
\hat{I}_{i,j}(x) \triangleq \sum_{k=1}^K \Bigl( \hat{N}_{x, \mathcal{V}_k} \ln{ \frac{p_{i,k}}{p_{j,k}} } + \hat{N}_{\mathcal{V}_k, x} \ln{ \frac{p_{k,i} \alpha_j}{p_{k,j} \alpha_{i}} } \Bigr) 
+  \Bigl(\frac{\hat{N}_{ \mathcal{V}_j, \mathcal{V}}}{\alpha_j n} - \frac{\hat{N}_{ \mathcal{V}_{i}, \mathcal{V}}}{\alpha_{i} n} \Bigr).
$$
We have $\mathbb{E}[\hat{I}_{i,j}(x) ] = {T\over n}I_{i,j}(\alpha,p)$ where $I_{i,j}(\alpha,p)$ is the quantity involved in the definition of $I(\alpha,p)$ (\ref{eqn:Ialphabetap}). In particular, $\mathbb{E}[\hat{I}_{i,j}(x) ] \ge {T\over n}I(\alpha,p)$. Hence, $x$ does not satisfy (H1) implies that for some $j\neq i$, $\hat{I}_{i,j}(x) < \frac{T}{2n} I(\alpha,p)$ and $\mathbb{E}[\hat{I}_{i,j}(x) ] \ge {T\over n}I(\alpha,p)$. Using concentration results for the BMC \refAppendixSection{sec:Concentration_inequalities_for_BMCs} and more precisely (\ref{eq:KLsim}), this event happens with probability at most $\exp{} \bigl( - C \allowbreak {T\over n} I(\alpha,p) \bigr)$ with $C = \alpha_{\min}^2 / ( 720 \eta^3 \alpha_{\max}^2 )$. Hence, the expected number of states not satisfying (H1) is bounded by $n \exp{} \bigl( - C \frac{T}{n} I(\alpha,p) \bigr)$. From there, using Markov inequality, we obtain that the number of states not satisfying (H1) does not exceed $n \exp{} \bigl( - C \frac{T}{n} I(\alpha, p) \bigr)$ with high probability.

Next, we prove the following intermediate claim:

\begin{lemma}\label{lem:ssizecond} Define 
$
s 
\triangleq \bigl\lfloor 2 n \exp{ \bigl( - \frac{\alpha_{\min}^2}{720 \eta^3 \alpha_{\max}^2} (T/n) I(\alpha, p) \bigr) } \bigr\rfloor
$. 
If $s \geq 1$, then with high probability there does not exist a subset $\mathcal{S}\subset \mathcal{V}$ of size $\cardinality{ \mathcal{S} } = s$ such that $\hat{N}_{\mathcal{S},\mathcal{S}} \geq s \ln (T/n)^2$.
\end{lemma}

\begin{proof}
Let $\mathcal{S} \subset \mathcal{V}$ be such that $|\mathcal{S}| = s$ with $s$ as above. We decompose $\hat{N}_{\mathcal{S},\mathcal{S}}$ as $\hat{N}_{\mathcal{S},\mathcal{S}} = \hat{N}^{(\textrm{even})}_{\mathcal{S},\mathcal{S}} + \hat{N}^{(\textrm{odd})}_{\mathcal{S},\mathcal{S}}$ where
$
\hat{N}^{(\textrm{even})}_{\mathcal{S},\mathcal{S}} 
\triangleq \sum_{t=0}^{\lceil T/2\rceil -1} \indicator{ X_{2t}\in {\cal S}, X_{2t+1}\in {\cal S}}
$,
and similarly
$
\hat{N}^{(\textrm{odd})}_{\mathcal{S},\mathcal{S}} 
\triangleq \sum_{t=0}^{\lfloor T/2\rfloor -1} \indicator{ X_{2t+1}\in {\cal S}, X_{2t+2}\in {\cal S}}
$.
This decomposition has the same purpose as that in the proofs presented in \textsection\ref{sec:Spectral_analysis_for_BMCs}, i.e., to get tight bounds on $\hat{N}_{\mathcal{S},\mathcal{S}}$. More precisely, we have
\begin{align} 
\mathbb{P}\Bigl[ \hat{N}^{(\textrm{even})}_{\mathcal{S},\mathcal{S}} & \geq \frac{s}{2} \ln{ \Bigl( \frac{T}{n} \Bigr)^2 } \Bigr]
\eqcom{i}\leq \inf_{\lambda \ge  0} \frac{ \expectation{ \e{ \lambda \hat{N}^{(\textrm{even})}_{\mathcal{S},\mathcal{S}} } } }{ \e{ (s/2) \lambda \ln ( T/n)^2 } }
\eqcom{ii}\leq  \inf_{\lambda \ge 0} \prod_{i=1}^{\lceil T/2 \rceil -1} \frac{ 1+ (s/n)^2 p_{\max} \e{\lambda} }{ \e{ (s/2)\lambda \ln ( T/n)^2 } }
\nonumber \\ &
\leq \inf_{\lambda \ge 0} \exp{ \Bigl( \frac{s^2 p_{\max} T}{n^2} \e{\lambda} - \frac{s}{2} \lambda \ln{ \Bigl( \frac{T}{n} \Bigr)^2 } \Bigr) }
\nonumber \\ &
\eqcom{iii}\leq \exp{ \Bigl( - \frac{T}{n} s \Bigl( \frac{1}{2}\ln{ \frac{T}{n} } - \frac{s p_{\max} }{n} \e{ \frac{T/n}{ \ln{ (T/n) } } } \Bigr) \Bigr) }
\eqcom{iv}\leq \e{ - \frac{ T s \ln{ (T/n) } }{4n} }, \label{eq:bndss}
\end{align}
\jaron{for sufficiently large $n,T$.} 
\revisedPartEnd
Here, (i) is obtained by applying Markov's inequality. (ii) by directly bounding the transition probabilities -- recall that $p_{\max}$ is defined so that $( \mathrm{diag}( P_{x,\cdot} ) P )_{y,z} \leq p_{\max} / n^2$ for all $x, y, z \in \mathcal{V}$, see \textsection\ref{sec:Spectral_analysis_for_BMCs} for details. (iii) is obtained by specifying $\lambda = ( T/n ) / ( \ln{(T/n)} )$. Finally to get (iv), we used the fact that $n/s \geq \exp{ \bigl( (T/n) / ( \ln{(T/n)} ) }$. Analogously, one can prove that
\begin{equation} 
\probabilityBig{ \hat{N}^{(\textrm{odd})}_{\mathcal{S},\mathcal{S}} \geq \frac{s}{2} \ln{ \Bigl( \frac{T}{n} \Bigr)^2 } } 
\leq \exp{ \Bigl( - \frac{ T s \ln{(T/n)} }{4n} \Bigr) }. \label{eq:bndss2}
\end{equation}

Because the number of subsets $\mathcal{S} \subset \mathcal{V}$ of size $s$ satisfies
${{n}\choose{s}} \leq ( \e{} n / s )^{s}$ we deduce using \refEquation{eq:bndss} and \refEquation{eq:bndss2} that
\begin{align} 
&
\expectationBig{ \Bigl| \Bigl\{ \mathcal{S} : \hat{N}_{\mathcal{S},\mathcal{S}} \geq s \ln{ \Bigl( \frac{T}{n} \Bigr)^2 }, |\mathcal{S}|= s \Bigl\} \Bigr| } 
\nonumber \\ &
\leq \expectationBig{ \Bigl| \Bigl\{ \mathcal{S} : \hat{N}^{(\textrm{even})}_{\mathcal{S},\mathcal{S}} \geq \frac{s}{2} \ln{ \Bigl( \frac{T}{n} \Bigr)^2 }, |\mathcal{S}| = s \Bigr\} \Bigr| } 
\nonumber \\ &
\phantom{=} + \expectationBig{ \Bigl| \Bigl\{ \mathcal{S} : \hat{N}^{(\textrm{odd})}_{\mathcal{S},\mathcal{S}} \geq \frac{s}{2} \ln{ \Bigl( \frac{T}{n} \Bigr)^2 }, |\mathcal{S}|= s \Bigr\} \Bigr| } 
\nonumber \\ &
\leq 2 \Bigl( \frac{ \e{}n }{s} \Bigr)^s \e{ - \frac{T s \ln (T/n)}{4n} } 
= 2 \e{ -s \bigl( \frac{T \ln (T/n)}{4n} - \ln{ \frac{ \e{} n }{s} } \bigr) }
\leq 2 \e{ - \frac{Ts \ln (T/n)}{8n} }
\end{align}
\jaron{for sufficiently large $n,T$.} Using Markov's inequality, we can now conclude that with high probability there does not exist a subset $\mathcal{S} \subset \mathcal{V}$ of size $|\mathcal{S}|= s$  such that \jaron{$\hat{N}_{\mathcal{S},\mathcal{S}} \ge s \ln{ (T/n)^2 }$}.
\end{proof}

\revisedPartBegin
We then complete the proof of \refProposition{prop:Set_of_forlorn_states_shrinks} applying the following argument. Consider this iterative construction: start with the set $Z(0)$ of all states in $\Gamma$ that do not satisfy (H1). The $t$-th iteration consists of adding to $Z(t-1)$ a state $v$ not satisfying (H2) written w.r.t. $Z(t-1)$, i.e., $\hat{N}_{v,Z(t-1)} +\hat{N}_{Z(t-1),v} > 2 \ln{ ( (T/n)^2 ) }$. If such a state does not exist, the construction ends. Let $Z(t^*)$ be the final set: $t^*$ is the number of iterations before the construction stops. By definition of ${\cal H}$, the size of ${\cal V}\setminus Z(t^*)$ is smaller than that of ${\cal H}$, and thus $|{\cal H}^{\mathrm{c}}| \le |Z(t^*)|= |Z(0)| + t^*$. 
Note that with high probability $|Z(0)| \le s/2$. 

Finally, we show that $|Z(0)| + t^* \le s$ with high probability when $|Z(0)| \le s/2$.
We first consider the case when $s=0$, i.e., $Z(0) = \emptyset$. When $Z(0) = \emptyset$, $t^* = 0$ since  $\hat{N}_{x,\emptyset}+\hat{N}_{\emptyset,x}=0 <2\ln((T/n)^2)$ for every $x \in \mathcal{V}$. Therefore, $|{\cal H}^{\mathrm{c}}| =  s$. 
We then consider the case when $s\ge 1$. By construction, we have $\hat{N}_{Z(t) ,Z(t)} \allowbreak \geq 2t \ln (T/n)^2$ and $\cardinality{Z(t)} \le s/2 + t$ because of the sequence generating rule. Further observe that $Z(t^*)$ is such that for all $x\notin Z(t^*)$, $x$ satisfies (H1) and (H2) written w.r.t. $Z(t^*)$.  Assume that $t^* \ge s/2$. When $t = s/2$, the set $Z(s/2)$ satisfies that $\hat{N}_{Z(s/2) ,Z(s/2)} \ge  s \ln (T/n)^2$ and $\cardinality{Z(s/2)} \le s$. From Lemma~\ref{lem:ssizecond}, however, there does not exist a subset $\mathcal{S} \subset \mathcal{V}$ of size $|\mathcal{S}|= s$ such that $\hat{N}_{\mathcal{S},\mathcal{S}} \ge  s \ln{ (T/n)^2 }$ with high probability. Therefore,  $|{\cal H}^{\mathrm{c}}| \le |Z(t^*)|\le s/2 + t^*<s$ with high probability.
\revisedPartEnd

\section{Proof of Lemma~\ref{lem:Leading_behavior_of_E1_E2_E3_E4}}
\label{supple:Proof_EEE}

\subsection{Leading order behavior of $E_1$}
\label{suppl:Proof__Leading_order_behavior_of_E1}

The statement $- E_1 = \Omega_{\mathbb{P}}\bigl( I(\alpha,p) \allowbreak (T/n) \itr{e_n}{t+1} \bigr)$ is a direct consequence of condition (H1). Specifically, $x \in \itr{\EHset}{t+1}$ implies that $x \in \mathcal{H}$, and therefore that condition (H1) is satisfied for $x$. Summing condition (H1) over all $x \in \itr{\EHset}{t+1}$ yields the desired result.

\subsection{Leading order behavior of $E_2$}
\label{suppl:Proof__Leading_order_behavior_of_E2}

By assumption, we can bound all the ratios of the type ${p_{a,k}\over p_{b,k}}$ or ${p_{k,a}\over p_{k,b}}$ involved in $E_2$ by $\eta$. The triangle inequality then yields:

%
%

\begin{equation}
| E_2 |
\leq \ln{\eta}  \Bigl( \sum_{ x \in \itr{\EHset}{t+1} } \sum_{k=1}^K \bigl| \hat{N}_{x,\itr{\hat{\mathcal{V}}}{t}_k} - \hat{N}_{x,\mathcal{V}_k} \bigr| + \sum_{ x \in \itr{\EHset}{t+1} } \sum_{k=1}^K \bigl| \hat{N}_{\itr{\hat{\mathcal{V}}}{t}_k,x} - \hat{N}_{\mathcal{V}_k,x} \bigr| \Bigr).
\label{eqn:Intermediate_bound_on_E2}
\end{equation}

Let us upper bound the first summation in the r.h.s.\ of \refEquation{eqn:Intermediate_bound_on_E2}. Using the triangle inequality and the nonnegativity of entries of $\hat{N}_{x,y}$, we get
\begin{align}
&
\sum_{ x \in \itr{\EHset}{t+1} } \sum_{k=1}^K \bigl| \hat{N}_{x,\itr{\hat{\mathcal{V}}}{t}_k} - \hat{N}_{x,\mathcal{V}_k} \bigr|
\leq \sum_{ x \in \itr{\EHset}{t+1} } \sum_{k=1}^K \Bigl| \hat{N}_{x, \itr{\hat{\mathcal{V}}}{t}_k \cap \mathcal{H} } - \hat{N}_{x, \mathcal{V}_k \cap \mathcal{H} } \Bigr|
\nonumber \\ &
\phantom{\sum_{ x \in \itr{\EHset}{t+1} } \sum_{k=1}^K \bigl| \hat{N}_{x,\itr{\hat{\mathcal{V}}}{t}_k} - \hat{N}_{x,\mathcal{V}_k} \bigr| \leq} + \sum_{ x \in \itr{\EHset}{t+1} } \sum_{k=1}^K \Bigl( \hat{N}_{x, \itr{\hat{\mathcal{V}}}{t}_k \cap \mathcal{H}^{\mathrm{c}} } + \hat{N}_{x, \mathcal{V}_k \cap \mathcal{H}^{\mathrm{c}} } \Bigr)
\nonumber \\ &
= \sum_{ x \in \itr{\EHset}{t+1} } \sum_{k=1}^K \Bigl| \hat{N}_{x, \itr{\hat{\mathcal{V}}}{t}_k \cap \mathcal{H} } - \hat{N}_{x, \mathcal{V}_k \cap \mathcal{H} } \Bigr| 
+ 2 \sum_{ x \in \itr{\EHset}{t+1} } \hat{N}_{x, \mathcal{V} \backslash \mathcal{H} }. 
\label{eqn:Intermediate_equation_1_for_bounding_E2}
\end{align}
To further upper bound the first summation in \refEquation{eqn:Intermediate_equation_1_for_bounding_E2}, observe that 
\begin{align}
&
\sum_{ x \in \itr{\EHset}{t+1} } \sum_{k=1}^K \bigl| \hat{N}_{x,\itr{\hat{\mathcal{V}}}{t}_k \cap \mathcal{H}} - \hat{N}_{x,\mathcal{V}_k \cap \mathcal{H}} \bigr|
= \sum_{ x \in \itr{\EHset}{t+1} } \sum_{k=1}^K \Bigl| \sum_{ y \in \itr{\hat{\mathcal{V}}}{t}_k \cap \mathcal{H} } \hat{N}_{x,y} - \sum_{ y \in \mathcal{V}_k \cap \mathcal{H} } \hat{N}_{x,y} \Bigr|
\\ &
= \sum_{ x \in \itr{\EHset}{t+1} } \sum_{k=1}^K \Bigl| \sum_{ y \in ( \itr{\hat{\mathcal{V}}}{t}_k \cap \mathcal{H} ) \backslash ( \mathcal{V}_k \cap \mathcal{H} ) } \hat{N}_{x,y} - \sum_{ y \in ( \mathcal{V}_k \cap \mathcal{H} ) \backslash ( \itr{\hat{\mathcal{V}}}{t}_k \cap \mathcal{H} ) } \hat{N}_{x,y} \Bigr|
\nonumber \\ &
\leq \sum_{ x \in \itr{\EHset}{t+1} } \sum_{k=1}^K \sum_{ y \in ( \itr{\hat{\mathcal{V}}}{t}_k \cap \mathcal{H} ) \Delta ( \mathcal{V}_k \cap \mathcal{H} ) } \hat{N}_{x,y}
= 2 \sum_{ x \in \itr{\EHset}{t+1} } \sum_{ y \in \itr{\EHset}{t} } \hat{N}_{x,y}.
\nonumber 
\end{align}
For the second summation in \refEquation{eqn:Intermediate_equation_1_for_bounding_E2}, note that since $x \in \itr{\EHset}{t+1}$ implies that $x \in \mathcal{H}$, it follows from condition (H2) that $\sum_{ x \in \itr{\EHset}{t+1} } \hat{N}_{x, \mathcal{V} \backslash \mathcal{H} } \leq 2 \cardinality{ \itr{\EHset}{t+1} } \jaron{ \ln{ ( (T/n)^2 ) } }$.

Now, aside from swapping the indices, the conclusion holds similarly for the second summation in \refEquation{eqn:Intermediate_bound_on_E2}. We thus conclude that
\begin{align}
| E_2 |
&
\leq \jaron{ 2 \ln{ \eta }  } \Bigl( \hat{N}_{ \itr{\EHset}{t+1}, \itr{\EHset}{t} } + \hat{N}_{ \itr{\EHset}{t}, \itr{\EHset}{t+1} } + 2 \cardinality{ \itr{\EHset}{t+1} } \ln{( (T/n)^2 )} \Bigr).
\label{eqn:Intermediate_equation_2_for_bounding_E2}
\end{align}

Next to bound the two first terms of the r.h.s.\ of \refEquation{eqn:Intermediate_equation_2_for_bounding_E2}, we center both terms around their means. Since the Markov chain is in equilibrium by assumption, it holds for the first term that
\begin{align}
\hat{N}_{ \itr{\EHset}{t+1}, \itr{\EHset}{t} }
&
= N_{ \itr{\EHset}{t+1}, \itr{\EHset}{t} } + \hat{N}_{ \itr{\EHset}{t+1}, \itr{\EHset}{t} } - N_{ \itr{\EHset}{t+1}, \itr{\EHset}{t} } 
\\ &
\leq \max_{x,y} \{ T \Pi_x P_{x,y} \} \cardinality{ \itr{\EHset}{t} } \cardinality{ \itr{\EHset}{t+1} } + \hat{N}_{ \itr{\EHset}{t+1}, \itr{\EHset}{t} } - N_{ \itr{\EHset}{t+1}, \itr{\EHset}{t} }.
\nonumber
\end{align}
Applying \refLemma{lem:Summation_to_innerproduct_matrix_notation} presented in \refAppendixSection{sec:Spectral_norm_bound_for_sums_of_elements_of_matrices}, we find that
\begin{equation}
\hat{N}_{ \itr{\EHset}{t+1}, \itr{\EHset}{t} } - N_{ \itr{\EHset}{t+1}, \itr{\EHset}{t} }
= \transpose{ \vectOnes{ \itr{\EHset}{t+1} } } ( \hat{N} - N ) \vectOnes{ \itr{\EHset}{t} }
\leq \pnorm{ \hat{N} - N }{} \sqrt{ \cardinality{ \itr{\EHset}{t} } \cardinality{ \itr{\EHset}{t+1} } }. 
\label{eqn:Nhat_Etp1_Et_minus_N_Etp1_Et_spectral_bound}
\end{equation}
The same conclusion holds for $\hat{N}_{ \itr{\EHset}{t}, \itr{\EHset}{t+1} } - N_{ \itr{\EHset}{t}, \itr{\EHset}{t+1} }$. 

Summarizing, we have shown that
\begin{align}
| E_2 |
\leq 4 \ln{\eta} \Bigl( & \underbrace{\max_{x,y} \{ T \Pi_x P_{x,y} \} \cardinality{ \itr{\EHset}{t} } \cardinality{ \itr{\EHset}{t+1} } }_{\triangleq F_1}
\nonumber \\ &
+ \underbrace{ \pnorm{ \hat{N} - N }{} \sqrt{ \cardinality{ \itr{\EHset}{t} } \cardinality{ \itr{\EHset}{t+1} } } }_{\triangleq F_2}
+ \underbrace{ \cardinality{ \itr{\EHset}{t+1} }  \ln{ ((T/n)^2) } }_{\triangleq F_3} \Bigr).
\end{align}

Recall that $\Pi_x P_{x,y} = \bigO{ 1 / n^2 }$, and hence $F_1=\bigOP{ (T/n) (\itr{e_n}{t}/n) \itr{e_n}{t+1} }$. Next in view of Proposition \ref{prop:Spectral_concentration_bound_for_BMCs}, $\pnorm{\hat{N}-N}{} = \bigOP{ \sqrt{ (T/n) \ln{(T/n)}} }$ so that $F_2 = \bigOP{ \sqrt{{(T/n)}\ln{(T/n)} \itr{e_n}{t} \itr{e_n}{t+1} }}$. $F_3 =\bigOP{  \ln{(T^2/n^2)} \itr{e_n}{t+1}}$ is immediate. This completes the analysis of $E_2$.

%
%
%
%
%
%

\subsection{Leading order behavior of $U$}

We write $U = E_3 + E_4$ where $E_3 = E_3^{\mathrm{out}} + E_3^{\mathrm{in}}$,
\begin{gather}
E_3^{\mathrm{out}}
= \sum_{ x \in \itr{\EHset}{t+1} } \sum_{k=1}^K  
\hat{N}_{x,\itr{\hat{\mathcal{V}}}{t}_k} \Bigl( \ln{ \frac{ \hat{p}_{ \itr{\sigma}{t+1}(x), k } }{ \hat{p}_{ \sigma(x), k } } } - \ln{ \frac{ p_{ \itr{\sigma}{t+1}(x), k } }{ p_{ \sigma(x), k } } } \Bigr),
\nonumber \\
E_3^{\mathrm{in}}
= \sum_{ x \in \itr{\EHset}{t+1} } \sum_{k=1}^K  
\hat{N}_{\itr{\hat{\mathcal{V}}}{t}_k,x} \Bigl( \ln{ \frac{ \hat{p}_{ k, \itr{\sigma}{t+1}(x) } }{ \hat{p}_{ k, \sigma(x) } } } - \ln{ \frac{ p_{ k, \itr{\sigma}{t+1}(x) } }{ p_{ k, \sigma(x) } } } \Bigr).
\end{gather}
and
\begin{equation}
E_4
= \sum_{ x \in \itr{\EHset}{t+1} } \Bigl( \frac{ \hat{N}_{ \itr{\hat{\mathcal{V}}_{\sigma(x)}}{t}, \mathcal{V} } }{ \cardinality{ \itr{\hat{\mathcal{V}}_{\sigma(x)}}{t} } } - \frac{ \hat{N}_{ \mathcal{V}_{\sigma(x)}, \mathcal{V} } }{ \cardinality{ \mathcal{V}_{\sigma(x)} } } \Bigr) 
- \sum_{ x \in \itr{\EHset}{t+1} } \Bigl( \frac{ \hat{N}_{ \itr{\hat{\mathcal{V}}_{ \itr{\sigma}{t+1}(x) }}{t}, \mathcal{V} } }{ \cardinality{ \itr{\hat{\mathcal{V}}_{ \itr{\sigma}{t+1}(x) }}{t} } } - \frac{ \hat{N}_{ \mathcal{V}_{ \itr{\sigma}{t+1}(x) }, \mathcal{V} } }{ \cardinality{ \mathcal{V}_{ \itr{\sigma}{t+1}(x) } } } \Bigr).
\label{eqn:Expression_for_E4}
\end{equation}


\begin{lemma}\label{lemE3}
We have $| E_3 |
=  \bigOPbig{  \itr{e_n}{t+1}  \sqrt{\frac{T}{n}}\ln{ \frac{T}{n} } + \itr{e_n}{t+1} \frac{ \itr{e_n}{t} }{n}\frac{T}{n} \ln \frac{T}{n}}
$.
\end{lemma}

\begin{proof}
By the triangle inequality, we have $ | E_3 | \leq | E_3^{\mathrm{out}} | + | E_3^{\mathrm{in}} |$ with
\begin{align}
| E_3^{\mathrm{out}} | 
&
\leq \sum_{ x \in \itr{\EHset}{t+1} } \sum_{k=1}^K  
\hat{N}_{\itr{\hat{\mathcal{V}}}{t}_k,x} \Bigl( \Bigl| \ln{ \frac{ \hat{p}_{ k,\itr{\sigma}{t+1}(x) } }{ p_{ k,\itr{\sigma}{t+1}(x) } } } \Bigr| + \Bigl| \ln{ \frac{ \hat{p}_{ k,\sigma(x) } }{ p_{ k,\sigma(x) } } } \Bigr| \Bigr),
\nonumber \\
| E_3^{\mathrm{in}} |
&
\leq
\sum_{ x \in \itr{\EHset}{t+1} } \sum_{k=1}^K  
\hat{N}_{x,\itr{\hat{\mathcal{V}}}{t}_k} \Bigl( \Bigl| \ln{ \frac{ \hat{p}_{ \itr{\sigma}{t+1}(x), k } }{ p_{ \itr{\sigma}{t+1}(x), k } } } \Bigr| + \Bigl| \ln{ \frac{ \hat{p}_{ \sigma(x), k } }{ p_{ \sigma(x), k } } } \Bigr| \Bigr).
\label{eqn:Abs_bounds_on_E3in_and_E3out}
\end{align}

\revisedPartBegin
We first bound the summands. From the inequalities $x / ( 1 + x ) \leq \ln{ (1+x) } \allowbreak \leq x$ for $x > -1$, it follows that for $a,b = 1, \ldots, K$,
\begin{align}
\Bigl| \ln{ \frac{ \hat{p}_{a,b} }{ p_{a,b} } } \Bigr|
&
= \Bigl| \ln{ \Bigl( 1 + \frac{ \hat{p}_{a,b} - p_{a,b} }{ p_{a,b} } \Bigr) } \Bigr|
\leq \Bigl| \frac{ \hat{p}_{a,b} - p_{a,b} }{ p_{a,b} } \Bigr|
\nonumber \\ &
\leq \Bigl| \frac{ 1 }{ p_{a,b} } \frac{\hat{N}_{\hat{\mathcal{V}}_a^{[t]}, \hat{\mathcal{V}}_b^{[t]}} }{\hat{N}_{\hat{\mathcal{V}}_a^{[t]},\mathcal{V}}}- 1 \Bigr| =\Bigl| \frac{ 1 }{ p_{a,b} } \frac{N_{\hat{\mathcal{V}}_a^{[t]}, \hat{\mathcal{V}}_b^{[t]}} }{N_{\hat{\mathcal{V}}_a^{[t]},\mathcal{V}}} \frac{\hat{N}_{\hat{\mathcal{V}}_a^{[t]}, \hat{\mathcal{V}}_b^{[t]}} }{N_{\hat{\mathcal{V}}_a^{[t]}, \hat{\mathcal{V}}_b^{[t]}}} \frac{N_{\hat{\mathcal{V}}_a^{[t]},\mathcal{V}} }{\hat{N}_{\hat{\mathcal{V}}_a^{[t]},\mathcal{V}}}- 1 \Bigr| 
\nonumber \\ &
\leq \Bigl| \frac{ N_{\hat{\mathcal{V}}_a^{[t]}, \hat{\mathcal{V}}_b^{[t]}} }{ N_{\mathcal{V}_a, \mathcal{V}_b} } \frac{N_{\mathcal{V}_a,\mathcal{V}} }{N_{\hat{\mathcal{V}}_a^{[t]},\mathcal{V}}} \frac{\hat{N}_{\hat{\mathcal{V}}_a^{[t]}, \hat{\mathcal{V}}_b^{[t]}} }{N_{\hat{\mathcal{V}}_a^{[t]}, \hat{\mathcal{V}}_b^{[t]}}} \frac{N_{\hat{\mathcal{V}}_a^{[t]},\mathcal{V}} }{\hat{N}_{\hat{\mathcal{V}}_a^{[t]},\mathcal{V}}}- 1 \Bigr| 
= \bigOPbig{\frac{\itr{e_n}{t} }{n} + \sqrt{\frac{n}{T}}},
\label{eqn:Taylor_bound_on_ln_phat_over_p}
\end{align}
where the last equality is obtained from the following observations: the four ratios involved in the last inequality are all close to 1. The two first ratios capture the error due to the fact that our estimated clusters are not the true clusters, i.e., ${\cal V}_a^{[t]}\neq {\cal V}_a$. However, by assumption, we have 
$
\bigl| { N_{\hat{\mathcal{V}}_a^{[t]}, \hat{\mathcal{V}}_b^{[t]}} } / { N_{\mathcal{V}_a, \mathcal{V}_b} } - 1 \bigr| 
= \bigOP{ {\itr{e_n}{t}} / {n} }
$.
The same inequality holds for the second ratio. To control the two last ratios, we use the concentration result \eqref{eq:maxAB} (this result is uniform over all subsets of states) to get for example:
$
\bigl| {\hat{N}_{\hat{\mathcal{V}}_a^{[t]}, \hat{\mathcal{V}}_b^{[t]}} } / {N_{\hat{\mathcal{V}}_a^{[t]}, \hat{\mathcal{V}}_b^{[t]}}} - 1 \big| 
= \bigOP{\sqrt{n/T}}
$.
\revisedPartEnd

Next, since $x \in \itr{ \EHset }{t+1}$ implies that $x \in \Gamma$, it follows from \refEquation{eqn:Concentration_result_on_the_maximum_row_and_column_sum_over_Nhat_Gamma} that
\begin{equation}
\hat{N}_{x,\mathcal{V}} + \hat{N}_{\mathcal{V},x} 
= \bigObig{ \frac{T}{n} \ln \frac{T}{n} }. \label{eq:E3numedge}
\end{equation}
with high probability. Then from \refEquation{eqn:Taylor_bound_on_ln_phat_over_p} and \refEquation{eq:E3numedge} it follows that
\begin{align}
| E_3 | = 
& \bigOPbig{ \itr{e_n}{t+1} \sqrt{\frac{T}{n}}\ln{ \frac{T}{n} } +\itr{e_n}{t+1} \frac{\itr{e_n}{t}}{n}\frac{T}{n} \ln \frac{T}{n}}.
\end{align}
This completes the proof.
\end{proof}


\begin{lemma}
We have
$
| E_4 |
= \bigOPbig{ \sqrt{\frac{T}{n}} \itr{e_n}{t+1} +\frac{\itr{e_n}{t} }{n}\frac{T}{n}\itr{e_n}{t+1} }
$.
\end{lemma}

\begin{proof}
Let $k \in \{ 1, \ldots, K \}$ to examine any one of the summands in $E_4$. We (i) center and use the triangle inequality to bound all summands as
\begin{align}
\Bigl| \frac{ \hat{N}_{ \itr{\hat{\mathcal{V}}_k}{t}, \mathcal{V} } }{ \cardinality{ \itr{\hat{\mathcal{V}}_k}{t} } } - \frac{ \hat{N}_{ \mathcal{V}_k, \mathcal{V} } }{ \cardinality{ \mathcal{V}_k } } \Bigr|
\eqcom{i}\leq & 
\Bigl| \frac{ \hat{N}_{ \itr{\hat{\mathcal{V}}_k}{t}, \mathcal{V} } -N_{ \itr{\hat{\mathcal{V}}_k}{t}, \mathcal{V} } }{ \cardinality{ \itr{\hat{\mathcal{V}}_k}{t} } } \Bigr| 
+ \Bigl|\frac{ \hat{N}_{ \mathcal{V}_k, \mathcal{V} } -N_{ \mathcal{V}_k, \mathcal{V} } }{ \cardinality{ \mathcal{V}_k } } \Bigr| 
\nonumber \\ &
+ \Bigl| \frac{ N_{ \itr{\hat{\mathcal{V}}_k}{t}, \mathcal{V} } }{ \cardinality{ \itr{\hat{\mathcal{V}}_k}{t} } } - \frac{ N_{ \mathcal{V}_k, \mathcal{V} } }{ \cardinality{ \mathcal{V}_k } } \Bigr|
= \jaron{ \bigOPbig{ \sqrt{\frac{T}{n}}  +\frac{\itr{e_n}{t} }{n}\frac{T}{n}}},
\label{eqn:Bound_on_NhatSk_over_Sk_vs_NhatVk_over_Vk}
\end{align}
where the last equality is obtained the same way as in the proof of Lemma \ref{lemE3} (i.e., from \eqref{eq:maxAB}). Thus, $| E_4 |  = \bigOPbig{ \jaron{ \sqrt{ T/n } }\itr{e_n}{t+1} +\frac{\itr{e_n}{t} }{n}\frac{T}{n}\itr{e_n}{t+1} }$. 
\end{proof}


\chapter{Supporting propositions}

\section{Properties of uniform vertex selection}
\label{sec:Properties_of_uniform_vertex_selection}

\begin{lemma}
\label{lem:Relation_between_unif_selection_from_two_clusters_vs_all_clusters}
If a state $\Vstar$ is selected uniformly at random from two specific clusters $a, b \in \{ 1, \ldots, K \}$, $a \neq b$, and a state $V$ is selected uniformly at random from all states,  
\begin{equation}
\probabilityWrt{ \Vstar \in \mathcal{E} }{\Phi}
= \probabilityWrt{ V \in \mathcal{E} | V \in \mathcal{V}_a \cup \mathcal{V}_b }{\Phi}.
\end{equation}
\end{lemma}

\begin{proof}
We have:
\begin{align}
\probabilityWrt{ \Vstar \in \mathcal{E} }{\Phi}
&
= \sum_{ v \in \mathcal{V}_a \cup \mathcal{V}_b } \probabilityWrt{ \Vstar \in \mathcal{E} | \Vstar = v }{\Phi} \probabilityWrt{ \Vstar = v }{\Phi}
\nonumber \\ &
= \frac{1}{ \cardinality{ \mathcal{V}_a } + \cardinality{ \mathcal{V}_b } } \sum_{ v \in \mathcal{V}_a \cup \mathcal{V}_b } \probabilityWrt{ v \in \mathcal{E} }{\Phi},
\end{align}
and
\begin{align}
&
\probabilityWrt{ V \in \mathcal{E} | V \in \mathcal{V}_a \cup \mathcal{V}_b }{\Phi}
= \frac{ \sum_{ v \in \mathcal{V} } \probabilityWrt{ V \in \mathcal{E}, V \in \mathcal{V}_a \cup \mathcal{V}_b | V = v }{\Phi} \probabilityWrt{ V = v }{\Phi} }{ \probabilityWrt{ V \in \mathcal{V}_a \cup \mathcal{V}_b }{\Phi} }
\nonumber \\ &
= \frac{ \sum_{ v \in \mathcal{V}_a \cup \mathcal{V}_b } \probabilityWrt{ v \in \mathcal{E} }{\Phi} / \cardinality{ \mathcal{V} } }{ ( \cardinality{ \mathcal{V}_a } + \cardinality{ \mathcal{V}_b } ) / \cardinality{ \mathcal{V} } }
= \frac{1}{ \cardinality{ \mathcal{V}_a } + \cardinality{ \mathcal{V}_b } } \sum_{ v \in \mathcal{V}_a \cup \mathcal{V}_b } \probabilityWrt{ v \in \mathcal{E} }{\Phi}.
\end{align}
The lemma follows.
\end{proof}

\begin{lemma}
\label{lem:Relation_between_expected_nr_of_misclassified_vertices_and_probability_of_a_uniformly_selected_vertex_being_misclassified}
If a state $V$ is selected uniformly at random from all states, then
$
\expectationWrt{ \cardinality{\mathcal{E}} }{\Phi} = n \probabilityWrt{ V \in \mathcal{E} }{\Phi}
$.
\end{lemma}

\begin{proof}
We have:
\begin{equation}
\expectationWrt{ \cardinality{\mathcal{E}} }{\Phi}
= \expectationWrt{ \sum_{ v \in \mathcal{V} } \indicator{ v \in \mathcal{E} } }{\Phi}
= \sum_{ v \in \mathcal{V} } \expectationWrt{ \indicator{ v \in \mathcal{E} } }{\Phi}
= \sum_{ v \in \mathcal{V} } \probabilityWrt{ v \in \mathcal{E} }{\Phi}, 
\end{equation}
and
\begin{align}
n \probabilityWrt{ V \in \mathcal{E} }{\Phi}
&
= n \sum_{ v \in \mathcal{V} } \probabilityWrt{ V \in \mathcal{E} | V = v }{\Phi} \probabilityWrt{ V = v }{\Phi}
\nonumber \\ &
= n \sum_{ v \in \mathcal{V} } \probabilityWrt{ v \in \mathcal{E} }{\Phi} \frac{1}{ \cardinality{ \mathcal{V} } }
= \sum_{ v \in \mathcal{V} } \probabilityWrt{ v \in \mathcal{E} }{\Phi},
\end{align}
which completes the proof.
\end{proof}

\section{Asymptotic comparisons between $P$ and $Q$'s entries}
\label{sec:Asymptotic_comparisons_between_P_and_Qs_entries}

Recall that $R_{x,y} = Q_{x,y} / P_{x,y}$ for $x,y \in \mathcal{V}$.

\begin{lemma}
\label{lem:Leading_order_behavior_of_Qxy_Pxy}
The following properties hold:
\begin{enumerate}
\item $R_{x,y} = 1 + n^{-1} ( \indicator{ \sigma(y) = \sigma(\Vstar) } / \alpha_{\sigma(y)} - q_{\sigma(x),\star} / ( p_{\sigma(x),\sigma(y)} K ) ) + \bigO{ n^{-2} }$ for $x, y \neq \Vstar$,
\item $R_{x,\Vstar} = q_{\omega(x),\star} \alpha_{\sigma(\Vstar)} / p_{\omega(x),\sigma(\Vstar)} + \bigO{ n^{-1} }$ for $x \in \mathcal{V} \backslash \{ \Vstar \}$,
\item $R_{\Vstar,y} = q_{\star,\omega(x)} / p_{\sigma(\Vstar),\omega(x)} + \bigO{ n^{-1} }$ for $y \in \mathcal{V} \backslash \{ \Vstar \}$.
\end{enumerate}
\end{lemma}

\begin{proof}
Let $x, y \in \mathcal{V} \backslash \{ \Vstar \}$. Using a Taylor expansion (i), we find that:
\begin{align}
R_{x,y}
&
\eqcom{\ref{eqn:Definition_of_P},\ref{eqn:Definition_of_Qs_entries}}= \frac{ p_{\sigma(x),\sigma(y)} - \frac{q_{\sigma(x),\star}}{ K n } }{ p_{\sigma(x),\sigma(y)} } \cdot \frac{ \cardinality{ \mathcal{V}_{\sigma(y)} } - \indicator{ \sigma(x) = \sigma(y) } }{ \cardinality{ \mathcal{V}_{\sigma(y)} } - \indicator{ \sigma(y) = \sigma(\Vstar) } - \indicator{ \sigma(x) = \sigma(y) } }
\nonumber \\ &
\eqcom{i}= 1 + \frac{1}{n} \Bigl( \frac{ \indicator{ \sigma(y) = \sigma(\Vstar) } }{ \alpha_{\sigma(y)} } - \frac{ q_{\sigma(x),\star} }{ p_{\sigma(x),\sigma(y)} K } \Bigr) + \bigObig{ \frac{1}{n^2} }.
\end{align}
Similarly for $x \in \mathcal{V} \backslash \{ \Vstar \}$
\begin{equation}
R_{x,\Vstar}
\eqcom{\ref{eqn:Definition_of_P},\ref{eqn:Definition_of_Qs_entries}}= \frac{ q_{\omega(x),\star} }{ p_{\sigma(x),\sigma(\Vstar)} } \cdot \frac{ \cardinality{ \mathcal{V}_{\sigma(\Vstar)} } - \indicator{ \sigma(x) = \sigma(\Vstar) } }{ n } 
\eqcom{i}= \frac{ q_{\omega(x),\star} \alpha_{\sigma(\Vstar)} }{ p_{\sigma(x),\sigma(\Vstar)} } + \bigObig{ \frac{1}{n} },
\end{equation}
and for $y \in \mathcal{V} \backslash \{ \Vstar \}$
\begin{equation}
R_{\Vstar,y}
\eqcom{\ref{eqn:Definition_of_P},\ref{eqn:Definition_of_Qs_entries}}= \frac{ q_{\star,\omega(y)} }{ p_{\sigma(\Vstar),\sigma(y)} } \cdot \frac{ \cardinality{ \mathcal{V}_{\sigma(y)} } - \indicator{ \sigma(\Vstar) = \sigma(y) } }{ \cardinality{ \mathcal{W}_{\omega(y)} } } 
\eqcom{i}= \frac{ q_{\star,\omega(y)} }{ p_{\sigma(\Vstar),\sigma(y)} } + \bigObig{ \frac{1}{n} }.
\end{equation}
This completes the proof.
\end{proof}

Recall that $S_{x,y,u,v} = \ln{ R_{x,y} } \cdot \ln{ R_{u,v} }$ for $x,y,u,v \in \mathcal{V}$.

\begin{corollary}
\label{cor:Leading_order_behavior_of_ln_Qxy_Pxy_products}
The following properties hold:
\begin{enumerate}
\item $S_{x,y,u,v} = \bigO{n^{-2}}$ if all $x,y,u,v \neq \Vstar$,
\item $S_{x,y,u,v} = \bigO{n^{-1}}$ if one of $x,y,u,v$ is $\Vstar$,
\item $S_{x,y,u,v} = \bigO{1}$ if two of $x \neq y, u \neq v$ are $\Vstar$.
\end{enumerate}
\end{corollary}

\begin{proof}
These properties are all direct consequences of  \refLemma{lem:Leading_order_behavior_of_Qxy_Pxy}, which can be seen by using the Taylor expansion $\ln{(1+x)} = x + \bigO{x^2}$ for $x \approx 0$ and expanding the product. Consider for example the case $x,y,u,v \in \mathcal{V} \backslash \{ \Vstar \}$:
\begin{equation}
S_{x,y,u,v} 
= \ln{R_{x,y}} \cdot \ln{R_{u,v}} 
= \ln{ \Bigl( 1 + \bigObig{ \frac{1}{n} } \Bigr) } \cdot \ln{ \Bigl( 1 + \bigObig{ \frac{1}{n} } \Bigr) }
= \bigObig{ \frac{1}{n^2} }. 
\end{equation}
The remaining cases follow similarly.
\end{proof}

\revisedPartBegin
\section{The KL-divergence and the log-square expression}
\label{sec:The_KL_divergence_and_the_log_square_expression}

\begin{lemma}
\label{lem:KL-var}
When $\sum_{x \in \mathcal{X}} p_x = \sum_{x \in \mathcal{X}} q_x =1$ and $p_x \ge 0$ and $q_x \ge 0$ for all $x \in \mathcal{X}$,
$
\sum_{x \in \mathcal{X}} p_x \bigl( \ln \frac{p_x}{q_x} \bigr)^2 
\leq 2 \bigl( \max_x \frac{p_x \vee q_x}{p_x \wedge q_x} \bigr)^2 \mathrm{KL}(p\| q)
$.
\end{lemma}

\begin{proof}
Let $\mathcal{X}_+ \triangleq \{ x \in \mathcal{X} : p_x > q_x \}$  and $\mathcal{X}_- \triangleq \{ x \in \mathcal{X} : p_x < q_x \}$. For any given $x \in \mathcal{X}_+$ and $x' \in \mathcal{X}_-$, for all $p_x> a> q_x$ and $p_{x'}<b<q_{x'}$
\begin{align}
\frac{d}{d\varepsilon} \Bigr((a + \varepsilon) \ln \frac{a+ \varepsilon}{q_x} + (b - \varepsilon) \ln\frac{b - \varepsilon}{q_{x'}} \Bigr) \Big|_{\varepsilon =0} 
&
= \ln \frac{a }{q_x}  - \ln\frac{b }{q_{x'}}
\nonumber \\
\geq \frac{a-q_x}{a } + \frac{q_{x'} -b}{q_{x'}}
&
\geq \frac{a-q_x}{p_x \vee q_x } + \frac{q_{x'} -b}{p_{x'} \vee q_{x'}}.
\end{align}
Therefore,
\begin{align}
\mathrm{KL}(p \| q)=& \sum_{x \in \mathcal{X}} p_x \ln \frac{p_x}{q_x} \ge  \frac{1}{2}\sum_{x \in \mathcal{X}}  \frac{(p_x - q_x)^2}{p_x \vee q_x}.  \label{eq:KLpqLower}
\end{align}
Then, we have
\begin{align}
\sum_{x \in \mathcal{X}} p_x & \Bigl( \ln \frac{p_x}{q_x} \Bigr)^2
\leq \sum_{x \in \mathcal{X}} p_x \frac{(p_x - q_x)^2}{(p_x \wedge q_x)^2}
\nonumber \\ &
\leq \sum_{x \in \mathcal{X}} \Bigl( \max_x \frac{p_x \vee q_x}{p_x \wedge q_x} \Bigr)^2  \frac{(p_x - q_x)^2}{(p_x \vee q_x)}
\eqcom{i}\leq 2 \Bigl( \max_x \frac{p_x \vee q_x}{p_x \wedge q_x} \Bigr)^2  \mathrm{KL}(p\| q),
\end{align}
where the last inequality marked (i) stems from \eqref{eq:KLpqLower}.
\end{proof}
\revisedPartEnd

\section{Spectral norm bound for sums of elements of matrices}
\label{sec:Spectral_norm_bound_for_sums_of_elements_of_matrices}

\begin{lemma}
\label{lem:Summation_to_innerproduct_matrix_notation}
For any matrix $B \in \realNumbers^{n \times n}$ and subsets $\mathcal{A},\mathcal{C} \subseteq \{ 1, \ldots, n \}$, we have
$
\sum_{r \in \mathcal{A}} \sum_{c \in \mathcal{C}} B_{rc} = \transpose{ \vectOnes{\mathcal{A}} } B \vectOnes{\mathcal{C}}.
$.
Furthermore, $\transpose{ \vectOnes{\mathcal{A}} } B \vectOnes{\mathcal{C}} \leq \pnorm{B}{} \sqrt{ \cardinality{\mathcal{A}} \cardinality{ \mathcal{C} } }$.
\end{lemma}

\begin{proof}
We have:
\begin{align}
\transpose{ \vectOnes{\mathcal{A}} } B \vectOnes{\mathcal{C}}
&
= \transpose{ \vectOnes{\mathcal{A}} } \Bigl( \sum_{r=1}^n \Bigl( \sum_{c=1}^n B_{rc} \indicator{ c \in \mathcal{C} } \vectElementary{n}{r} \Bigr) \Bigr)
\nonumber \\ &
= \sum_{c'=1}^n \indicator{ c' \in \mathcal{A} } \transpose{ \vectElementary{n}{c'} } \Bigl( \sum_{r=1}^n \Bigl( \sum_{c \in \mathcal{C}} B_{rc} \vectElementary{n}{r} \Bigr) \Bigr)
= \sum_{c' \in \mathcal{A} } \sum_{r=1}^n \sum_{c \in \mathcal{C}} B_{rc} \transpose{ \vectElementary{n}{c'} } \vectElementary{n}{r}
\nonumber \\ &
= \sum_{c' \in \mathcal{A} } \sum_{r=1}^n \sum_{c \in \mathcal{C}} B_{rc} \indicator{ c' = r }
= \sum_{r \in \mathcal{A}} \sum_{c \in \mathcal{C}} B_{rc},
\nonumber 
\end{align}
which proves the first statement. 

For the second statement, first note that (i) $\transpose{ \vectOnes{\mathcal{A}} } B \vectOnes{\mathcal{C}} \in \realNumbers$ and therefore $\transpose{ \vectOnes{\mathcal{A}} } B \vectOnes{\mathcal{C}} \leq | \transpose{ \vectOnes{\mathcal{A}} } B \vectOnes{\mathcal{C}} |$. By (ii) applying the Cauchy--Schwarz inequality twice, and (iii) the consistency of subordinate norms, we obtain
\begin{align}
\transpose{ \vectOnes{\mathcal{A}} } B \vectOnes{\mathcal{C}} 
&
\eqcom{i}\leq | \transpose{ \vectOnes{\mathcal{A}} } B \vectOnes{\mathcal{C}} |
\eqcom{ii}\leq \pnorm{ \vectOnes{\mathcal{A}} }{2} \pnorm{ B \vectOnes{\mathcal{C}} }{2}
\eqcom{iii}\leq \pnorm{ \vectOnes{\mathcal{A}} }{2} \pnorm{B}{} \pnorm{ \vectOnes{\mathcal{C}} }{2}.
\end{align}
Lastly for any set $\mathcal{A} \subseteq \{ 1, \ldots, n \}$, we have that $\vectOnes{\mathcal{A}} \in \{ 0, 1 \}^n$, and therefore $\pnorm{ \vectOnes{\mathcal{A}} }{2} = \sqrt{ \pnorm{ \vectOnes{\mathcal{A}} }{1} } = \sqrt{ \cardinality{ \mathcal{A} } }$. Applying this bound for the sets $\mathcal{A}, \mathcal{C}$ concludes the proof.
\end{proof}

\section{Stochastic boundedness properties}
\label{sec:Stochastic_boundedness_properties}

Recall that when we write $X_n = \bigOP{a_n}$ for a sequence of random variables $\{ X_n \}_{n=1}^\infty$ and some deterministic sequence $\{ a_n \}_{n=1}^\infty$, this is equivalent to saying
\begin{equation}
\forall_{\varepsilon > 0} \exists_{ \delta_\varepsilon, N_\varepsilon } : \probabilityBig{ \Bigl| \frac{X_n}{a_n} \Bigr| \geq \delta_\varepsilon } \leq \varepsilon \, \forall_{n > N_\varepsilon}.
\end{equation}

\begin{lemma}
\label{lem:BigOP_product_X_Y_property}
Let $\cup_{n = 1}^\infty \{ X_n \}_{n \geq 0}$, $\cup_{n = 1}^\infty \{ Y_n \}$ denote two families of random variables with the properties that $X_n, Y_n \geq 0$, $X_n = \bigOP{x_n}$, and $Y_n = \bigOP{y_n}$, where $\{ x_n \}_{n=1}^\infty$, $\{ y_n \}_{n=1}^\infty$ denote two deterministic sequences with $x_n, y_n \in \positiveRealNumbers$. Then $X_n Y_n = \bigOP{x_n y_n}$. Similarly if $X_n = \Omega_{\mathbb{P}}(x_n)$, $Y_n = \Omega_{\mathbb{P}}(y_n)$, then $X_n Y_n = \Omega_{\mathbb{P}}(x_n y_n)$.
\end{lemma}

\begin{proof}
Let $\varepsilon > 0$. Choose $\delta_\varepsilon^X, N_\varepsilon^X$ and $\delta_\varepsilon^Y, N_\varepsilon^Y$ such that $\probability{ X_n \geq \delta_\varepsilon^X x_n } \leq \varepsilon / 3$ \jaron{for $n > N_\varepsilon^X$} and $\probability{ Y_n \geq \delta_\varepsilon^Y y_n } \leq \varepsilon / 3$ \jaron{for $n > N_\varepsilon^Y$}. Pick any $\delta_\varepsilon > \delta_\varepsilon^X \delta_\varepsilon^Y$. With these choices, 
\begin{align}
\probabilityBig{ \Bigl| \frac{ X_n Y_n }{ x_n y_n } \Bigr| \geq \delta_\varepsilon }
= & \probabilityBig{ \Bigl| \frac{ X_n Y_n }{ x_n y_n } \Bigr| \geq \delta_\varepsilon, X_n \geq \delta_\varepsilon^X x_n, Y_n \geq \delta_\varepsilon^Y y_n } 
\nonumber \\ &
+ \probabilityBig{ \Bigl| \frac{ X_n Y_n }{ x_n y_n } \Bigr| \geq \delta_\varepsilon, X_n \geq \delta_\varepsilon^X x_n, Y_n < \delta_\varepsilon^Y y_n }
\nonumber \\ &
+ \probabilityBig{ \Bigl| \frac{ X_n Y_n }{ x_n y_n } \Bigr| \geq \delta_\varepsilon, X_n < \delta_\varepsilon^X x_n, Y_n \geq \delta_\varepsilon^Y y_n }
\nonumber \\ &
+ \probabilityBig{ \Bigl| \frac{ X_n Y_n }{ x_n y_n } \Bigr| \geq \delta_\varepsilon, X_n < \delta_\varepsilon^X x_n, Y_n < \delta_\varepsilon^Y y_n }
\leq \varepsilon.
\end{align}
We have shown that
\begin{equation}
\forall_{\varepsilon > 0} \exists_{ \delta_\varepsilon = \delta_\varepsilon^X \delta_\varepsilon^Y, N_{\varepsilon} = \max \{ N_\varepsilon^X, N_\varepsilon^Y \} } : \probabilityBig{ \Bigl| \frac{ X_n Y_n }{ x_n y_n } \Bigr| \geq \delta_\varepsilon } \leq \varepsilon \, \forall_{ n > N_{\varepsilon} }.
\end{equation}
This completes the proof.
\end{proof}

The following lemma can be proved similarly:

\begin{lemma}
\label{lem:BigOP_ratio}
Let $\cup_{n = 1}^\infty \{ X_n \}_{n \geq 0}$, $\cup_{n = 1}^\infty \{ Y_n \}$ denote two families of random variables with the properties that $X_n, Y_n \geq 0$, $X_n/Y_n =  \Omega_{\mathbb{P}}({y_n})$, and $X_n = \bigOP{x_n}$, where $\{ x_n \}_{n=1}^\infty$, $\{ y_n \}_{n=1}^\infty$ denote two deterministic sequences with $x_n, y_n \in \positiveRealNumbers$. Then $Y_n = \bigOP{x_n/ y_n}$. 
\end{lemma}

\begin{lemma}
\label{lem:BigOP_asymptotic_stochastic_boundedness_of_sums_of_bounded_RV}
Let $\{ s_n \}_{n=1}^\infty$ denote a deterministic sequence with $s_n \in \naturalNumbersPlus$. Let $\cup_{n=1}^\infty \cup_{m=1}^{s_n} \{ X_{m,n} \}$ denote a family of random variables with the properties that $X_{m,n} \geq 0$, and $\exists_{\delta,N} : \expectation{X_{m,n}} \leq \delta x_n \, \forall_{m=1,\ldots,s_n} \forall_{n > N}$. Then $S_n = \sum_{m=1}^{s_n} X_{m,n} = \bigOP{ s_n x_n }$.
\end{lemma}

\begin{proof}
Let $\varepsilon > 0, \delta_\varepsilon^\Sigma > 0$. Since (i) $X_{m,n} > 0$ for all $m,n$, by (ii) Markov's inequality 
\begin{equation}
\probabilityBig{ \Bigl| \frac{ S_n }{ s_n x_n } \Bigr| \geq \delta_\varepsilon^{\Sigma} }
\eqcom{i}= \probabilityBig{ \frac{1}{ s_n x_n } \sum_{m=1}^{s_n} X_{m,n} \geq \delta_\varepsilon^{\Sigma} }
\eqcom{ii}\leq \frac{ \sum_{m=1}^{s_n} \expectation{X_{m,n}} }{ \delta_{\varepsilon}^\Sigma s_n x_n }.
\label{eqn:Intermediate_bound_on_Sn}
\end{equation}
By assumption $\exists_{\delta,N} : \expectation{X_{m,n}} \leq \delta x_n \, \forall_{m=1,\ldots,s_n} \forall_{n > N}$. Choose $\delta, N$ as such. Specify $\delta_\varepsilon^\Sigma = \delta / \varepsilon$. By \refEquation{eqn:Intermediate_bound_on_Sn}, we have thus shown that
\begin{equation}
\forall_{\varepsilon > 0} \exists_{ \delta_\varepsilon^\Sigma = \delta / \varepsilon, N_{\varepsilon} = N} : 
\probabilityBig{ \Bigl| \frac{ S_n }{ s_n x_n } \Bigr| \geq \delta_\varepsilon^{\Sigma} }
\leq \varepsilon 
\, \forall_{ n > N_{\varepsilon} }. 
\end{equation}
Equivalently, $S_n = \bigOP{ s_n x_n }$. This completes the proof.
\end{proof}

\begin{lemma}
\label{lem:BigOP_asymptotic_stochastic_boundedness_of_random_sum_of_bounded_RV}
Let $\cup_{n=1}^\infty \cup_{m=1}^{n} \{ X_{m,n} \}$ denote a family of random variables with the properties that $X_{m,n} \geq 0$, and $\exists_{\delta,N} : \expectation{X_{m,n}} \leq \delta x_n \, \forall_{m=1,\ldots,n} \forall_{n > N}$. If $\{ Y_n \}_{n=1}^\infty$ is a sequence of random variables with the properties that $Y_n \in \{ 1, \ldots, n \}$, and $Y_n = \bigOP{y_n}$ for some deterministic sequence $\{ y_n \}_{n=1}^\infty$ with $y_n \in \naturalNumbersPlus$, then $Z_n = \sum_{m=1}^{ Y_n \wedge n } X_{m,n} = \bigOP{ (y_n \wedge n) x_n }$. 
\end{lemma}

\begin{proof}
Let $\varepsilon > 0, \delta_\varepsilon^Z > 0$. Then
\begin{align}
\probabilityBig{ \Bigl| \frac{ Z_n }{ y_n x_n } \Bigr| \geq \delta_\varepsilon^Z }
&
= \probabilityBig{ \Bigl| \frac{ Z_n }{ y_n x_n } \Bigr| \geq \delta_\varepsilon^Z, \Bigl| \frac{ Y_{n} }{y_n} \Bigr| \geq \delta_\varepsilon^{Y} } 
+ \probabilityBig{ \Bigl| \frac{ Z_n }{ y_n x_n } \Bigr| \geq \delta_\varepsilon^Z, \Bigl| \frac{ Y_{n} }{y_n} \Bigr| < \delta_\varepsilon^{Y} }
\nonumber \\ &
\leq \probabilityBig{ \Bigl| \frac{ Y_{n} }{y_n} \Bigr| \geq \delta_\varepsilon^{Y} } 
+ \probabilityBig{ \Bigl| \frac{1}{ y_n x_n } \sum_{m=1}^{ ( \delta_\varepsilon^{Y} y_n ) \wedge n } X_{m,n} \Bigr| \geq \delta_\varepsilon^Z }.
\end{align}
By assumption $Y_n = \bigOP{ y_n }$, so we can choose $\delta_\varepsilon^Y \in \naturalNumbersPlus, N_{\varepsilon}^Y > 0$ such that $\probability{ | Y_{n} / y_n | \geq \delta_\varepsilon^{Y} } \leq \varepsilon / 2$ for all $n > N_{\varepsilon}^Y$. Write $\delta_\varepsilon^Z = \delta_\varepsilon^Y \delta_\varepsilon^\Sigma$, and we will specify $\delta_\varepsilon^\Sigma$ in a moment. Presently, we are at
\begin{align}
\probabilityBig{ \Bigl| \frac{ Z_n }{ y_n x_n } \Bigr| \geq \delta_\varepsilon^Z }
&
\leq \frac{\varepsilon}{2} + \probabilityBig{ \Bigl| \frac{1}{ ( \delta_\varepsilon^Y y_n ) x_n } \sum_{m=1}^{ ( \delta_\varepsilon^{Y} y_n ) \wedge n } X_{m,n} \Bigr| \geq \delta_\varepsilon^\Sigma }
\nonumber \\ &
\leq \frac{\varepsilon}{2} + \probabilityBig{ \Bigl| \frac{1}{ ( \delta_\varepsilon^Y y_n \wedge n ) x_n } \sum_{m=1}^{ \delta_\varepsilon^{Y} y_n \wedge n } X_{m,n} \Bigr| \geq \delta_\varepsilon^\Sigma }.
\end{align}
The assumptions on the family $\{ X_{m,n} \}_{m,n=1}^{\infty}$ now allow us to apply \refLemma{lem:BigOP_asymptotic_stochastic_boundedness_of_sums_of_bounded_RV}: specifically, there exist $\delta_\varepsilon^\Sigma, N_\varepsilon^\Sigma$ such that the final term is bounded by $\varepsilon / 2$ for all $n > N_\varepsilon^\Sigma$. Summarizing, we have shown that
\begin{equation}
\forall_{\varepsilon > 0} \exists_{ \delta_\varepsilon^Z = \delta_\varepsilon^Y \delta_\varepsilon^\Sigma, N_\varepsilon^Z = \max{ \{ N_\varepsilon^Y, N_\varepsilon^\Sigma, \} } } :
\probabilityBig{ \Bigl| \frac{ Z_n }{ y_n x_n } \Bigr| \geq \delta_\varepsilon^{\Sigma} } \leq \varepsilon
\, \forall_{ n > N_\varepsilon^Z }.
\end{equation}  
Equivalently, $Z_n = \bigOP{ y_n x_n }$.
\end{proof}

\begin{lemma}
\label{lem:BigOP_consistency_argument}
Let $\cup_{n = 1}^\infty \{ X_n \}_{n \geq 0}$, $\cup_{n = 1}^\infty \{ Y_n \}$ denote two families of random variables with the properties that $\probability{ X_n \leq Y_n } = 1$, $X_n = \Omega_{\mathbb{P}}(x_n)$, and $Y_n = \bigOP{y_n}$, where $\{ x_n \}_{n=1}^\infty$, $\{ y_n \}_{n=1}^\infty$ denote two deterministic sequences with $x_n, y_n \in \realNumbers$. Then, $x_n = \bigO{y_n}$.
\end{lemma}

\begin{proof}
We prove the result by contradiction. Recall first that the assumptions imply that for every $\varepsilon^X, \varepsilon^Y > 0$, there exist $\delta_\varepsilon^X, \delta_\varepsilon^Y > 0$ such that 
\begin{equation}
\lim_{n \to \infty} \probability{ X_n \leq \delta_\varepsilon^X x_n } \leq \varepsilon^X,
\quad
\lim_{n \to \infty} \probability{ Y_n \geq \delta_\varepsilon^Y y_n } \leq \varepsilon^Y.
\end{equation}
Also note that by (i) definition of conditional probability, (ii) the De Morgan laws, and (iii) $\probability{ \{ X_n \leq \delta^X x_n \} \cap \{ Y_n \geq \delta^Y y_n \} } \geq 0$, it follows that
\begin{align}
0 
= & \probability{ X_n > Y_n }
\geq \probability{ \{ X_n > Y_n \} \cap \{ X_n > \delta^X x_n \} \cap \{ Y_n < \delta^Y y_n \} }
\nonumber \\ 
\eqcom{i}= & \probability{ X_n > Y_n | \{ X_n > \delta^X x_n \} \cap \{ Y_n < \delta^Y y_n \} } \times \cdots
\nonumber \\ & 
\times \bigl( 1 - \probability{ ( \{ X_n > \delta^X x_n \} \cap \{ Y_n < \delta^Y y_n \} )^{\mathrm{c}} } \bigr)
\nonumber \\ 
\eqcom{ii}= & \probability{ X_n > Y_n | \{ X_n > \delta^X x_n \} \cap \{ Y_n < \delta^Y y_n \} } \times \cdots
\nonumber \\ & 
\times \bigl( 1 - \probability{ \{ X_n \leq \delta^X x_n \} \cup \{ Y_n \geq \delta^Y y_n \} } \bigr)
\nonumber \\ 
\eqcom{iii}\geq & \probability{ X_n > Y_n | \{ X_n > \delta^X x_n \} \cap \{ Y_n < \delta^Y y_n \} } \times \cdots
\nonumber \\ & 
\times \bigl( 1 - \probability{ \{ X_n \leq \delta^X x_n \} } - \probability{ \{ Y_n \geq \delta^Y y_n \} } \bigr).
\end{align}
Now suppose that $x_n = \omega(y_n)$. By then taking the limit $n \to \infty$ both left and right, we obtain the inequality $0 \geq 1 - \varepsilon^X - \varepsilon^Y$, which is a contradiction. Hence it must be that $x_n = \bigO{y_n}$.
\end{proof}

~ 
\end{supplement}

\end{document}